\documentclass[leqno]{article}
\usepackage{amsmath}
\usepackage{mathrsfs} 
\usepackage{amsthm}
\usepackage{amsfonts}
\usepackage{amssymb}
\usepackage{latexsym}
\usepackage{tikz} 
\usepackage{esint} 
\usepackage{tikz}
\usepackage{tikz-cd}
\usepackage{cancel}
\usepackage{sectsty}
\usepackage{hyperref}
\usepackage{caption}
\usepackage{float}
\usepackage{setspace}

\usepackage{titlesec}
\titleformat{\subsection}[runin]{\normalfont\bfseries}{\thesubsection}{1em}{}
\titleformat{\subsubsection}[runin]{\normalfont\bfseries}{\thesubsubsection}{1em}{}

\renewcommand{\thesubsubsection}{(\thesubsection.\arabic{subsubsection})}

\usepackage[OT2,T1]{fontenc}
\DeclareSymbolFont{cyrletters}{OT2}{wncyr}{m}{n}
\DeclareMathSymbol{\Sha}{\mathalpha}{cyrletters}{"58}

\sectionfont{\fontsize{12}{15}\selectfont}
\subsectionfont{\fontsize{11}{15}\selectfont}
\subsubsectionfont{\fontsize{10}{15}\selectfont}

\usepackage[matrix,tips,graph,curve]{xy}
\usepackage{enumitem}

\makeatletter

\usepackage[
backend=biber,
style=numeric,
sorting=anyt, doi=false, url=false, isbn=false, giveninits=true, 
]{biblatex}

\theoremstyle{plain}
\newtheorem{theorem}[subsubsection]{Theorem}
\newtheorem{thm}[]{Theorem}
\newtheorem{corollary}[subsubsection]{Corollary}
\newtheorem{lemma}[subsubsection]{Lemma}
\newtheorem{proposition}[subsubsection]{Proposition}
\newtheorem{AppProp}[subsection]{Proposition}

\newtheorem{AppLem}[subsection]{Lemma}

\newtheorem{defn}[thm]{Definition}

\theoremstyle{definition}
\newtheorem{definition}[subsubsection]{Definition}

\newtheorem{remark}[subsubsection]{Remark}

\newtheorem{set-up}[equation]{Set-up}

\newtheorem{notation}[subsubsection]{Notation}

\newcommand{\IA}{\mathbb{A}}
\newcommand{\IB}{\mathbb{B}}
\newcommand{\IC}{\mathbb{C}}
\newcommand{\ID}{\mathbb{D}}

\newcommand{\IF}{\mathbb{F}}
\newcommand{\IG}{\mathbb{G}}

\newcommand{\IP}{\mathbb{P}}
\newcommand{\IQ}{\mathbb{Q}}
\newcommand{\IR}{\mathbb{R}}

\newcommand{\IZ}{\mathbb{Z}}

\newcommand{\sC}{\mathcal{C}}
\newcommand{\sK}{\mathcal{K}}

\newcommand{\shO}{\mathscr{O}}

\newcommand{\shT}{\mathscr{T}}

\newcommand{\End}{\mathrm{End}}
\newcommand{\tr}{\mathrm{tr}}

\newcommand{\Hom}{\mathrm{Hom}}
\newcommand{\Aut}{\mathrm{Aut}}

\newcommand{\Spin}{\mathrm{Spin}}

\newcommand{\Spec}{\mathrm{Spec}}

\newcommand{\Lie}{\mathrm{Lie}}

\newcommand{\Rep}{\mathrm{Rep}}
\newcommand{\MF}{\mathrm{MF}}
\newcommand{\rank}{\rm rank} 

\newcommand\iso{\,{\cong}\,} 
\newcommand\tensor{{\otimes}}
 
\newcommand\union{\bigcup} 

\newcommand{\<}{\langle}
\renewcommand{\>}{\rangle}

\newcommand{\into}{\hookrightarrow}

\def\d/{/\mspace{-6.0mu}/}

\def\wt{\widetilde}

\def\what{\widehat}

\newcommand{\sI}{\mathcal{I}}

\newcommand{\w}{\omega}

\newcommand{\cl}{\mathrm{cl}}

\newcommand{\Ohm}{\Omega}

\newcommand{\fm}{\mathfrak{m}}
\newcommand{\fp}{\mathfrak{p}}

\newcommand{\Pic}{\mathrm{Pic}\,}

\newcommand{\sE}{\mathcal{E}}
\newcommand{\sL}{\mathcal{L}}
\newcommand{\sM}{\mathcal{M}}
\newcommand{\sO}{\mathcal{O}}
\newcommand{\sP}{\mathcal{P}}
\newcommand{\sZ}{\mathcal{Z}}
\newcommand{\fk}{\mathfrak{k}}

\newcommand{\Cl}{\mathrm{Cl}}

\newcommand{\Gal}{\mathrm{Gal}}

\newcommand{\NS}{\mathrm{NS}}
\newcommand{\Hdg}{\mathrm{Hdg}}

\newcommand{\Hilb}{\mathrm{Hilb}}

\newcommand{\et}{\mathrm{\acute{e}t}}

\newcommand{\Sh}{\mathrm{Sh}}
\newcommand{\shS}{\mathscr{S}}
\newcommand{\shA}{\mathscr{A}}

\newcommand{\CSpin}{\mathrm{CSpin}}
\newcommand{\SO}{\mathrm{SO}}
\newcommand{\ad}{\mathrm{ad}}
\renewcommand{\sp}{\mathrm{sp}}
\newcommand{\GSp}{\mathrm{GSp}}
\newcommand{\GL}{\mathrm{GL}}

\newcommand{\sA}{\mathcal{A}}
\newcommand{\Br}{\mathrm{Br}}

\newcommand{\Mod}{\mathsf{M}}

\newcommand{\sX}{\mathsf{X}}
\newcommand{\cris}{\mathrm{cris}}
\newcommand{\MT}{\mathrm{MT}}

\newcommand{\bpi}{\mathbf{\pi}}
\newcommand{\dR}{\mathrm{dR}}

\newcommand{\Fil}{\mathrm{Fil}}

\newcommand{\shX}{\mathscr{X}}

\newcommand{\shD}{\mathscr{D}}

\newcommand{\Cris}{\mathrm{Cris}}
\newcommand{\shE}{\mathscr{E}}

\newcommand{\Ch}{\mathrm{Ch}}

\newcommand{\AJ}{\mathrm{AJ}}
\newcommand{\sT}{\mathcal{T}}

\renewcommand{\sK}{\mathsf{K}}
\newcommand{\LEnd}{\mathrm{LEnd}}
\newcommand{\num}{\mathrm{num}}

\renewcommand{\Spec}{\mathrm{Spec\,}}
\newcommand{\BK}{\mathrm{BK}}
\newcommand{\fS}{\mathfrak{S}}

\renewcommand{\ss}{\mathrm{ss}}
\newcommand{\sB}{\mathcal{B}}
\newcommand{\disc}{\mathrm{disc}}
\newcommand{\bH}{\mathbf{H}}
\newcommand{\bP}{\mathbf{P}}
\newcommand{\Mon}{\mathrm{Mon}}
\newcommand{\bL}{\mathbf{L}}
\newcommand{\shY}{\mathscr{Y}}

\renewcommand{\O}{\mathrm{O}}

\newcommand{\CF}{\mathsf{CF}}

\newcommand{\HKn}{\mathrm{K3}^{[n]}}

\newcommand{\Def}{\mathrm{Def}}
\newcommand{\shZ}{\mathscr{Z}}

\newcommand{\uX}{\underline{X}}
\newcommand{\PL}{\mathrm{PL}}
\newcommand{\shP}{\mathscr{P}}
\renewcommand{\sX}{\mathcal{X}}

\newcommand{\sS}{\mathcal{S}}
\newcommand{\fM}{\mathfrak{M}}
\newcommand{\sR}{\mathscr{R}}

\newcommand{\bd}{\boldsymbol{\delta}}
\renewcommand{\bpi}{\boldsymbol{\pi}}
\newcommand{\loc}{\mathrm{loc}}
\newcommand{\fl}{\mathrm{fl}}
\newcommand{\sY}{\mathcal{Y}}
\newcommand{\bxi}{\boldsymbol{\xi}}
\newcommand{\bzeta}{\boldsymbol{\zeta}}

\newcommand{\conj}{\mathrm{conj}}
\newcommand{\sQ}{\mathcal{Q}}
\newcommand{\shM}{\mathscr{M}}

\newcommand{\rat}{\iso_{\mathrm{bir}}}
\renewcommand{\H}{\mathrm{H}}
\renewcommand{\P}{\mathrm{P}}
\newcommand{\bU}{\mathbf{U}}
\newcommand{\sto}{\stackrel{\sim}{\to}}

\title{{\large{\textbf{On Irreducible Symplectic Varieties of $\HKn$-type in Positive Characteristic}}}
\author{\normalsize{Ziquan Yang}}
\date{\vspace{-5ex}}}
\begin{document}

\maketitle

\begin{abstract}
    We show that there is a good notion of irreducible sympelectic varieties of $\HKn$-type over an arbitrary field of characteristic zero or $p > n + 1$. Then we construct mixed characteristic moduli spaces for these varieties. Our main result is a generalization of Ogus' crystalline Torelli theorem for supersingular K3 surfaces. For applications, we answer a slight variant of a question asked by F. Charles on moduli spaces of sheaves on K3 surfaces and give a crystalline Torelli theorem for supersingular cubic fourfolds. 
\end{abstract}

\tableofcontents

\section{Introduction}
A complex \textit{irreducible symplectic} manifold, or a \textit{hyperk\"ahler} manifold, is a compact simply connected K\"ahler manifold $M$ whose $\H^0(M, \Ohm_M^2)$ is generated by a nowhere degenerate $2$-form. Similarly, a smooth projective variety $X$ over $\IC$, or rather over any algebraically closed field of characteristic zero, is an irreducible symplectic variety if its \'etale fundamental group is trivial, and its $\H^0(X, \Ohm_X^2)$ satisfies the same condition. One of the earliest known classes of irreducible symplectic manifolds are those of $\HKn$-type, i.e., those obtained as deformations of the Hilbert scheme of $n$ points on a K3 surface. 

The first goal of our paper is to show that there is a good notion of $\HKn$-type varieties over a general field $k$, when $\mathrm{char\,} k = 0$ or $p > n + 1$. In the following, whenever $k$ denotes a perfect field of characteristic $p >0$, we write the associated ring of Witt vectors $W(k)$ as $W$. For any projective variety $Y$, $Y^{[n]}$ denotes the Hilbert scheme of $n$-points on $Y$. For any field $k$, $k = \bar{k}$ means that $k$ is algebraically closed. 
\begin{defn}
\label{def} 
Let $X$ be a smooth projective variety over a field $k$. 
\begin{enumerate}[label=\upshape{(\alph*)}]
    \item If $k = \bar{k}$ and $\mathrm{char\,} k = 0$, $X$ is said to be of $\HKn$-type if for some connected variety $S$ over $k$ there exists a smooth projective family $\sX \to S$ of irreducible symplectic varieties over $k$ and points $s, s' \in S(k)$ such that $\sX_s \iso X$ and $\sX_{s'}$ is birational to $Y^{[n]}$ for some K3 surface $Y$ over $k$. 
    \item If $\bar{k} = k$ and $\mathrm{char\,} k = p$, we say $X$ is a $\HKn$-type variety if for some finite flat extension $V$ of $W$, there exists a smooth projective scheme $X_V$ over $V$ lifting $X$ such that (i) $X_V$ carries a lifting of a primitive polarization on $X$ and (ii) some geometric generic fiber of $X_V$ is of $\HKn$-type and has the same Hodge numbers as $X$. 
    \item If $k$ is not algebraically closed, $X$ is said to be of $\HKn$-type if for some algebraic closure $\bar{k}$ of $k$, the base change $X_{\bar{k}}$ is of $\HKn$-type. 
\end{enumerate}
\end{defn}

If $k = \IC$, $X$ is of $\HKn$-type if and only if its underlying manifold is of $\HKn$-type (see \ref{defcompC}). Part (b) of our definition is motivated by Deligne's result \cite[Cor.\,\,1.7]{Deligne2} that every (polarized) K3 surface in characteristic $p$ lifts to characteristic zero. Therefore, when $n = 1$, our definition is equivalent to the usual definition. We will verify that when $p > n + 1$, part (b) is independent of the choices involved: 
\begin{thm}
\label{transmit}
Let $X$ be a $\HKn$-type variety over an algebraically closed field $k$ with $\mathrm{char\,} k = p > n + 1$.
\begin{enumerate}[label=\upshape{(\alph*)}]
    \item For any line bundle $\zeta \in \Pic(X)$, there exists a finite flat extension $W'$ of $W$ together with a formal deformation $\what{X}_{W'} \to \mathrm{Spf\,} W'$ of $X$ such that $\zeta$ extends to $\what{X}_{W'}$.
    \item If $X_{V'}$ is another deformation of $X$ over a finite flat extension $V'$ of $W$ such that some primitive polarization on $X$ extends to $X_{V'}$, then every geometric generic fiber of $X_{V'}$ is of $\HKn$-type. 
\end{enumerate}
\end{thm}

Under mild restrictions, the following constructions provide examples of $\HKn$-type varieties (see \S5.): (i) The Hilbert scheme of $n$ points on a K3 surface. (ii) The moduli space stable sheaves on a K3 surface with a fixed Mukai vector. (iii) The Fano variety of lines on a cubic fourfold. Deformations of $\HKn$-type varieties in reasonable families are still $\HKn$-type varieties in an arithmetic setting (see \S2.3). 

Next, we show that $\HKn$-type varieties have a good mixed-characteristic moduli theory. We consider the functor $\Mod_{n, d}$ which sends each scheme $S$ over $\IZ_{(p)}$ to the groupoid of families of primitively polarized $\HKn$-type varieties of degree $d$ over $S$. Our result is:
\begin{thm}
\label{moduli}
Assume that $p > n + 1$. $\Mod_{n, d}$ is a Deligne-Mumford stack of finite type over $\IZ_{(p)}$. Moreover, $\Mod_{n, d}$ has a finite \'etale cover which is representable by a regular quasi-projective scheme. 
\end{thm} 

More precisely, when equipped with orientations and appropriate level structures, primitively polarized $\HKn$-type varieties of degree $d$ have a fine moduli space which is representable by a scheme and admits a finite \'etale forgetful morphism to $\Mod_{n, d}$ (see \ref{fineMod} for details). The $n = 1$ case was treated thoroughly in \cite{Rizov1}, \cite{Keerthi} and \cite{Maulik}. When we take $n > 1$, the above theorem in particular addresses a question asked by F. Charles in \cite{Charles2} on the boundedness of moduli spaces of stable sheaves on K3 surfaces (see \ref{ChQues}).

Thanks to the foundational work by Verbitsky \cite{Verbitsky}, $\HKn$-type varieties over $\IC$, as hyperk\"ahler manifolds, satisfy a global Torelli theorem. In positive characteristics, Ogus proved a crystalline Torelli theorem \cite[Thm II]{Ogus2} for supersingular K3 surfaces. We generalize Ogus' theorem to supersingular $\HKn$-type varieties: 

\begin{thm} 
\label{crysTor}
Assume $p > n + 1$ and let $k$ be an algebraically closed field of characteristic $p$. Let $X$ and $X'$ be two supersingular $\HKn$-type varieties over $k$ and let $\psi : \NS(X) \stackrel{\sim}{\to} \NS(X')$ be an isomorphism which respects the top intersection numbers and sends some ample class to another ample class. Then $\psi$ is induced by an isomorphism $f : X' \stackrel{\sim}{\to} X$ if and only if $\psi$ fits into commutative diagrams
    \begin{center}
        \begin{tikzcd}
        \NS(X) \arrow{r}{\psi} \arrow{d}{c_1} & \NS(X') \arrow{d}{c_1} \\
        \H^2_\cris(X/W) \arrow{r}{\psi_\cris} & \H^2_\cris(X'/W)
        \end{tikzcd}
        \begin{tikzcd}
        \NS(X) \arrow{r}{\psi} \arrow{d}{c_1} & \NS(X') \arrow{d}{c_1} \\
        \H^2_\et(X, \what{\IZ}^p) \arrow{r}{\psi_\et} & \H^2_\et(X', \what{\IZ}^p)
        \end{tikzcd}
    \end{center}
    with an isomorphism $\psi_\cris$ of $W$-modules and a primitively polarizable \'etale parallel transport operator $\psi_\et$. If $n - 1 = 0, 1,$ or a prime power, then $\psi$ is induced by an isomorphism $f$ if and only if $\psi$ fits into a diagram as the one on the left with some isomorphism $\psi_\cris$. 
\end{thm}

Clearly this requires some explanation. In general, if $k$ is a perfect field of characteristic $p > 0$, a $\HKn$-type variety $X$ over $k$ is said to be \textit{supersingular} if the Newton polygon of $\H^2_\cris(X/W)$ is a straight line. Just as Ogus' theorem was modeled on the classical Hodge theoretic Torelli theorem for K3 surfaces, the above theorem is modeled on Verbitsky's theorem (\cite[Thm~1.18]{Verbitsky}, see also \cite[Thm~1.3]{MarkmanSurvey}): 

\begin{thm}
\label{Verbitsky}
\emph{(Verbitsky)} Let $X, X'$ be complex irreducible symplectic varieties. An isomorphism of Hodge structures $\psi : \H^2(X, \IZ) \to \H^2(X', \IZ)$ is induced by an isomorphism $f : X' \stackrel{\sim}{\to} X$ if and only if $\psi$ is a polarizable parallel transport operator.
\end{thm}

As the name suggests, $\psi$ is a \textbf{parallel transport operator} if and only if there exists a complex analytic family $\sX \to S$ of irreducible symplectic manifolds which contains $X, X'$ as fibers and realizes $\psi$ via parallel transport along a topological path $\gamma : [0, 1] \to S$. We define \textbf{\'etale parallel transport operators} for the \'etale cohomology of $\HKn$-type varieties over $k$ in the same way except that $S$ is replaced by a connected $W$-scheme and topological paths are replaced by \'etale paths (see \ref{defetPT}). In either case, $\psi$ is said to be (primitively) \textbf{polarizable} if the family $\sX/S$ can be chosen to admit a relative (primitive) polarization. We will show that deformations of line bundles in families of supersingular $\HKn$-type varieties naturally give rise to \'etale parallel transport operators (see \ref{ssptoetPT}). The notion of parallel transport operators play a fundamental role in the study of complex hyperk\"ahler geometry. It is the author's hope to illustrate by example of $\HKn$-type varieties how one can adapt and work with this notion in an arithmetic setting. 

We study cubic fourfolds as applications. Let $k$ be a perfect field with $\mathrm{char\,} k = p > 0$. A cubic hypersurface $Y \subset \IP^5$, or cubic fourfold, is said to be supersingular if the Newton polygon of $\H^i_\cris(X/W)$ is a straight line for every $i$. Below we denote by $\Ch^j(-)_\num$ the Chow group of codimension $j$ cycles modulo numerical equivalence. Let $\bar{\IF}_p$ denote an algebraic closure of $\IF_p$.
\begin{thm}
\label{cubicTorelli}
Assume $p \ge 7$. Let $Y$ and $Y'$ be two supersingular cubic fourfolds over $\bar{\IF}_p$. Let $h$ and $h'$ denote their hyperplane sections. An isomorphism $\psi : \Ch^2(Y')_\num \stackrel{\sim}{\to} \Ch^2(Y)_\num$ is induced by a projective isomorphism $f : Y \stackrel{\sim}{\to} Y'$ if and only if $\psi$ sends $h'^2$ to $h^2$ and fits into a commutative diagram
    \begin{center}
        \begin{tikzcd}
        \Ch^2(Y')_\num \arrow{r}{\psi} \arrow{d}{\cl} & \Ch^2(Y)_\num \arrow{d}{\cl} \\
        \H^4_\cris(Y'/W) \arrow{r}{\psi_\cris} & \H^4_\cris(Y/W)
        \end{tikzcd}
    \end{center}
    with an isomorphism of $W$-modules $\psi_\cris$. 
\end{thm} 
For a quick example of a supersingular cubic fourfold, the Fermat cubic 
$$ X_0^3 + X_1^3 + \cdots + X_5^3 = 0 $$
is supersingular in every characteristic $p \equiv -1 \mod 3$, by a general result of Shioda and Katsura \cite{Fermat}. In fact, in characteristic $p \ge 7$, the Fermat cubic is \textit{superspecial} whenver it is supersingular (see \ref{sspCubic}). We will show that superspecial cubic fourfolds exist for almost every prime \ref{ssext}. 
\begin{thm}
Superspecial cubic fourfolds over $\bar{\IF}_p$ exist for a set of primes $p$ of Dirichlet density $1$. 
\end{thm}

\paragraph{Sketch of Main Ideas} We first explain the proof of Thm~\ref{crysTor}, which is the main interest of the paper. This is done in two steps. Let $X, X'$ be two supersingular $\HKn$-type varieties and $\psi$ and $\psi_\cris$ be as in Thm~\ref{crysTor}. First, the local Torelli theorem for $X, X'$ allows us to lift them to $X_K, X_K'$ over a finite extension $K$ of $K_0 := W[1/p]$, such that the isomorphism $\psi_\dR : \H^2_\dR(X_K/K) \stackrel{\sim}{\to} \H^2_\dR(X'_K/K)$ induced by $\psi_\cris$ via the Berthelot-Ogus isomorphism preserves the Hodge filtrations. Then we show that $\psi_\dR$ is \textit{absolutely Hodge}, which is predicted by the $p$-adic variational Hodge conjecture. To give an unconditional proof, we choose appropriate polarizations and construct for $X, X'$ Kuga-Satake abelian varieties $A, A'$, which are abelian varieties with CSpin structures. By the isogeny theory of mod $p$ points of Shimura varieties of Hodge type, $\psi$ is induced by an isogeny $\wt{\psi} : A \to A'$ which commutes with CSpin structures. We show that $\wt{\psi}$ can be lifted to induce the desired absolute Hodge cycle. Moreover, with the aid of integral $p$-adic Hodge theory, we show that after choosing an embedding $K \into \IC$, we get an \textit{integral} Hodge isometry $\psi_{\dR, \IC} : \H^2(X_K(\IC), \IZ) \stackrel{\sim}{\to} \H^2(X'_K(\IC), \IZ)$. 

Next, we need to show that $\psi_{\dR, \IC}$ is induced by parallel transport along a topological path. If $n - 1$ has only one prime factor, this is automatic. The key ingredient to show this for general $n$ is Markman's explicit description of the monodromy groups of $\HKn$-type manifolds. If $Y$ is a manifold of $\HKn$-type, its monodromy group $\Mon^2(Y)$ is (up to sign) a congruence subgroup of $\O(\H^2(Y, \IZ))$. Moreover, if $p \nmid 2(n-1)$, the $p$-component of this congruence subgroup is the full group. This allows us to check whether $\psi_{\dR, \IC}$ is a parallel transport operator by passing to \'etale cohomology with prime-to-$p$ coefficients, of which we have good control. The trick then is to interpolate topological path with the \'etale path produced by $\psi[1/p]$ using the moduli spaces we construct for Thm~\ref{moduli}, so that we can ``lift'' the \'etale path to a topological path. To finish the proof, apply Verbitsky's theorem to get an isomorphism $f_\IC : X_K'(\IC) \stackrel{\sim}{\to} X_K(\IC)$ and then use \cite[Thm~2]{MM} to specialize this isomorphism to the desired $f : X' \stackrel{\sim}{\to} X$.

The main step for Thm~\ref{transmit} is to show that ``being of $\HKn$-type'' is a property which is invariant under specialization and generization in reasonable families (see \ref{defstrengthen} and \ref{charpspread}). Then Thm~\ref{transmit} will follow from some connectedness properties of local deformation spaces. For K3 surfaces, Thm~\ref{moduli} is the main result of Rizov's paper \cite{Rizov1}, with some additions from \cite{Keerthi}. The constructions in Rizov's paper can easily be generalized provided that we can prove (A) polarized $\HKn$-type varieties of fixed degree form a bounded family and (B) non-trivial automorphisms of such varieties can be killed by adding level structures on second \'etale cohomology. For K3 surfaces, Saint-Donat proved (A) in \cite{SD} using purely algebro-geometric arguments. By contrast, our proof of (A) is a simplified variant of that of Thm~\ref{crysTor} sketched above. The philosophy is to combine Verbitsky and Markman's results to reduce the boundedness of $\HKn$-type varieties to that of abelian varieties. (B) is known for $\HKn$-type varieties over $\IC$. To treat the $\mathrm{char\,} p$ case, we show that automorphisms of non-supersingular $\HKn$-type varieties over $\bar{\IF}_p$ always lift to characteristic zero. The supersingular case will be treated by a different argument, which entails proving the divisorial Tate conjecture in advance.  

\paragraph{Organization of Paper} In \S2, we establish some basic properties of $\HKn$-type varieties, with an emphasis on deformation theory. In addition, we will also give preparational lemmas on quadratic lattices and recap Markman's monodromy results. In \S3, we review and extend results on spinor Shimura varieties and the Kuga-Satake period morphisms. In \S4.1, 4.2, and 4.3, we prove the Theorem 2, 3, and 4 respectively. In \S4.4, we give examples of \'etale parallel transport operators and relate to Ogus' work \cite{Ogus2}. In \S5, we discuss concrete examples of $\HKn$-type varieties and prove Theorems 6 and 7.

\paragraph{Notations and Conventions}
\begin{itemize}
    \item We write $\what{\IZ}$ for the profinite completion of $\IZ$, $\IA_f$ for $\what{\IZ} \tensor \IQ$, and $\IA_f^p$ for the prime-to-$p$ component of $\IA_f$. For any field $\kappa$, $\bar{\kappa}$ denotes an algebraic closure of $\kappa$. When the letter $k$ denotes a perfect field of characteristic $p > 0$, we write $W$ for $W(k)$, $K_0$ for $W[1/p]$, and $\sigma$ for the lift of the Frobenius endomorphism $x \mapsto x^p$ on $k$. The letter $n$ always stands for some natural number $\ge 1$. 
    \item By a \textit{variety}, we mean a separated scheme of finite type over a field. Unless otherwise noted, by a point $s$ on a scheme (or stack) $S$ we always mean a field-valued point on $S$ and we denote the field of definition of $s$ by $k(s)$. 
    \item Let $S$ be a scheme and $X$ be a smooth proper algebraic space over $S$ of finite relative dimension. If $\xi$ is a line bundle on $X$, we use $m \xi$ to denote the $m$-th power of $\xi$. We say $\xi$ is primitive if on every geometric fiber of $X$, the restriction of $\xi$ is not a nontrivial power of another line bundle. If $S$ is the spectrum of a field, $\xi^{\dim X}$ denotes the top self-intersection number of $\xi$. 
    \item For a smooth and proper morphism $f : X \to S$ of schemes and natural numbers $i, j, r$, we write $\H^j(X, \Ohm_{X/S}^i)$ for $\IR^j f_* \Ohm^i_{X/S}$ and $\H^r_\dR(X/S)$ for the $r$th relative de Rham cohomology $\IR^r f_* \Ohm^\bullet_{X/S}$. If $k$ is a perfect field of characteristic $p > 0$ and $S$ is an Artinian $k$-algebra, we denote by $\H^r_\cris(X)$ the sheaf $\IR^j f_{\cris *} \shO_{X/W}$ on $\Cris(S/W)$ when $S$ is understood. Moreover, for an object $T$ of $\Cris(S/W)$, we use $\H^r_\cris(X)_T$ to denote the Zariski sheaf on $T$ given by restriction.
    \item Let $R$ be an integral domain with fraction field $F$. A quadratic lattice over $R$ is a finitely generated free $R$-module equipped with a symmetric bilinear pairing. Unless otherwise noted, the pairing will be assumed to be non-degenerate. Let $L$ be a quadratic lattice over $R$. Denote by $\disc(L)$ the discriminant group $L^\vee / L$, which is equipped with a natural $F/R$-valued quadratic form. When $R = \IZ$ or $\what{\IZ}$, we write $L_p$ for $L \tensor_R \IZ_p$ and $L^p$ for $L \tensor_R \what{\IZ}^p$ for any prime $p$. If $x \in L$ is any element, we write $x^2$ for $\< x, x \>$. Let $\lambda_n$ denote the constant $(2n)!/(2^n n!)$. 
\end{itemize}

\paragraph{Acknowledgements}

It is a pleasure to thank Lin Chen and Yuchen Fu for discussions on topology, Chenglong Yu and Zhiwei Zheng for discussions on quadratic forms and hyperk\"ahler geometry, Sasha Petrov for his help in the proof of \ref{Sasha} and Dori Bejleri for discussions on moduli spaces. I also had helpful conversations with Qirui Li, Lucia Mocz, Ananth Shankar, Dan Bragg and Nikon Kurnosov. Many thanks to my advisor Mark Kisin who offered valuable comments on the writing of the paper. Finally, presenting an early version of these results at UC Berkeley and UGA was very helpful and I would like to thank the organizers. 

\section{General Properties of $\HKn$-type Varieties}

\subsection{Beauville-Bogomolov Forms} We will repeatedly use the following simple observation:
\begin{lemma}
\label{lem: simple observation}
Let $R$ be a ring in which $(2n)!$ is not a zero divisor and $V$ be a free $R$-module of finite rank. Suppose $w: V^{\tensor 2n} \to R$ and $q : V^{\tensor 2} \to R$ are symmetric multilinear maps such that 
\begin{align}
\label{wq}
    w(\alpha_1, \cdots, \alpha_{2n}) = \frac{1}{n! 2^n} \sum_{\sigma \in S_{2n}} q(\alpha_{\sigma(1)}, \alpha_{\sigma(2)}) \cdots q(\alpha_{\sigma(2n - 1)}, \alpha_{\sigma(2n)}),
\end{align}
where $S_{2n}$ is the symmetry group of $2n$ elements. In particular, $w(\alpha, \cdots, \alpha) = \lambda_n q(\alpha, \alpha)^n$ for all $\alpha \in V$. 

Then for any element $\xi \in V$ such that $q(\xi, \xi)$ is not a zero divisor, $w(\xi, \cdots, \xi, \alpha) = 0$ if and only if $q(\xi, \alpha) = 0$, and $q$ is completely determined by $w$ and the value $q(\xi, \xi)$ by the following condition: 
\begin{align}
\label{wqperp}
    w(\xi, \cdots, \xi, \alpha, \beta) = \lambda_{n- 1} q(\alpha, \beta) q(\xi, \xi)^{n - 1}, \text{ for all } \alpha, \beta \text{ such that } w(\xi, \cdots, \xi, \alpha) = w(\xi, \cdots, \xi, \beta) = 0.
\end{align}

\end{lemma}

\subsubsection{} \label{sec: known BBFs} For a $\HKn$-type variety $X$ over $\IC$, the \textit{Beauville-Bogomolov form} on $\H^2(X, \IZ)$ is the $\IZ$-valued quadratic form $q$ on $\H^2(X, \IZ)$ such that (\ref{wq}) holds for $w : \H^2(X, \IZ)^{\tensor 2n} \to \IZ$ given by cup product and for any polarization $\xi$ on $X$, $q(c_1(\xi), c_1(\xi)) \in \IZ_{>0}$. Note that $q(c_1(\xi), c_1(\xi))$ is the unique root in $\IZ_{> 0}$ to the equation $\xi^{2n} - \lambda_n x^n = 0$. $\H^2(X, \IZ)$ is isomorphic as a quadratic lattice under $q$ to $\Lambda_n$ where (\cite{Known})
\begin{align}
    \label{eqn: K3 n lattices}
    \Lambda_1 = U^{\oplus 3} \oplus E_8^{\oplus 2} \text{ and } \Lambda_n = U^{\oplus 3} \oplus E_8^{\oplus 2} \oplus \IZ(2 - 2n) \text{ for } n > 1.
\end{align}
Here, $U$ denotes the standard hyperbolic plane, $E_8$ is the unique negative definite self-dual even lattice of rank $8$ and $\IZ(2 - 2n)$ is the quadratic lattice over $\IZ$ generated by a single element $v$ with $\<v, v\> = 2 - 2n$. Clearly $\Lambda_1$ is self-dual and $\Lambda_n^\vee / \Lambda_n \iso \IZ/ (2n - 2) \IZ$ for $n > 1$. 

\begin{definition}
\label{def: excellent good reduction}
Let $k$ be a perfect field of characteristic $p > 0$ and $X$ be a variety over $k$. $X$ is said to be an \textit{excellent reduction} of a $\HKn$-type variety if there exists a totally ramified extension $K$ of $K_0$ with $V := \sO_K$ and a smooth projective family $X_V$ over $V$ such that the generic fiber $X_K$ is of $\HKn$-type and $X_V \tensor k \iso X$, $\dim_K \H^r_\dR(X_K/K) = \dim_k \H^r_\dR(X/k)$ for every $r \ge 0$ and the Hodge-de Rham spectral sequence of $X_V$ degenerates at $E_1$-page.
\end{definition}

Note that an excellent reduction $X$ necessarily has the same Hodge numbers as a complex $\HKn$-type variety. In particular, $h^{i,j} = 0$ for $i + j$ odd. By the universal coefficient theorem $\H^r_\cris(X/W)$ is torsion free for all $r \ge 0$. 

\begin{proposition}
\label{NSfree}
Let $k$ be an algebraically closed field of characteristic $p > 0$ and $X/k$ be an excellent reduction of a $\HKn$-type variety. Then we have: 
\begin{enumerate}[label=\upshape{(\alph*)}]
    \item $c_1 : \NS(X) \tensor \IZ_\ell \to \H^2_\et(X, \IZ_\ell(1))$ is injective with torsion-free cokernel for every $\ell \neq p$. 
    \item $c_1 : \NS(X) \tensor \IZ_p \to \H^2_\cris(X/W)^{F = p}$ is injective with torsion-free cokernel. 
    \item $\NS(X) = \Pic(X)$ and $\NS(X)$ is torsion-free. 
\end{enumerate}
\end{proposition}
\begin{proof}
By \cite{MarkmanGen}, the singular cohomology with $\IZ$-coefficints of a complex $\HKn$-type manifold is torsion-free, so in particular by the Artin comparison isomorphism and the smooth proper base change theorem, $\H^1_\et(X, \IZ_\ell) = 0$ and $\H^2_\et(X, \IZ_\ell)$ are torsion-free. By taking the cohomology of the Kummer sequence, one obtains an exact sequence
 $$ 0 \to \Pic(X) \stackrel{\times \ell^m}{\to} \Pic(X) \to \H^2_\et(X, \mu_{\ell^m}) \to \Br(X)[\ell^m]. $$
This readily implies (a). Part (b) follows from \cite[Rmk~3.5]{Deligne2}. 
In (c), the equality $\Pic(X) = \NS(X)$ follows from the fact that Lie algebra of the Picard scheme of $X$ can be identified with $\H^1(X, \sO_X)$, which is zero. Since $\H^2_\et(X, \IZ_\ell)$ and $\H^2_\cris(X/W)$ are torsion-free, $\NS(X)$ is also torsion-free by (a) and (b). 
\end{proof}

\begin{proposition}
\label{extBBF0}
Let $k$ be a perfect field and $X$ be a variety over $k$. Assume $n > 1$. Suppose either $(i)$ $\mathrm{char\,} k = 0$ and $X$ is a $\HKn$-type variety or $(ii)$ $\mathrm{char\,} k = p \nmid 2(n - 1)$ and $X$ is an excellent reduction of a $\HKn$-type variety. 

Then for any $\ell \neq \mathrm{char\,} k$, there exists a unique $\Gal_k$-invariant pairing $q_{X, \ell} : \H^2_\et(X_{\bar{k}}, \IZ_\ell(1))^{\tensor 2} \to \IZ_\ell$, and in case $(ii)$ there exists a unique perfect pairing $q_X : \H^2_\cris(X/W)^{\tensor 2} \to W(-2)$, such that for $q = q_{X, \ell}$ or $q_X$, $(a)$ equation $(\ref{wq})$ holds for $q$ with $w$ being the cup product, and $(b)$ $q(c_1(\zeta), c_1(\zeta))$ lies in $\IZ$ for every $\zeta \in \NS(X)$ and is positive if $\zeta$ is ample.
\end{proposition}
\begin{proof}
In case (i), we may assume that $k$ can be embedded into $\IC$ and obtain the desired $q_{X, \ell}$ via the Artin comparison isomorphism. To see the $\Gal_k$-invariance, note first the cup product $\H^2(X_{\bar{k}}, \IZ_\ell(1))^{\tensor 2n} \to \IZ_\ell$ has this property. Moreover, $q_{X, \ell}(c_1(\xi), c_1(\xi))$ is the unique value in $\IZ_{>0}$ such that $\xi^{2n} = \lambda_n q_{X, \ell}(c_1(\xi), c_1(\xi))^n$. In particular, this value is independent of $\ell$. Then apply \ref{lem: simple observation}.

Now suppose we are in case (ii) and let $(X_V, V, K)$ be as in \ref{def: excellent good reduction} for $X$. Let $\xi$ be a polarization on $X$ which extends to $X_V$. For $\ell \neq p$, $q_{X_K, \ell}$ induces the desired $q_{X, \ell}$ via the smooth and proper base change theorem and we have $q_{X, \ell}(c_1(\xi)) = q_{X_K, \ell}(c_1(\xi))$. To construct $q_X$, we use the fact that $\H^2_\cris(X/W)$ can be recovered from $\H^2_\et(X_{\bar{K}}, \IZ_p)$ by integral $p$-adic Hodge theory. Let $\fS$ be the ring $W[\![u]\!]$. Let $\BK(-)$ denote the Breuil-Kisin module functor\footnote{In \cite[(1.2)]{int}, the functor $\BK(-)$ is denoted by $\fM(-)$ and $\fS$ is denoted by $W[\![u]\!]$.} which takes a $\IZ_p$-lattice in a crystalline $G_K := \Gal(\bar{K}/K)$-representation to a $\fS$-module with a Frobenius action. There is a comparison isomorphism $$C_\cris : \H^2_\et(X_{\bar{K}}, \IZ_p) \tensor_{\IZ_p} B_\cris \stackrel{\sim}{\to} \H^2_\cris(X/W) \tensor_{W} B_\cris$$ which restricts to an isomorphism $(\H^2_\et(X_{\bar{K}}, \IZ_p) \tensor B_\cris)^{G_K} \stackrel{\sim}{\to} \H^2_\cris(X/W) \tensor K_0$. $\BK(\H^2_{\et}(X_{\bar{K}}, \IZ_p)) \tensor_\fS W$, which is naturally a $W$-submodule of $(\H^2_\et(X_{\bar{K}}, \IZ_p) \tensor B_\cris)^{G_K}$ (\cite[Thm~1.2.1(1)]{int}), is mapped isomorphically onto $\H^2_\cris(X/W)$ by \cite[Thm~14.6(iii)]{BMS}. By assumption $p \nmid 2(n - 1)$, $\H^2_\et(X_{\bar{K}}, \IZ_p(1))$ is self-dual under $q_{X_{K}, p}$. As a $G_K$-invariant tensor on $\H^2_\et(X_{\bar{K}}, \IZ_p(1))$, $q_{X_{K}, p}$ defines a reductive group over $\IZ_p$, i.e., the orthogonal group $\O(\H^2_\et(X_{\bar{K}}, \IZ_p(1)))$. By \cite[Cor.~1.3.6]{int}, $q_{X_K, p}$ induces via $C_\cris$ a pairing $q_X$ on $\H^2_\cris(X/W)$. To check that the pairing is perfect, we may assume that $k = \bar{k}$, in which case by \cite[Cor.~1.3.6]{int} again $q_X \iso q_{X_K, p} \tensor_{\IZ_p} W$ as quadratic forms. Finally, $q_X$ satisfies (a) because the morphism $C_\cris$ respects cup products (\cite[Prop.~11.6]{IIK}).

Finally, we check (b). It is clear in case (i). For case (ii), we first observe that the restriction of $q$ on $\NS(X)_\IQ$ is completely determined by the value $q(c_1(\xi), c_1(\xi))$ and equation (\ref{wqperp}). In particular, since the top intersection numbers of line bundles are always integers, this restriction is $\IQ$-valued and is independent of the choice of $q = q_{X, \ell}$ or $q_X$. If $\zeta \in \NS(X)$, $q(c_1(\zeta), c_1(\zeta))$ is an element of $\IQ$ but also an element of $\what{\IZ}^p$ and $W$. Therefore, $q(c_1(\zeta), c_1(\zeta)) \in \IZ$. Now suppose $\zeta$ is ample. Apply (\ref{wq}) with $\alpha_1 = \cdots = \alpha_{2n - 1} = c_1(\xi)$ and $\alpha_{2n} = c_1(\zeta)$, we see that 
$$ \lambda_n q(c_1(\xi), c_1(\xi))^{n - 1} q(c_1(\xi), c_1(\zeta)) = \xi^{n - 1} \zeta > 0. $$
Therefore, $q(c_1(\xi), c_1(\zeta)) > 0$. If $n$ is even, by switching the roles of $\zeta$ and $\xi$ in the above equation, we get $q(c_1(\zeta), c_1(\zeta)) > 0$. If $n$ is odd, then we get $q(c_1(\zeta), c_1(\zeta)) > 0$ by replacing $\xi$ with $\zeta$. 
\end{proof}

\begin{definition}
Whenever $X$ is a smooth proper variety to which \ref{extBBF0} is applicable, $q_{X, \ell}$ (resp. $q_X$) is called the $\ell$-adic (resp. crystalline) Beauville-Bogomolov form on $X$. 
\end{definition}

\begin{corollary}
\label{cor: perfect pairing}
Let $k$ be a perfect field of characteristic $p \nmid 2n!$ and let $X/k$ be an excellent reduction of a $\HKn$-type variety. Then the natural cup product pairing $$\H^1(X, T_X) \times \H^1(X, \Ohm_X) \to \H^2(X, \sO_X) \iso k$$ is perfect, and the map $\H^2(X, \sO_X) \to \H^{2n}(X, \sO_X)$ defined by raising each element to the $n$th cup product power is an isomorphism.
\end{corollary}
\begin{proof}
Let $\w$ and $\rho$ be generators of $\H^0(\Ohm_X^2)$ and $\H^2(\sO_X)$ respectively, we have the following diagram 
\begin{equation}
\label{diag: compare pairings}
\begin{tikzcd}
    \H^1(\Ohm_X^1) \times \H^1(T_X) \arrow{d}{\mathrm{id} \times \w} \arrow{r}{\cup} & \H^2(\sO_X) \arrow{r}{\rho^{n - 1}} \arrow{d}{\w} & \H^{2n}(\sO_X) \arrow{d}{\w^n} \\
    \H^1(\Ohm_X) \times \H^1(\Ohm_X) \arrow{r}{\cup} & \H^2(\Ohm_X^2) \arrow{r}{\w^{n-1} \rho^{n-1}} & \H^{2n}(\Ohm_X^{2n})
    \end{tikzcd}    
\end{equation}
whose vertical arrows are all isomorphisms. Note that $q_X \tensor_W k$ induces an orthogonal decomposition 
$$ \H^2_\dR(X/k) = \H^0(\Ohm_X) \oplus \H^1(\Ohm_X^1) \oplus \H^2(\sO_X). $$
Let $w$ denote the multilinear symmetric map $\H^2_\dR(X/k)^{\tensor 2n} \to k$ given by cup product. Note that 
\begin{align*}
    w(\overbrace{\w, \cdots, \w}^{n}, \overbrace{\rho, \cdots, \rho}^{n}) &= n! q(\w, \rho)^{n}, \text{ and} \\
    w(\alpha, \beta, \overbrace{\w, \cdots, \w}^{n - 1}, \overbrace{\rho, \cdots, \rho}^{n - 1}) &= (n - 1)! q(\alpha, \beta) q(\w, \rho)^{n - 1}, \text{ for any }\alpha, \beta \in \H^1(\Ohm_X)
\end{align*}
Now we conclude using that $q$ restricts to perfect pairings $\H^0(\Ohm_X^2) \times \H^2(\sO_X)$ and $H^1(\Ohm_X^1)^{\tensor 2}$. 
\end{proof}

\begin{remark}
\label{rmk: fcomp}
Suppose we are in case (ii) of \ref{extBBF0}. Then using the relation between $q_X$ and the cup product, one deduces that $q_X$ has the following relation with the Frobenius action $F$ on $\H^2_\cris(X/W)$: 
\begin{align}
\label{ffcomp}
q_X(F(x), F(y)) = p^2 \sigma(q_X(x, y)), \forall x, y \in \H^2_\cris(X/W).
\end{align}
Together with the Mazur-Ogus inequality \cite[Thm~8.26]{BO}, this implies that $q_X$ endows the F-crystal $\H^2_\cris(X/W)$ the structure of a K3 crystal (\cite[Def.~3.1]{Ogus}). The reader will soon see that this K3 crystal structure plays a fundamental role in the geometry of excellent reductions of $\HKn$-type varieties. 
\end{remark}

\subsubsection{}\label{2.1.7} We will need some basic theory of gluing lattices. See \cite[\S2]{Curt} for a summary. Let $R$ be $\IZ$, $\what{\IZ}$, $\IZ_p$ or $\what{\IZ}^p$ for some prime $p$ and $F := R \tensor_\IZ \IQ$. Let $L$ be a non-degenerate quadratic lattice over $R$. Then there is a natural $F/R$-valued form on $\disc(L)$. The natural map $M \mapsto M/L$ gives a bijective correspondence 
$$ \{ \text{Lattices $M$ over $R$ with } L \subseteq M \subseteq L^\vee \} \leftrightarrow \{ \text{Isotropic subgroups $\bar{M}$ with } 0 \subseteq \bar{M} \subseteq \disc(L) \}.$$
If $R = \IZ$ or $\what{\IZ}$, then there is a canonical decomposition $\disc(L) = \prod_\ell \disc(L_\ell)$, as $\ell$ runs over all primes. 

By a pointed lattice we mean a pair $(\Lambda, \lambda)$ where $\Lambda$ is a quadratic lattice over $R$ and $\lambda \in \Lambda$ is a point. We say $\lambda$ is primitive if whenever $\lambda = u \lambda'$ for some $\lambda' \in \Lambda$, $u$ is a unit in $R$. By $\O(\Lambda, \lambda)$ we mean the subgroup of $\O(\Lambda)$ which fixes $\lambda$. If we define $L := \lambda^\perp$, then $\O(\Lambda, \lambda)$ can also be viewed as a subgroup of $\O(L)$. One easily checks by the theory of extending isometries of lattices ( \cite[6]{Curt}) that the \textit{discriminant kernel} of $L$, i.e., the kernel of the natural map $\O(L) \to \O(\disc(L))$, is contained in $\O(\Lambda, \lambda)$. 

Next, we give some lemmas on the classification of quadratic lattices. 
\begin{lemma}
\label{lem: apply strong approx}
Let $L$ be a quadratic lattice over $R = \what{\IZ}$ or $\IZ$. The discriminant kernel of $L$ contains a congruence subgroup of $\O(L)$. 
\end{lemma}
\begin{proof}
We can always find an integer $m$ such that $ L \subseteq L^\vee \subseteq m^{-1} L$. If $g \in \O(L)$ satisfies $g \equiv 1 \mod m$, then for any $x \in L$, $g \cdot (m^{-1} x) - (m^{-1} x) \in L$. In particular, for any $y \in L^\vee$ with image $\bar{y} \in L^\vee/L$, we have $g \cdot \bar{y} = \bar{y}$. Therefore, the discriminant kernel contains the congruence subgroup $\ker(\O(L) \to \O(L/ m L))$.
\end{proof}
\begin{lemma}
\label{strongapprox}
Let $L = U \oplus L'$ be a quadratic lattice over $\IZ$, where $U$ standards for the standard hyperbolic plane. If an open subgroup $\sK \subseteq \O(L \tensor \what{\IZ})$ contains $\O(U \tensor \what{\IZ})$, viewed as a subgroup of $\O(L \tensor \what{\IZ})$ in the obvious way, then the double quotient $\O(L \tensor \IQ) \backslash \O(L \tensor \IA_f) / \sK$ is a singleton. 
\end{lemma}
\begin{proof}
Let $\ad : \Spin(L \tensor \IA_f) \to \SO(L \tensor \IA_f)$ be the natural map given by conjugation. Set $\sK' := \sK \cap \SO(L \tensor \what{\IZ})$ and $\wt{\sK} :=  \ad^{-1}(\sK')$. Using the assumption that $\sK$ contains $\O(U \tensor \what{\IZ})$, one can easily adapt the proof of \cite[Lem.~7.7]{Ogus} to show that the following two maps are both surjections: 
$$ \Spin(L \tensor \IQ) \backslash \Spin(L \tensor \IA_f ) / \wt{\sK} \stackrel{\ad}{\to} \SO(L \tensor \IQ) \backslash \SO(L \tensor \IA_f ) / \sK' \stackrel{j}{\to}  \O(L \tensor \IQ) \backslash \O(L \tensor \IA_f ) / \sK.$$ By the strong approximation theorem, the first double quotient is a singleton. Hence the third double quotient is also a singleton.
\end{proof}

\begin{lemma}
\label{pointedHasse}
Let $(\Lambda, \lambda)$ be a pointed lattice over $\IZ$ such that $\Lambda$ contains $U^{\oplus 2}$ as an orthogonal direct summand. Then any other pointed lattice $(\Lambda', \lambda')$ is isomorphic to $(\Lambda, \lambda)$ if and only if $\Lambda$ and $\Lambda'$ have the same signature and $(\Lambda \tensor \what{\IZ}, \lambda \tensor 1) \iso (\Lambda' \tensor \what{\IZ}, \lambda' \tensor 1)$.
\end{lemma}
\begin{proof}
Clearly it suffices to show the ``if'' direction. Let $L = \lambda^\perp \subset \Lambda$. The pointed lattices $(\Lambda', \lambda')$ which satisfy the given hypothesis are classified by the double quotient $\O(L \tensor \IQ) \backslash \O(L \tensor \IA_f) / \O(\Lambda \tensor \what{\IZ}, \lambda \tensor 1)$. Note that $\O(\Lambda \tensor \what{\IZ}, \lambda \tensor 1)$ is a compact open subgroup of $\O(L \tensor \IA_f)$. 

Let $U_1, U_2$ be the first and the second copy of $U$ in $U^{\oplus 2}$. Suppose $\Lambda = U_1 \oplus U_2 \oplus M$ and $\lambda = u + m$ for some $u \in U_1 \oplus U_2$ and $m \in M$. By James' theorem \cite[Prop.~1.3.2]{Dol}, we can move $u$ into $U_1$ by applying an automorphism of $U_1 \oplus U_2$. Therefore, $L$ contains a copy $U_3$ of $U$. Since $U$ is self-dual, $U_3$ is a direct summand of both $L$ and $\Lambda$, and $\O(\Lambda \tensor \what{\IZ}, \lambda \tensor 1)$ contains $\O(U \tensor \what{\IZ})$. Now we conclude by \ref{strongapprox}. 
\end{proof}

\begin{lemma}
\label{ZpPrim}
Let $p$ be any prime. Let $\Lambda$ be a quadratic lattice over $\IZ_p$ which contains $U_p^{\oplus 2}$ as an orthogonal direct summand, where $U_p = U \tensor \IZ_p$. If $p = 2$, assume $\Lambda$ is even. Let $\lambda$ and $\lambda'$ be two primitive vectors in $\Lambda$. Then $(\Lambda, \lambda) \iso (\Lambda, \lambda')$ if and only if $\lambda^2 = (\lambda')^2$. 
\end{lemma}
\begin{proof}
Clearly only the ``if'' part needs a proof. Let $L = \lambda^\perp \subset \Lambda$. Let $q$ be the associated quadratic form on $\Lambda$ such that $q(x) = x^2/2$. Since $\IZ_p$ is a local ring, by \cite[III Thm~4.1]{Baeza}, there exists an automorphism of $\Lambda$ which carries $\lambda$ to $\lambda'$ provided that $\{ q(y) : y \in L^\perp \}$ generates $\IZ_p$. This condition can be easily checked since $L$ contains a copy of $U_p$ as an orthgonal direct summand. 
\end{proof}

\begin{corollary}
\label{defPLLem}
Let $k$ be an algebraically closed field. Let $X$ be a smooth proper variety over $k$ and $\xi$ be a primitive polarization on $X$. Assume either $(i)$ $\mathrm{char\,} k = 0$ and $X$ is a $\HKn$-type variety or $(ii)$ $\mathrm{char}\, k = p \nmid 2(n-1)$, the pair $(X, \xi)$ lifts to a projective scheme $X_V$ whose generic fiber is of $\HKn$-type. Then there exists a unique primitively pointed lattice $(\Lambda, \lambda)$ over $\IZ$ up to isomorphism such that $\Lambda \iso \Lambda_n$ and $(\H^2_\et(X, \what{\IZ}^p), c_1(\xi)) \iso (\Lambda^p, \lambda^p)$. 
\end{corollary}
The above corollary defines a map $\PL$ from the pairs $(X, \xi)$ which satisfy the hypothesis to the set of isomorphism classes of pointed lattices over $\IZ$. Recall that $\Lambda_n$ was defined in (\ref{eqn: K3 n lattices}).
\begin{proof}
To see the existence, when $\mathrm{char\,} k = 0$, embed $k$ to $\IC$ and use Betti cohomology and when $\mathrm{char\,} k = p$, use the existence of a lift. The point is to show that the isomorphism class of $(\Lambda, \lambda)$ is indenpendent of these choices. Note that $\Lambda_p$ is self-dual, so by \ref{ZpPrim}, the isomorphism class of $(\Lambda_p, \lambda_p)$ is completely determined by the value $\lambda_p^2$, which depends only on $\xi$. Now we conclude by \ref{pointedHasse}. 
\end{proof}

\subsection{Local Deformation Theory}

\subsubsection{} \label{2.2.1} We first consider smooth proper varieties $X$ over an algebraically closed field $k$ of characteristic $p > 2$ which satisfy the following hypotheses: 
\begin{align}
\label{Hnum} 
    h^{0, 2} = h^{0, 2} = 1, h^{i, j} = 0 \textit{ for $i + j$ odd}, K_X \iso \sO_X.
\end{align}
These hypotheses are satisfied, for example, by excellent reductions of $\HKn$-type varieties. Let $\Def(X)$ be the formal deformation functor which sends each Artinian $W$-algebra $R$ to the set of isomorphism classes of flat deformations of $X$ over $R$. Let for a set of line bundles $\xi_1, \cdots, \xi_m$ on $X$, define the deformation functor $\Def(X; \xi_1, \cdots, \xi_m)$ for the tuple $(X; \xi_1, \cdots, \xi_m)$ in a completely analogous way. 
The proof of \cite[Cor.~1.2, Prop. 1.5]{Deligne2} shows that $\Def(X)$ is pro-representable by a formal scheme $\ID:= \mathrm{Spf\,} \sR$ for $\sR \iso W[\![x_1, \cdots, x_{h^{1, 1}}]\!]$ and for any line bundle $\xi$, $\Def(X; \xi)$ is pro-representable by a formal subscheme of $\ID_\xi := \mathrm{Spf\,} \sR/(f_\xi)$ where $f_\xi \in \sR$ is some element dependent on $\xi$. We denote the universal family over $\ID$ by $\uX \to \ID$ and its special fiber $\uX \tensor_W k \to \ID \tensor_W k$ by $\uX_0 \to \ID_0$.

\begin{proposition}
\label{Deligne}
\begin{enumerate}[label=\upshape{(\alph*)}]
    \item The Hodge-de Rham spectral sequence degenerates at $E_1$ page for the family $\uX \to \ID$. For every $i, j, r \in \IZ_{\ge 0}$, the coherent sheaves $\H^j(\uX, \Ohm^i_{\uX/\ID})$ and $\H^r_\dR(\uX/\ID)$ are locally free. Moreover, for any $\sR$-algebra $R$, the natural morphisms $\H^r_\dR(\uX/\ID) \tensor R \to \H^r_\dR(\uX_R/R)$ and $\H^j(\uX, \Ohm^i_{\uX/\ID}) \tensor R \to \H^j(\uX_R, \Ohm^i_{\uX_R/R})$ are isomorphisms.
     \item The Hodge filtration on $\H^2_\dR(\uX/\ID)$ has the following relationship with the Frobenius structure:
    \begin{align*}
        \Fil^1 \H^2_\dR(\uX/\ID) &\subseteq \{ x \in \H^2_\dR(\uX/\ID) : F_{\uX}(x) \in p \H^2_\dR(\uX/\ID) \} \\
        \Fil^2 \H^2_\dR(\uX_0/\ID_0) &= \mathrm{Im}(\{ x \in \H^2_\dR(\uX/\ID) : F_{\uX}(x) \in p^2  \H^2_\dR(\uX/\ID)\} \to \H^2_\dR(\uX_0/\ID_0))
    \end{align*}
    \item Suppose $S_0', S_0$ are Artin $W$-algebras and $S' \in \Cris(S_0'/W), S \in \Cris(S_0/W)$ are objects such that there is a PD morphism $(S_0', S') \to (S_0, S)$. For any morphism $S_0 \to \ID$ and every $r \ge 0$, the canonical base change morphism $j_{S', S} : \H^r_\cris(\uX_{S_0})_{S} \tensor_S S' \to \H^r_\cris(\uX_{S_0} \tensor_{S_0} S_0')_{S'}$
    is an isomorphism. In particular, $\H^r_\cris(\uX_0)$ is a \textit{crystal} of vector bundles over $\ID_0$ for every $r$.
\end{enumerate}
\end{proposition}
\begin{proof}
(a) is an easy consequence of the assumption that $h^{i, j} = 0$ for $i + j$ odd (cf. \cite[Prop.~2.2]{Deligne2}). (b) is a variant of the Mazur-Ogus inequality \cite[Thm~8.26]{BOBook} (cf. \cite[Prop.~2.6, Rmk.~2.7]{Deligne2}). For (c), choose a morphism $S \to \ID$ which extends $S_0 \to \ID$. Then $j_{S', S}$ is identified with the natural morphism $\H^r_\dR(\uX_{S}) \tensor_S S' \to \H^r_\dR(\uX_S \tensor S')$, which is an isomorphism by (a). 
\end{proof}


\begin{proposition}
\label{loc}
Assume in addition that $p \nmid n(n-1)(n+1)$, $X$ is projective, $h^{1, 1} \ge 1$, and the conclusion of \ref{cor: perfect pairing} holds for $X$. Then exists a horizontal perfect pairing $\IB : \H^2_\dR(\uX/\ID)^{\tensor 2} \to \sR$, which is unique uo to a factor in $W^\times$, such that $\Fil^1 = [\Fil^2]^\perp$.

Let $\bar{\shO}$ be an Artinian $W$-algebra and $\shO \in \Cris(\bar{\shO}/W)$. Let $X_{\bar{\shO}}$ be a deformation of $X$ over $\bar{\shO}$. Let $\sI(X_{\bar{\shO}}, \shO)$ be the set of isotropic direct summand of $\H^2_\cris(X_{\bar{\shO}})_\shO$ which lifts $\Fil^2 \H^2_\dR(X_{\bar{\shO}}/\bar{\shO})$. 
\begin{enumerate}[label=\upshape{(\alph*)}]
    \item  The map $\Psi$ from deformations of $X_{\bar{\shO}}$ to $\shO$ to $\sI(X_{\bar{\shO}}, \shO)$ defined by sending $X_\shO$ to $\Fil^2 \H^2_\dR(X_{\shO}/\shO) \subset \H^2_\dR(X_\shO/\shO) \iso \H^2_\cris(X_{\bar{\shO}})_\shO$ is an isomorphism.
   \item Let $\xi_{\bar{\shO}}$ be a line bundle on $X_{\bar{\shO}}$. Then $\xi_{\bar{\shO}}$ extends to $X_\shO$ if and only if $\Psi(X_\shO)$ is orthogonal to $c_{1, \cris}(\xi_{\bar{\shO}})_\shO$. 
\end{enumerate}
\end{proposition}
\begin{proof}
This is due to Langer-Zink \cite{Langer-Zink}. Any $X$ which satisfies the given condition is in particular a variety of K3 type in their terminology (see Def.~22 and the proof of \ref{rmk: fcomp}). In \textit{loc. cit.}, the form $\IB$ is defined in Def.~23, its properties are given in Lem.~25 (see also (50) in its proof), Prop.~26, 29, and (a) is given as Cor.~32. Given (a), (b) follows from \cite[Prop.~1.12]{Ogus}. 
\end{proof}

We remark that (b) implies that if $\xi$ be any line bundle on $X$ and $m$ be a positive number prime to $p$, then $\Def(X; \xi) \subseteq \Def(X; m \xi)$ is an equality, as $m$ is a unit in any $W$-algebra.

\begin{proposition}
\label{flatloc}
Let $X$ be as in the preceeding proposition.
\begin{enumerate}[label=\upshape{(\alph*)}]
    \item If the Newton polygon of $\H^2_\cris(X/W)$ is not pure of slope $1$ and $\xi_1, \cdots, \xi_m \in \Pic(X)$ span a direct summand, $\Def(X; \xi_1, \cdots, \xi_m)$ is smooth over $W$. 
    \item For any primitive line bundle $\xi$ on $X$, if $c_1(\xi) \not\in \Fil^2 \H^2_\dR(X/k)$, then $\Def(X; \xi)$ is smooth over $W$, or otherwise $q_X(\xi, \xi)$ has $p$-adic valuation $1$. In the latter case, $\Def(X; \xi)$ is isomorphic to 
    \begin{align}
        \label{+p}
        \mathrm{Spf\,} W[\![x_1, \cdots, x_{b - 2}]\!]/ (x_1^2 + \cdots + x_{h^{1, 1}}^2 + p).
    \end{align}
    \item For any primitive line bundle $\xi$ on $X$, $\Def(X; \xi)$ is quasi-healthy regular.
\end{enumerate}
\end{proposition}

\begin{proof}
(a) is a direct generalization of \cite[Thm~4.1]{LM}. Recall that the tangent space of $\Def(X)$ at the origin is canonically identified with $\H^1(T_X)$. Let the span of the images of $\xi_i$'s under the map $d \log : \NS(X) \to \H^1(\Ohm_X^1)$ be $V$. By Prop.~\ref{GK} in the appendix, $\dim V = m$. By \cite[Cor.~1.14]{Ogus}, the tangent space of $\Def(X; \xi_1, \cdots, \xi_m)$ at the origin is precisely the annhilator of $V$ under the cup product pairing $\H^1(\Ohm_X) \times \H^1(T_X) \to \H^2(\sO_X)$. Since this pairing is perfect by assumption, $\Def(X; \xi_1, \cdots, \xi_m) \tensor k$ and its tangent space have the same dimension. Hence $\Def(X; \xi_1, \cdots, \xi_m)$ is formally smooth.   

For K3 surfaces, (b) is stated in \cite[Thm~3.8]{Keerthi}, where (\ref{+p}) is written in a slight different form. As remarked in \textit{loc. cit.}, the first statement follows from an adaption of the argument in \cite[Prop.~5.21]{CSpin}. We claim that, similarly, with \ref{loc}, the second statement follows from a slight adaption of the proof of Prop.~6.20 in \textit{loc. cit}. Set $L$ to be the quadratic lattice over $W$ given by $c_1(\xi)^\perp \subseteq M_0 := \H^2_\cris(X/W)$. Let $M^\loc$ be the $W$-scheme which parametrizes isotropic lines in $L$ (see \cite[(2.10)]{CSpin}). Recall that $\Def(X; \xi)$ is represented by $\mathrm{Spf\,} \sR_\xi$. Let $(\shX, \bxi)$ denote the universal family over $\ID_\xi$. Denote $\H^2_\dR(\uX/\ID)$ by $\shM$. Consider the $\sR_\xi$-scheme $\sP$ which sends every $\sR_\xi$-algebra $R$ to the set 
\begin{equation}
    \label{eqn: define torsor P}
    \{ \text{Isometry } \beta : M_0 \tensor_W R \sto \shM \tensor_\sR \text{ such that }  \beta(c_1(\xi) \tensor 1) = c_{1, \dR}(\bxi_R) \}.
\end{equation}
Let $G$ be the stabilizer of $c_1(\xi)$ in $\O(M_0)$. It follows from \cite[Lem.~2.8]{CSpin} that $G$ is smooth. Then $\sP$ is a $G$-torsor over $\mathrm{Spec\,} \sR_\xi$. 
We have a natural diagram of morphisms of $W$-schemes
$$ \begin{tikzcd}
& \sP \arrow[swap]{dl}{p_1} \arrow{dr}{p_2} & \\ \mathrm{Spec\,} \sR_\xi & &  M^\loc 
\end{tikzcd} $$
where $p_2$ is given by sending an $R$-point $\beta$ to $\beta^{-1}(\Fil^2 \H^2_\dR(\shX_R/R))$.
Now we make two claims
\begin{enumerate}[label=\upshape{(\roman*)}]
    \item $p_2$ is $G$-equivariant and smooth of dimension $\dim G_\IQ$. 
    \item There exists a section $s_1$ of $p_1$ such that $p_2 \circ s_1$ is \'etale. 
\end{enumerate}

For (i), it suffices to show the smoothness, since the other claims are then clear. For each surjection $\shO \to \bar{\shO}$ of Artin $W$-algebras with square-zero ideal, we need to show that the natural map 
$$ \Psi : \sP(\shO) \to \sP(\bar{\shO}) \times_{M^\loc(\bar{\shO})} M^\loc(\shO) $$
is surjective. Suppose we have an element on the RHS, i.e., a morphism $\bar{t} : R_\xi \to \bar{\shO}$, together with an isomorphism $\bar{\beta} : M_0 \tensor \bar{\shO} \sto \H^2_\dR(\shX_{\bar{t}}/\bar{\shO}) = \shM_{\shO}$ sending $c_1(\xi) \tensor 1$ to $c_{1, \dR}(\bxi_t)$, and an isotropic line $\Fil \subset L \tensor \shO$ which lifts $\beta^{-1}(\Fil^2 \H^2_\dR(\shX_{\bar{t}}/ \bar{\shO}))$. Viewing $\shO$ as a PD thickening of $\bar{\shO}$, we have a direct summand $c_{1, \cris}(\bxi_{\bar{t}})_{\shO}$ in $\H^2_\cris(\shX_{\bar{t}})_\shO$. By \cite[Lem.~2.8]{CSpin}, there exists an isomorphism $\beta : M_0 \tensor \shO \sto \H^2_\cris(\shX_{\bar{t}})_\shO$ sending $c_1(\xi) \tensor 1$ to $c_{1, \cris}(\bxi)_\shO$ extending $\bar{\beta}$. Now $\beta(\Fil)$ is an isotropic line in $\H^2_\cris(\shX_{\bar{t}})_\shO$ which is orthogonal to $c_{1, \cris}(\bxi)_\shO$ and lifts $\Fil^2 \H^2_\dR(\shX_{\bar{t}}/\bar{\shO})$. By \ref{loc}, we obtain a morphism $t : R_\xi \to \shO$ lifting $\bar{t}$, such that $\beta$ becomes an element of $\sP(\shO)$ which is the preimage under $\Psi$ we are seeking. 

For (ii), let $T_1$ be the first order neighborhood of the closed point of $\Spec (\sR_\xi \tensor_W k)$. Note that the Gauss-Manin connection gives us a canonical identification 
$$ \beta_0 : M_0 \tensor_W T_1 = \H^1_\dR(X/k) \tensor_k T_1  \sto \shM_{T_1} = \H^2_\dR(\sX_{T_1}/T_1), $$
which sends $c_1(\xi) \tensor 1$ to $c_{1, \dR}(\bxi_{T_1})$. This gives a section $\bar{s}_1 : T_1 \to \sP_{T_1}$ to projection $(p_1)_{T_1}$. It is easy to check, using \ref{loc}, that the composition $p_2 \circ \bar{s}_1$ is an isomorphism onto the first order neighborhood of the image of the closed point. Now, any extension $s_1$ of $\bar{s}_1$ to $\mathrm{Spec\,} \sR_\xi$ will do the job for (ii).

The special point of $\mathrm{Spec\,} \sR_\xi$ is sent to a singular point in $M^\loc$, and the complete local ring of a singular point on $M^\loc$ is isomorphic to 
$W[\![x_1, \cdots, x_{r-1}]\!]/ (x_1^2 + \cdots + x_{r - 1}^2 + p)$ where $r = \mathrm{rank\,} L_W = h^{1, 1} + 1$ (cf. the proof of Prop.~2.16 in \textit{loc. cit.}).

(c) is a consequence of (a) and (b) by Vasiu-Zink's criteria (see \cite[Thm~3.8]{Keerthi}). 
\end{proof}

Before we proceed, we recall Nygaard-Ogus' definition of a K3 crystal \cite[Def.~5.1]{NO}. Note that when the base is a point, the definition is equivalent to \cite[Def.~3.1]{Ogus}.
\begin{proposition}
\label{K3crystal}
Assume $p > 3, n + 1$ and $X$ is an excellent reduction of a $\HKn$-type variety. Then there exists a unique perfect horizontal pairing $\underline{q}_\dR : \H^2_\dR(\uX/\ID) \to \sO_\ID$ which extends the evaluation on $q_X$ on any $W$-point of $\ID$ and satisfies the equation $\cup_{i = 1}^n \alpha = \lambda_n \underline{q}_\dR(\alpha, \alpha)^n$ for any section $\alpha$. 

Moreover, if $\underline{q}_\cris$ is the form on $\H^2_\cris(\uX_0)$ induced by $\underline{q}_\dR$, and $\Phi : F^*_{\ID_0/W} \H^2_\cris(\uX_0) \to \H^2_\cris(\uX_0)$ denotes the morphism given by the F-crystal structure, then tuple $(\H^2_\cris(\uX_0), \Phi, \underline{q}_\cris, \Fil^2 \H^2_\dR(\uX_0/\ID_0))$ defines a K3 crystal on $\ID_0/W$. 
\end{proposition}
\begin{proof}
For the first statement, it is clear by a parallel transport argument that such pairing is unique if it exists. Let $s$ be any $W$-valued point on $\ID$ and let $\IB$ be any perfect horizontal form on $\H^2_\dR(\uX/\ID)$, which exists by \ref{loc} and \ref{cor: perfect pairing}. Up to replacing $\IB$ by a scalar in $W^\times$, we may assume that $\IB|_s$ agrees with the evaluation of $q_X$ on $s$. Now since $\alpha \mapsto \alpha^n$ and $\lambda_n q^{\tensor n}$ both defines horizontal tensors on $\H^2_\dR(\uX/\ID)$ and they agree on a $W$-point, they have to be equal. 

In \cite[Def.~5.1]{NO}, (5.1.1) is given, (5.1.3) is a consequence of the Mazur-Ogus inequality \ref{Deligne}(b) and the fact that $\Fil^1 = [\Fil^2]^\perp$ on $\H^2_\dR(\uX_0/\ID_0)$, so we are left with (5.1.2). Let $\Phi_\sR$ denote the map $\varphi_\sR^* \H^2_\dR(\uX/\ID) \to \H^2_\dR(\uX/\ID)$ given by evaluating $\Phi$. Note that $\Phi_\sR$ is horizontal, so we have two horizontal pairings on $\varphi_\sR^* \H^2_\dR(\uX/\ID)$, one is given by $\varphi_\sR^*(\underline{q}_\dR) : \varphi_\sR^* \H^2_\dR(\uX/\ID) \to \varphi_\sR^*(\sR) \iso \sR$ and the other is given by $\underline{q}_\dR(\Phi_\sR(-), \Phi_\sR(-))$, which we abusively write as $\underline{q}_\dR \circ \Phi_\sR$. The condition (5.1.2) in \textit{loc. cit.} amounts to showing that $\underline{q}_\dR \circ \Phi_\sR = p^2 \varphi_\sR^*(\underline{q}_\dR)$. Since both pairings are horizontal, it suffices to check this at a $W$-valued point on $\ID$, which follows from \ref{rmk: fcomp}. 
\end{proof}

We now phrase (a slightly extended version of) the analogue of \cite[Thm~5.3]{NO} for excellent reductions: 
\begin{theorem}
\label{locTor}
Assume $p > 3, n + 1$. Let $V$ be a finite flat extension of $W$. Set $R' := V/J$ for some $J \subseteq (p)$ and $R := V/(p)$. Let $\mathsf{ERK3}^{[n]}_{R'}$ be the category whose objects are smooth proper schemes $\sY$ over $R'$ such that $\sY \tensor_{R'} k$ is an excellent reduction of a $\HKn$-type variety. Let $\mathsf{DK3}^{[n]}_{R'}$ be the category of tuples $(Y, \bH, \alpha, \Fil_{R'})$ which consists of an excellent reduction $Y/k$ of a $\HKn$-type variety, a K3 crystal $\bH$ over $R$, an isomorphism $\alpha : \bH|_{k} \sto \H^2_\cris(Y/W)$, and an isotropic direct summand $\Fil_{R'} \subseteq \bH_{R'}$ which lifts the one $\Fil \subseteq \bH_R$ given by the K3 crystal structure. Morphisms in both categories are isomorphisms. 

Then the natural functor $\Psi_{R'} : \mathsf{ERK3}^{[n]}_{R'} \to \mathsf{DK3}^{[n]}_{R'}$ which sends $\sY$ to $(\sY \tensor k, \H^2_\cris(\sY_R), \mathrm{id}, \Fil^2 \H^2_\dR(\sY/R'))$ is an equivalence of categories. 
\end{theorem}
Note that by restricting the form $\underline{q}_\cris$ in \ref{K3crystal}, $\H^2_\cris(\sY_R)$ is naturally a K3 crystal, so $\Psi_{R'}$ make sense. Note also that we implicitly identified the restriction of $\H^2_\cris(\sY)$ to the special point with $\H^2_\cris(Y/W)$, which is ok by \ref{Deligne}(c). 
\begin{proof}
We first show this for $R' = R$. Let $\bar{\shO}, \shO$ be quotients of $R$ such that $\shO$ is a square zero extension of $\bar{\shO}$. Suppose by induction that $\Psi_{\bar{\shO}}$ is an isomorphism. Then we easily show that $\Psi_{\shO}$ is also an isomorphism using \ref{loc}(a) and \cite[Thm~5.2]{NO}. The point is that the objects in both categories have the same deformation theory, and by \ref{K3crystal}, any deformation of $\sY$ to $\shO$ gives a K3 crystal. Now to treat the case when $R'$ is a further extension of $R$, apply \ref{loc}(a) again.
\end{proof}

\subsection{Spreading out Lemmas} We prove that the property of ``being of $\HKn$-type'' is invariant under specialization and generization in reasonable families.

\begin{proposition}
\label{defcompC}
Let $\kappa$ be any field of characteristic zero and $X$ be a smooth projective variety over $\kappa$. If $\kappa = \bar{\kappa}$ and $\kappa$ can be embedded into $\IC$, then the following statements are equivalent: 
\begin{enumerate}[label=\upshape{(\alph*)}]
    \item $X$ is of $\HKn$-type;
    \item For every embedding $\kappa \subseteq \IC$, the complex manifold $X(\IC)$ is of $\HKn$-type.
    \item For some embedding $\kappa \subseteq \IC$, the complex manifold $X(\IC)$ is of $\HKn$-type.
\end{enumerate}
\end{proposition}
\begin{proof}
(a) implies (b) by a well known theorem of Huybrechts \cite[Thm~4.6]{HuyBasic}. That (b) implies (c) is trivial. Therefore, it suffices to show that (c) implies (a). Choose a polarization $\xi$ on $X$. By applying Artin approximation theorem to the deformation functor $\Def(X; \xi)$, we can construct a polarized family $(\sX \to \sS, \bxi)$ which universally deforms $(X, \xi)$. Of course, $\sS$ can be taken to be connected.

By a density result of Mongardi and Pacienza \cite[Cor.~1.2]{Density}, we can find a $\IC$-point $s_\IC'$ on $\sS$ such that the fiber $\sX_{s_\IC'}$ admits a birational equivalence $\sX_{s_\IC'} \rat Y_\IC^{[n]}$ for some K3 surface $Y_\IC$ over $\IC$. Let $\kappa' \subseteq \IC$ be a finitely generated extension of $\kappa$ such that the tuple $(s_\IC', Y_\IC, \sX_{s'_\IC} \rat Y_\IC^{[n]})$ descends to a tuple $(\wt{s}', Y_{\kappa'}, \sX_{\wt{s}'} \iso_{\mathrm{bir}} Y^{[n]}_{\kappa'})$ defined over $\kappa'$. Let $U$ be an irreducible smooth variety over $\kappa$ such that the generic point $\eta$ has residue field $\kappa'$. Up to shrinking $U$, we may assume that $\wt{s}' : \mathrm{Spec\,} \kappa' \to \sS \tensor_{\kappa} \kappa'$ spreads to a section $\wt{s}'_U : U \to \sS_U := \sS \times_\kappa U$. Let $\sX''$ be the pullback family $\sX \times_\sS \sS_U$ over $\sS_U$ and let $\sX' := (\wt{s}_U')^* \sX''$ be the pullback family over $U$. Up to shrinking $U$ again, we assume that $Y_{\kappa'}$ spreads to a family $\sY$ over $U$. Let $\sY^{[n]}$ be the relative Hilbert scheme of $n$ points over $U$. 

Note that the generic fiber $\sX'_{\eta}$ of $\sX'$ is just $\sX_{\wt{s}'}$, so it is birational to $\sY^{[n]}_\eta = Y^{[n]}_{\kappa'}$. We claim that for a general $\kappa$-point $u \in U$, the fibers $\sY^{[n]}_u$ and $\sX'_u$ are birational. The argument is standard. Indeed, let $\mathring{\sX}'_{\eta}$ and $\mathring{\sY}_{\eta}^{[n]}$ be open dense subschemes of $\sX'_{\eta}$ and $\sY^{[n]}_\eta$ respectively such that $\sX'_{\eta} \rat \sY^{[n]}_\eta$ restricts to an isomorphism $\mathring{\sX}'_{\eta} \iso \mathring{\sY}_{\eta}^{[n]}$. By the spreading out properties of open immersions and isomorphisms, up to shrinking $U$, $\mathring{\sX}'_{\eta}$ and $\mathring{\sY}_{\eta}^{[n]}$ extend to open subschemes $\mathring{\sX}'$ and $\mathring{\sY}^{[n]}$ of $\sX'$ and $\sY^{[n]}$ such that $\mathring{\sX}' \iso \mathring{\sY}^{[n]}$. Since $\sX'$ and $\sY^{[n]}$ have irreducible fibers, for a general $\kappa$-point $u \in U$, $\mathring{\sX}'_u$ and $\mathring{\sY}^{[n]}_u$ are open dense in  $\sX'_u$ and $\sY^{[n]}_u$. Then $\mathring{\sX}'_u \iso \mathring{\sY}^{[n]}_u$ and we are done.

Let $u$ be a point as above. Note that $\sX'_u$ can be found as a fiber on $\sX''$ by construction. Similarly, $X$ is a fiber on $\sX$, so it is also a fiber on $\sX''$. Now we can view $X$ as a deformation of $\sX'_u$ in the family $\sX'' \to \sS_U$, as $\sS_U$ is connected. 
\end{proof}

The proposition below makes Def.~\ref{def}(a) much more flexible to work with.
\begin{proposition}
\label{defstrengthen}
Let $S$ be a connected variety over a characteristic zero field $\kappa$ and $\sX \to S$ be some smooth proper scheme over $S$. If some geometric fiber is a $\HKn$-type variety, then so is every other geometric fiber.  
\end{proposition}
\begin{proof}
We first treat the case $\kappa = \bar{\kappa}$. We may assume that $S$ is irreducible. Let $s$ be any $\kappa$-point on $S$ and $\bar{\eta}$ be a geometric generic point on $S$. It suffices to show that $\sX_s$ is of $\HKn$-type if and only if $\sX_{\bar{\eta}}$ is also of $\HKn$-type. Since the system $(\sX, S, s, \bar{\eta})$ is defined by finitely many equations, we reduce to the case when $\kappa$ and $k(\bar{\eta})$ have finite transcendence degrees over $\bar{\IQ}$. Therefore, we can choose embeddings $\kappa \subseteq k(\bar{\eta}) \subseteq \IC$. Since $S \tensor_\kappa \IC$ is still connected and $s, \bar{\eta}$ both give rise to $\IC$-points on $S \tensor \IC$, we are done by \ref{defcompC}. 

Now we do not assume that $\kappa$ is algebraically closed and let $\bar{\kappa}$ be some algebraic closure of $\kappa$ and set $\bar{S} := S \times_\kappa \bar{\kappa}$. Let $\{ S_i \}_{i \in I}$ be the connected components of $\bar{S}$ and set $\sX_i := \sX|_{S_i}$. Suppose $s \to S$ is some geometric point such that $\sX_s$ is a $\HKn$-type variety. Clearly $s$ factors through $S_i$ for some $i \in I$. By the first paragraph, every geometric fiber of $\sX_i$ is of $\HKn$-type. 

Finally, consider the action of $\Gal(\bar{\kappa}/\kappa)$ on $I$. Take the fiber $\sX_u$ for some $u \in S_i(\kappa)$. The base change $\sX_u \times_\sigma \bar{\kappa}$ is also of $\HKn$-type and is a fiber on $\sX_{\sigma(i)}$. Therefore, by the first paragraph, every geometric fiber of $\sX_{\sigma(i)}$ is of $\HKn$-type. Since $S$ is connected, $\Gal(\bar{\kappa}/\kappa)$ acts transitively on $I$, so we are done. 
\end{proof}


\begin{notation}
\label{hilb+} Let $P$ be a polynomial in $\IQ[T]$, $N := P(0) - 1$ and $\Hilb_P$ be the Hilbert scheme over $\mathrm{Spec\,} \IZ_{(p)}$ which parametrizes closed subschemes of $\IP^{N}$ with Hilbert polynomial $P$. Let $\sZ$ be the universal family over $\Hilb_P$. Let $\Hilb^+_P$ be the (possibly empty) maximal locally closed subscheme of $\Hilb_P$ such that for every geometric point $s \in \Hilb^+_P$, the fiber $\sZ_s$ satisfies the hypotheses of \ref{loc} with $h^{1, 1} > 1$, and $\H^i(\sZ_s, \sO_{\sZ_s}(1)) = 0$ for all $i > 0$. For every number $m$, let $\Hilb^{+, m}_P$ denote the (possibly empty) open subscheme of $\Hilb^+_P$ such that for every geometric point $s \to \Hilb^{+, m}_P$, $\sO_{\sZ_s}(1)$ is an $m$th power of a primitive polarization on $\sZ_s$. 
\end{notation}

\begin{lemma}
\label{Hilbconnected}
Let $k$ be any algebraically closed field with $\mathrm{char\,} k = p$. Assume $p \nmid m$. The generic fiber of every connected component of $\Hilb^{+, m}_{P, W}$ or $\Hilb^{+, m}_P$ is also connected. 
\end{lemma}
\begin{proof}
By \cite[Tag~055J]{stacks-project}, it suffices to show that $\Hilb^{+, m}_{P, W}$ and $\Hilb^{+, m}_P$ are flat over $W$ and $\IZ_{(p)}$ respectively and have reduced special fibers. Since $\Hilb^{+, m}_{P, W} = \Hilb^{+, m}_P \tensor_{\IZ_{(p)}} W$, it suffices to show the statement for $\Hilb^{+, m}_{P, W}$. Let $s$ be a $k$-point on $\Hilb^{+, m}_{P, W}$. Set $X := \sZ_s$ and let $\xi$ be a primitive polarization on $X$ with $m \xi = \sO_{\sZ_s}(1)$. Consider the universal family $(\sX, \bxi)$ over $\Def(X, \xi)$ and denote the structure morphism by $f$. Since $\H^i(X, m \xi) = 0$ for all $i > 0$, $f_* (m \bxi)$ is free. It is easily to see that the natural morphism from the formal neighborhood of $s$ on $\Hilb^{+, m}_{P, W}$ to $\Def(X; m \xi)$ is smooth. Now we conclude by \ref{flatloc}(b). 
\end{proof}

By considering Hilbert schemes, we show that ``being of $\HKn$-type'' is stable under specialization and generization in characteristic $p$: 

\begin{lemma}
\label{charpspread}
Let $S$ be a connected quasi-compact scheme over $\IF_p$ and $(f : \sX \to S, \bxi)$ be a primitively polarized smooth proper scheme. Assume that every geometric fiber of $\sX$ has the same Hodge numbers as a $\HKn$-type variety and satisfies the hypotheses of \ref{loc}. 

For a geometric point $s$ on $S$, we say that $s$ satisfies property $(\ast_{\mathrm{weak}})$ $($resp. $(\ast_{\mathrm{strong}}))$ provided that for some $($resp. any$)$ mixed characteristic discrete valuation ring $R$ with residue field $k(s)$, the generic fiber of some $($resp. any$)$ deformation of $\sX_s$ over $R$ to which $\bxi_s$ extends is of $\HKn$-type. 

Then every geometric point $s$ satisfies $(\ast_{\mathrm{strong}})$ provided some $s$ satisfies $(\ast_{\mathrm{weak}})$. In particular, if one geometric fiber is of $\HKn$-type, so is any other geometric fiber.
\end{lemma}
\begin{proof}
We first remark that if $S$ admits an open cover $\{ S_i \}$ such that each $S_i$ is connected and the lemma holds for each $S_i$, then the lemma also holds for $S$.

Since $S$ is quasi-compact, for some $m$ sufficiently large, $m \bxi$ is very ample and for every geometric point $s \to S$, $\H^i(\sX_s, m \bxi_s) = 0$ for $i > 0$. Of course we may assume $p \nmid m$. The semi-continuity theorem ensures that $f_* (m \bxi)$ is locally free. By the preceeding remark, we may assume that $S$ is affine and $f_* (m \xi)$ is free. 

Let $P$ be the Hilbert polynomial of $m \xi_s$, which is independent of $s$. By choosing $\sO_S$-generators of $f_*(m \bxi)$, we obtain a morphism $S \to \Hilb^{+, m}_{P}$ along which $(f : X \to S, m \bxi)$ is obtained by pulling back the universal family $\sZ$ over $\Hilb^{+, m}_{P}$. Since $S$ is connected, this morphism lands in some connected component $T$ of $\Hilb^{+, m}_{P}$. Therefore, we might as well assume that $S = T \tensor \IF_p$. Now let $s$ be a geometric point on $S$ and $R$ be a mixed characteristic discrete valuation ring. Any deformation of $\sX_s$ over $R$ to which $\bxi_s$ extends is given by some $R$-valued point $s_R$ on $T$ which lifts $s$. The conclusion follows from \ref{defstrengthen} and \ref{Hilbconnected}. 
\end{proof}

Note that \ref{defstrengthen} and \ref{charpspread} implies that in Def.~1(c), the choice of geometric fibers is not important. 

\subsection{Parallal Transport and Monodromy Groups} We recap some results on monodromy groups and parallel transport operators that we will later use. For the definitions, we refer the reader to \cite[Def.~1.1]{MarkmanSurvey}. As in \textit{loc. cit.}, for any complex hyperk\"ahler manifold $X$, denote the image of the monodromy group $\Mon(X)$ of $X$ in $\O(\H^2(X, \IZ))$ by $\Mon^2(X)$. Let $\O_+(\H^2(X, \IZ))$ denote the subgroup of $\O(\H^2(X, \IZ))$ which preserves the spin orientation of $X$ (see \S4 in \textit{loc. cit.}). If $X'$ is another hyperk\"ahler manifold and $\psi : \H^2(X, \IZ) \stackrel{\sim}{\to} \H^2(X', \IZ)$ is an isometry, $\psi$ is always orientation-preserving up to a sign, i.e., either $\psi$ or $- \psi$ is orientation-preserving. 

\begin{lemma}
\label{spinorientation}
Let $X, X'$ be hyperk\"ahler manifolds. An isometry $\psi : \H^2(X, \IZ) \to \H^2(X', \IZ)$ is orientation preserving if one of following holds: 
\begin{enumerate}[label=\upshape{(\alph*)}]
    \item $\psi$ preserves the Hodge structures and sends a K\"ahler class to a K\"ahler class. 
    \item $\psi$ is a parallel transport operator. 
\end{enumerate}
\end{lemma}
\begin{proof}
(a) follows from the fact that the orientation of $\H^2(X, \IZ)$ is completely determined by the real and imaginary parts of $\H^{2, 0}$, as well as the position of any K\"ahler class. For (b), see \cite[\S1.1]{MarkmanSurvey}.
\end{proof}
We remark that \cite[Thm~9.8]{MarkmanSurvey} gives a necessary and sufficient condition for $\psi$ to be a parallel transport operator. This description makes crucial use of Hodge structures and it is not clear how to work with it in an arithmetic setting. The group $\Mon^2(X)$ however, admits a purely arithmetic description up to a sign (\cite[Lem.~9.2]{MarkmanSurvey}): 

\begin{theorem}
\emph{(Markman)}
\label{Markman}
Let $X$ be a complex $\HKn$-type manifold. Let $\Lambda := \H^2(X, \IZ)$. $\Mon^2(X)$ is the pre-image of $\{ \pm 1\}$ under the natural morphism $\O_+(\Lambda) \to \O(\disc(\Lambda))$. In particular, if $n - 1 = 0, 1$ or a prime power, then $\Mon^2(X) = \O_+(\Lambda)$. 
\end{theorem}
It will be clear that the above result allows us to control parallel transport in an arithmetic setting.

For future reference we write down a simple observation from Galois theory: 
\begin{lemma}
\label{GalEx}
Let $S$ be a connected scheme and $\sL$ be a finite locally constant sheaf of sets on $S$ whose stalks are isomorphic to a fixed set $A$. Let $F$ be the functor which sends an $S$-scheme $T$ to the set of trivializations $\underline{A}_T \to \sL_T$, where $\underline{A}_T$ stands for the constant local system on $T$ with coefficients in $A$ and $\sL_T$ is the pullback of $\sL$ to $T$. Then $F$ is representable by an \'etale cover $\wt{S}$ of $S$. Moreover, each connected component of $\wt{S}$ is a Galois cover of $S$. 
\end{lemma}
\begin{proof}
Let $s$ be a geometric point on $S$. We know that $\sL$ corresponds to a representation $\rho : \pi^\et_1(S, s) \to \Aut(A)$. Let $Y$ be the Galois cover of $S$ which corresponds to the normal subgroup $\ker \rho \subseteq \pi^\et_1(S, s)$ (\cite[Tag~03SF]{stacks-project}). Then $G := \mathrm{im\,} \rho$ acts simply transitively on the fiber $Y \times_S s$. Let $\mathrm{Isom}(A, \sL_s)$ be the set of isomorphisms from $A$ to $\sL_s$, which is naturally a (right) torsor under $\Aut(A)$ via pre-composition. Consider the scheme $\wt{S} := \coprod_{j} Y_j$ where $j$ runs through orbits of $\mathrm{Isom}(A, \sL_s)/G$ and each $Y_j = Y$. Let $\sL_j$ be the pullback of $\sL$ to $Y_j$. For each $j \in \mathrm{Isom}(A, \sL_s)/G$, we pick a geometric point $y_j$ on $Y_j$ over $s$ and pick a representative $\sigma_j : A \stackrel{\sim}{\to} \sL_s$. Note that the fiber $\sL_{j, y_j}$ is naturally identified with $\sL_s$, so the additional choice of $\sigma_j$ gives rise to a trivialization $\epsilon(y_j, \sigma_j) : \underline{A}_{Y_j} \to \sL_j$. Note that under $\epsilon(y_j, \sigma_j)$, $g \cdot y_j$ also gives rise to an isomorphism $A \stackrel{\sim}{\to} \sL_{j, g \cdot y_j} = \sL_s$, which is nothing but $g^{-1} \cdot \sigma_j$. 

With the choices of $y_j$'s and $\sigma_j$'s above, we define a natural isomorphism from $F$ to $\wt{S}$. Let $T$ be an $S$-scheme with a trivialization $\epsilon : \underline{A}_T \to \sL_T$. Let $t$ be a geometric point on $T$ and assume without loss of generality $t$ maps to $s$. Then $\epsilon$ induces an isomorphism $\epsilon_t : A \stackrel{\sim}{\to} \sL_{T, t} = \sL_s$. Suppose $\epsilon_t = g \cdot \sigma_j$ for some $g \in G$ and $j$. Note that the existence of $\epsilon$ implies that the natural map $\pi^\et_1(T, t) \to \pi_1^\et(S, s)$ factors through $\ker \rho$, so by Galois theory there exists a unique lift $f : T \to Y_j$ of $T \to S$ which sends $t$ to $g^{-1} \cdot y_j$. The reader easily checks that by the map which sends each $(T \to S, \epsilon)$ to $f$ defines an isomorphism $F \stackrel{\sim}{\to} \wt{S}$. 
\end{proof}

\section{Shimura Varieties and Period Morphisms}
\subsection{A review of spinor Shimura varieties} The main reference for this section is \cite{CSpin}. 

\subsubsection{}
\label{reviewCliff} We give a brief recap of Clifford algebras and set up some notation. 
Let $R$ be an integral domain with $2$ invertible such that $R_\IQ := R \tensor_\IZ \IQ $ is a field. Let $L$ be a lattice over $R$ of rank $m$ with a non-degenerate quadratic form $q$. Let $H$ denote $\Cl(L)$, viewed as a $\Cl(L)$-bimodule. $\Cl(L)$ is naturally decomposed into a part with even degree $\Cl^+(L)$ and a part with odd degree $\Cl^-(L)$, so that $H$ has a natural $\IZ/2\IZ$-grading. There is a natural injection $\Cl(L) \into \End(H)$ given by left multiplication. We remark that this injection embeds $\Cl(L)$ as a direct summand into $\End(H)$: If the left multiplication by $t \in \Cl(L)_\IQ$ preserves the $R$-integral structure on $H$, then $t \in \Cl(L)$ as $t \cdot 1 \in H$. Similarly, the composition $L \into \Cl(L) \into \End(H)$ embeds $L$ as a direct summand into $\End(H)$. Define the group $\CSpin(L)$ by 
$$ \CSpin(L) = \{ v \in \Cl^+(L)^\times : v L v^{-1} \subseteq L \}. $$
If we set $\CSpin(L)(R') := \CSpin(L_{R'})$ for every $R$-algebra $R'$, then $\CSpin(L)$ is endowed with the structure of a group scheme over $R$.

Equip $\End(H_\IQ)$ with a symmetric pairing given by $(\alpha, \beta) \mapsto 2^{-m} \tr(\alpha \circ \beta)$, so that $L_\IQ$ embed into $\End(H_\IQ)$ isometrically. Let $\pi$ denote the orthogonal projection $H_\IQ^{\tensor (2, 2)} \to L_\IQ$. The $R$-module $L$ inside $\End(H_\IQ)$ is recovered by taking the dual of the image of $\End(H)$ under $\pi$ (cf. \cite[Lem.~1.4, Rem.~1.5]{CSpin}). If $L$ is self-dual over $R$, then $\pi$ is in fact defined over $R$. In this case, another way to characterize the group scheme $\CSpin(L)$ is (see also \cite[Lem. 1.4(iii)]{CSpin}):
\begin{lemma}
\label{defCSpin}
Suppose $L$ is self-dual over $R$. The group scheme $\CSpin(L)$ over $R$, as a subscheme of $\GL(H)$, is the stabilizer of the following tensors on $H$: right multiplication of $\Cl(L)$ on $H$, the natural $\IZ/2\IZ$-grading on $H$ and the projection operator $\pi \in H^{\tensor (2, 2)}$.
\end{lemma} 
The group $\CSpin(L)$ is naturally equipped with two representations: a spin representation $\sp : \CSpin(L) \to \GL(H)$ given by left multiplication and an adjoint representation $\ad : \CSpin(L) \to \SO(L)$ given by conjugation. The adjoint representation fits into an exact sequence 
\begin{align}
\label{fundexactseq}
    1 \to \IG_m \to \CSpin(L) \to \SO(L) \to 1.
\end{align}

\subsubsection{}\label{3.1.3} For the rest of \S3.1, let $L$ be an even quadaratic lattice of signature $(2, m-2)$ over $\IZ$.\footnote{This is the opposite of the usual sign convention for orthogonal Shimura varieties.} Let $\Ohm$ be the period domain 
$$ \{ \w \in \IP(L \tensor \IC) : \< \w, \bar{\w} \> < 0, \< \w, \w \> = 0 \} $$
which parametrizes Hodge structures of K3 type on $L$. Let $G$ denote the algebraic group $\CSpin(L_\IQ)$ over $\IQ$ and $G^\ad$ denote its adjoint group $\SO(L_\IQ)$. The pair $(G, \Ohm)$ (resp. $(G^\ad, \Ohm)$) gives a Shimura datum of Hodge type (resp. abelian type) with reflex field $\IQ$. We can equip $H$ with a non-degenerate symplectic form $\psi$ such that the spin representation $\mathrm{sp} : G \to \GL(H_\IQ)$ factors through $\GSp := \GSp(H_\IQ, \psi)$. Let $S^\pm$ be the associated Siegel half spaces. There is an inclusion of Shimura data $i : (G, \Ohm) \into (\GSp, S^\pm)$. Let $\sK_0 := G(\IA_f) \cap \Cl(L \tensor \what{\IZ})^\times$. For every compact open subgroup $\sK \subset \sK_0$, we write $\Sh_\sK(L)$ for the Shimura stack $\Sh_{\sK}(G, \Ohm)$. The inclusion $i$ endows $\Sh_{\sK}(L)$ with a family of abelian schemes, which we denote by $a : \sA \to \Sh_{\sK}(L)$. Let $\sK^\ad$ denote the image of $\sK$ in $G^\ad(\IA_f)$. The subgroup $\sK_0^\ad$ of $G^\ad(\IA_f)$ is the \textit{discriminant kernel}, i.e., the largest subgroup of $\SO(L \tensor \what{\IZ})$ which acts trivially on $L^\vee/L$. We denote the Shimura stack $\Sh_{\sK^\ad}(G^\ad, \Ohm)$ by $\Sh^\ad_\sK(L)$.

The representation $\sp : \CSpin(L) \to \GL(H)$ (resp. $\ad : \CSpin(L) \to \SO(L)$) endows $\Sh_\sK(L)_\IC$ with a $\IZ$-local system $\bH_B$ (resp. $\bL_B$). $\bH_B$ can be identified with the Betti cohomology $\IR^1 a_{\IC*} \underline{\IZ}$. Let $\bH_\ell := \IR^1 a_{\et *} \underline{\IZ_\ell}$ and $\bH_\dR$ be the first relative de Rham cohomology of $\sA$. Over $\Sh_\sK(L)_\IC$, we have canonical isomorphisms $\bH_B \tensor \IZ_\ell \iso \bH_\ell$ and $\bH_\dR \iso \bH_B \tensor_\IZ \shO$, where $\shO$ denotes the structure sheaf. The tensors stabilized by $\CSpin(L)$ naturally spread as global sections of these sheaves. More precisely, $\sA$ is equipped with a $\IZ / 2 \IZ$-grading, a left $\Cl(L)$-action, and global sections $\bpi_B$, $\bpi_\ell$ and $\bpi_\dR$ of $(\bH_B \tensor \IQ)^{\tensor (2, 2)}$, $(\bH_\ell \tensor \IQ_\ell)^{\tensor (2, 2)}$ and $\bH_\dR^{\tensor (2, 2)}$ respectively (cf. \cite[Prop.~3.11]{CSpin}). We call the triple of the $\IZ/2\IZ$-grading, left $\Cl(L)$ action and various realizations of $\pi$ the \textbf{CSpin structure} on the universal abelian scheme $\sA$. Note that $\bL_B$ equals to the \textit{dual} of $\bpi_B(\End (\bH_B))$. We denote the \textit{dual} of the images $\bpi_\ell(\End(\bH_\ell))$ and $\pi_\dR(\End(\bH_\dR))$ by $\bL_B, \bL_\ell$ and $\bL_\dR$. 

\begin{remark}
\label{rmk: orientation tensor}
We will also make use of an \textit{orientation tensor} on $\Sh_\sK(L)$. Let $\delta$ be a generator of $\det(L)$. Note that $\delta$ can alternatively be viewed as an element of $\wedge^m H^{\tensor (1, 1)}$. The action of $G$ on $L$ through the adjoint representation not only preserves the pairing on $L$, but also $\det(L)$ as well. Therefore, just like $\pi$, the tensor $\delta$ also spreads to global sections $\bd_B, \bd_\ell$ and $\bd_\dR$ (cf. \cite[(3.1.3)]{Yang}).  
\end{remark}

Now we consider the integral models of $\Sh_\sK(L)$. Let $p > 2$ be a prime. Note that $\sK_0$ decomposes into $\sK_{0, p} \sK^p_0$, where $\sK_{0,p} \subset G(\IQ_p)$ and $\sK^p_0 \subset G(\IA_f^p)$. We will only consider level structures $\sK$ of the form $\sK_{0, p} \sK^p$, where $\sK^p$ is some compact open subgroup of $\sK_0$. We first assume that $L_p := L \tensor \IZ_p$ is self-dual. In this case, $G$ has a natural smooth model over $\IZ_{(p)}$, $\sK_{0, p} = G(\IZ_p)$ is hyperspecial, and a canonical integral model $\Sh_\sK(L)$ is already constructed in \cite{int}. More precisely, consider the limit 
$$ \shS_p(L) := \varprojlim_{\sK^p} \shS_{\sK_{0,p} \sK^p}(L) $$
where $\sK^p$ runs through all compact open subgroups of $G(\IA_f^p)$. The main theorem of \cite{int} constructs a canonical integral model for $\shS_p(L)$ which carries a natural $G(\IA^p_f)$-action. For each compact open $\sK^p$, set $\shS_{\sK_{0,p} \sK^p}(L)$ to be the stacky quotient $\shS_p(L)/ \sK^p$. 

The abelian scheme $\sA$ has a natural extension $\shA$ over $\shS_\sK$. The \'etale and de Rham cohomology of $\shA$ give us extensions of $\bH_{\ell}$ for every $\ell \neq p$ and $\bH_\dR$. We also have extensions of $\bpi_{\ell}$, $\bpi_\dR$, $\bL_\ell$ and $\bL_\dR$. Denote these extensions by the same letters. The special fiber $\shA \tensor \IF_p$ gives us a crystal $\bH_\cris := \IR^1 a_{\cris *} \shO_{\shA}$. The evaluation of $\bH_\cris$ at the pro-nilpotent PD thickening $\varprojlim_n (\shS_{\sK} \tensor \IZ / p^n \IZ)$ can be identified with $\bH_\dR$, so that $\bH_\dR$ has the structure of a filtered $F$-crystal. The global section $\pi_\dR$ is horizontal, so we obtain a global section $\pi_\cris$ of the $F$-crystal $\bH_\cris$ (cf. \cite[Prop.~4.7, and \S4.14]{CSpin}). Madapusi-Pera has shown that more generally, if $L_p^\vee/L_p$ is \textit{cyclic}, and $p^2 \nmid \disc(L_p)$ (resp. $p^2 \mid \disc(L_p)$), then $\Sh_\sK(L)$ has a \textit{canonical integral model} (resp. \textit{smooth integral model}) $\shS_\sK(L)$ over $\IZ_{(p)}$. For the construction of the corresponding $\shA$, $\bL_\cris$, $\bL_\ell$, etc, we refer the reader to \cite[\S7.10]{CSpin}. 

The integral model $\shS^\ad_\sK(L)$ of $\Sh_\sK^\ad(L)$ is constructed as an \'etale quotient of $\shS_\sK(L)$ (cf. \cite[Thm~7.4]{CSpin}, see also the explanations below \cite[Thm~4.6]{Keerthi}). The sheaves $\bL_B, \bL_\ell, \bL_p, \bL_\cris, \bL_\dR$ on various fibers of $\shS_\sK(L)$ descend to the corresponding fibers of $\shS_\sK^\ad(L)$, and we denote these descents by the same letter ( \cite[\S5.24, and Rmk~7.16]{CSpin}). Note that the variations of $\IZ$-Hodge structures given by $\bL_B$ and the restriction of $\bL_\dR$ on $\Sh^\ad_\sK(L)_\IC$ can alternatively be viewed as given by the canonical representation of $G^\ad$ on $L$ as in \cite[\S3.3]{CSpin}. The global sections $\bd_\ell$ of $\bL_\ell$ over $\Sh_\sK(L)$ extend to global sections over $\shS_\sK(L)$ and descend to $\bL_\ell$'s over $\shS^\ad_\sK(L)$.

\begin{definition}
Let $s$ be a point on $\shS_\sK(L)$. If $s$ is a geometric point and the residue field $k(s)$ has characteristic $p >2$ (resp. $0$), an endomorphism $f \in \End(\shA_s)$ is a called a \textit{special endomorphism} if its cohomological realizations lie in $\bL_{\ell, s}$ for every $\ell \neq p$ and $\bL_{\cris, s}$ (resp. $\bL_{\ell, s}$ for every $\ell$). If $s$ is not a geometric point, then $f \in \End(\shA_s)$ is a special endomorphism if its base change to a geometric point over $s$ is a special endomorphism. 
\end{definition}
We write $\LEnd(\shA_s)$ for the space of special endomorphisms in the above situation. If $f \in \LEnd(\shA_s)$, then $f \circ f$ is a scalar. The map $f \mapsto f \circ f$ endows $\LEnd(\shA_s)$ the structure of a quadratic form over $\IZ$. 

\subsection{Isogeny Classes and the Supersingular Locus}
\label{ssIsog}
We consider the supersingular locus of $\shS_{\sK}(L)$, which we denote by $\shS_{\sK}^\ss(L)$. 

\begin{lemma}
\label{generalk}
Let $k$ be an algebraically closed field of characteristic $p>0$ and $\sA, \sB$ be abelian scheme over $k[\![t]\!]$. Let $\eta$ be the generic point of $k[\![t]\!]$ and $\bar{\eta}$ be a geometric point over $\eta$. If $\sA_{\bar{\eta}}$ and $\sB_{\bar{\eta}}$ are both supersingular, then the cokernel of the specialization map $ \End(\sA_{\bar{\eta}}, \sB_{\bar{\eta}}) \to \End(\sA_k, \sB_k) $ is a finite $p$-group. 
\end{lemma}
\begin{proof}
This is well known. We just remark that now we can give a quick argument by taking advantage of Morrow's general result on the deformation of line bundles. Recall that there is a natural decomposition
\begin{equation}
\label{NSEnd}
    \NS(\sA_{\bar{\eta}} \times \sB_{\bar{\eta}}^\vee) = \NS(\sA_{\bar{\eta}}) \times \NS(\sB^\vee_{\bar{\eta}}) \times \Hom(\sA_{\bar{\eta}}, \sB_{\bar{\eta}}).
\end{equation} 
Now apply \cite[Thm~3.10, Ex.~3.12]{Morrow} to $A \times B^\vee, A, B$. 
\end{proof}

\begin{definition}
\label{CSpin-isog} 
Let $k$ be a perfect field with an algebraic closure $\bar{k}$ and $s, s'$ be $k$-points on $\shS_{\sK}(L)$. Let $f : \shA_s \to \shA_{s'}$ be a quasi-isogeny over $k$ which respects the $\IZ/2\IZ$-grading and $\Cl(L)$-action on $\shA_s, \shA_{s'}$. If $\mathrm{char\,} k = 0$, we say $f$ is a \textit{CSpin-isogeny} if it sends $\bpi_{\ell, s \tensor \bar{k}}$ to $\bpi_{\ell, s' \tensor \bar{k}}$ for every prime $\ell$. If $\mathrm{char\,} k = p$, we say $f$ is a \textit{CSpin-isogeny} if it sends $\pi_{\ell, s \tensor \bar{k}}$ to $\bpi_{\ell, s' \tensor \bar{k}}$ for every prime $\ell \neq p$ and $\bpi_{\cris, s}$ to $\bpi_{\cris, s'}$. 
\end{definition}

\begin{proposition}
\label{ssdefLEnd}
Let $k$ be an algebraically closed field of characteristic $p > 2$. For any $k$-point $s$ on $\shS_\sK^\ss(L)$, the quadratic form $(\LEnd(\shA_s), f \mapsto f \circ f)$ has the following properties: 
\begin{enumerate}[label=\upshape{(\alph*)}]
    \item $\LEnd(\shA_s) \tensor \IR$ is negative definite.
    \item The natural maps $\LEnd(\shA_s) \tensor \IZ_\ell \stackrel{\sim}{\to} \bL_{\ell, s}$ and $\LEnd(\shA_s) \tensor \IZ_p \to \bL_{\cris}^{F = 1}$ are isomorphisms.
\end{enumerate}
\end{proposition}
\begin{proof}
(a) Let $f \in \LEnd(\shA_s)$. By \cite[Prop.\,7.18]{CSpin}, there exists a characteristic zero field $F$ which embeds into $\IC$ and an $F$-valued point $s_F$ which specializes to $s$ such that $f$ lifts to an element $f_\IC \in \LEnd(\shA_{s_\IC})$, where $s_\IC := s_F \tensor \IC$. Recall that $\bL_{\dR, s_\IC}^{1, -1}$ is generated by an element $\w$ such that $\< \w, \w \> = 0$ and $\< \w, \bar{\w} \> > 0$, we see that $\bL_{\dR, s_\IC}^{1, -1} \oplus \bL_{\dR, s_\IC}^{-1, 1}$ descends to a $2$-dimensional positive definite $\IR$-vector subspace of $\bL_{B, s_\IC} \tensor \IR$. Since $\bL_{B, s_\IC}$ has signature $(2, m-2)$ and the class of $f_\IC$ lies in $\bL_{\dR, s_\IC}^{(0, 0)}$, we see that $f \circ f < 0$.  \\\\
(b) If $k = \bar{\IF}_p$, then by \cite[Cor.\,6.11]{Keerthi} $\LEnd(\shA_s) \tensor \IQ_\ell \to \bL_{\ell, s} \tensor \IQ_\ell$ is an isomorphism. To check that it holds for a general $k$, apply \ref{generalk} and a specialization along DVR argument ( \cite[Tag~054F]{stacks-project}).
Consider the following diagram:
\begin{center}
    \begin{tikzcd}[row sep=small, column sep=small]
 & \LEnd(\shA_s) \tensor  \IZ_\ell \arrow[dl] \arrow[rr] \arrow[dd] & & \bL_{\ell, s} \arrow[dl] \arrow[dd] \\ 
  \End(\shA_s) \tensor \IZ_\ell \arrow[rr, crossing over] \arrow[dd] & & \End H^1_\et(\shA_s, \IZ_\ell)  \\ 
 & \LEnd(\shA_s) \tensor \IQ_\ell \arrow[dl] \arrow[rr] & & \bL_{\ell, s} \tensor \IQ_\ell \arrow[dl] \\ 
 \End(\shA_s) \tensor \IQ_\ell \arrow[rr] & & \End H^1_\et(\shA_s, \IQ_\ell)  \arrow[uu, crossing over, leftarrow]\\
\end{tikzcd}
\end{center}
All vertical diagrams except the one at the back are fiber diagrams, so the one at the back is also a fiber diagram. The argument for crystalline cohomology is entirely analogous.
\end{proof}

\begin{proposition}
\label{allisog}
Let $k$ be an algebraically closed field of characteristic $p$. Assume that $L_p$ is self-dual. Then for any two $k$-points $s, s'$ of $\shS_{p}^{\ss}(L)$, there exists a CSpin-isogeny $\psi : \shA_s \to \shA_{s'}$.
\end{proposition}
\begin{proof}
We first treat the case $k = \bar{\IF}_p$. This is already observed in \cite{HP}: The supersingular locus $\shS^\ss_p(L)$ coincides with the basic locus, and the $\bar{\IF}_p$-points of the basic locus consists of a single isogeny class (see (7.2.1) and Prop.~3.3.3 in \textit{loc. cit.}). The case for general $k$ follows again from \ref{generalk} and a specialization along DVR argument.  
\end{proof}

\begin{remark}
The computation which is subsumed in \cite[Prop.~3.3.3]{HP} is essentially the \textit{Hasse principle of quadratic forms}. Let $s \in \shS^\ss_p(L)(\bar{\IF}_p)$ and $L'_\IQ$ be the $\IQ$-lattice $\LEnd(\shA_s)_\IQ$. According to Kisin's description of the fibers of the map $\fk$ which sends isogeny classes on $\shS_p(L)(\bar{\IF}_p)$ to equivalence classes of Kottwitz triples (see  \cite[Prop.~4.4.9]{Modp}), the fiber of $\fk$ over the image of the isogeny class of $s$ is a torsor of a subgroup of $\ker (H^1(\IQ, \CSpin(L_\IQ')) \to \prod_{v} H^1(\IQ_v, \CSpin(L_\IQ')))$, where $v$ runs through all places of $\IQ$. There is commutative diagram 
\begin{center}
    \begin{tikzcd}
     H^1(\IQ, \CSpin(L'_\IQ)) \arrow{r}{} \arrow{d}{} & \prod_v H^1(\IQ_v, \CSpin(L'_\IQ)) \arrow{d}{} \\
     H^1(\IQ, \SO(L'_\IQ)) \arrow{r}{} & \prod_v H^1(\IQ_v, \SO(L'_\IQ)).
    \end{tikzcd}
\end{center}
Both vertical arrows are injective by Hilbert's theorem 90. The bottom arrow is injective by the Hasse principle. Hence the kernel of the top arrow is trivial. 
\end{remark}

\begin{lemma}
\label{relpos}
Assume that $L_p$ is self-dual and $k$ is an algebraically closed field of characteristic $p$. Let $s, s' \in \shS^\ss_{\sK_0}(L)(k)$ be any two points for which there exists a CSpin-isogeny $\psi : \shA_s \to \shA_{s'}$. If the induced map $\bL_{\cris, s'} \tensor K_0 \stackrel{\sim}{\to} \bL_{\cris, s} \tensor K_0$ restricts to a map $\bL_{\cris, s'} \stackrel{\sim}{\to} \bL_{\cris, s}$, then $p^h \psi$ is a prime-to-$p$ isogeny for some $h$. 
\end{lemma}
\begin{proof}
Identify $\bH_{\cris, s} \tensor K_0$ with $\bH_{\cris, s'} \tensor K_0$ through $\psi$. In this proof we temporarily use $G$ to denote $\CSpin(\bL_{\cris, s})$ and $G^\ad$ to denote $\SO(L_{\cris, s})$. Let $T$ be a maximal torus in $G$ and $\Ohm_{G}$ be the Weyl group of $T$. By Cartan's decomposition, the relative position between the $W$-lattices $\bH_{\cris, s}$ and $\bH_{\cris, s'}$ is determined by a cocharacter $\mu : \IG_m \to G_W$ up to the action of $\Ohm_G$ such that
$$ \bH_{\cris, s} = g \cdot \bH_{\cris, s'} \text{ for some } g \in G(W) \mu(p) G(W). $$
Let $\Ohm_{G^\ad}$ be the Weyl group of the image of $T$. Now the condition that $\psi$ induces a $W$-integral isometry $\bL_{\cris, s} \stackrel{\sim}{\to} \bL_{\cris, s'}$ implies that the composition $\ad \circ \mu$ is the trivial cocharacter up to the action by $\Ohm_{G^\ad}$. By the exact sequence (\ref{fundexactseq}), up to the action of $\Ohm_G$, $\mu$ factors through $\IG_m \subset G$. Therefore, $p^h \psi$ induces an isomorphism $H^1_\cris(\shA_{s}/W) \to H^1_\cris(\shA_{s'}/W)$ which commutes with the CSpin structures for some $h$. In particular, $p^h \psi$ is a prime-to-$p$ CSpin-isogeny $\shA_{s'} \to \shA_{s}$. 
\end{proof}

\subsection{Period Morphisms} 
\label{sec: Period Morphism}
We briefly recap the set-up for the period morphisms, which was used in \cite{Andre}, \cite{Rizov}, \cite{Keerthi}, \cite{Maulik}, \cite{Charles} and \cite{Charles2} in various degrees of generality. The use of period morphisms for mod $p$ reductions of hyperk\"ahler varieties first appeared in \cite{Charles}. Here we go through the construction in greater detail because there is a minor mistake in early literature (\cite[Rmk~5.12]{Taelman2}) concerning artificial orientations and for our purposes we need to keep track of connected components of the domains of our period morphisms. 

\begin{definition}
\label{defHKnfam}
Let $p > n + 1$ be a prime and $S$ be a $\IZ_{(p)}$-scheme.
\begin{enumerate}[label=\upshape{(\alph*)}]
    \item A $\HKn$-scheme over $S$ is a smooth proper morphism of schemes $f : \sX \to S$ such that for every geometric point $s \to S$, the fiber $\sX_s$ is a $\HKn$-type variety. A $\HKn$-space over $S$ is a smooth proper morphism of algebraic spaces $f : \sX \to S$ such that for some \'etale cover $S' \to S$, the pullback $\sX_{S'}$ is a $\HKn$-scheme. 
    \item A polarization (resp. primitive polarization) on a $\HKn$-space $f : \sX \to S$ is a global section $\bxi \in \underline{\Pic}_{\sX/S}(S)$ such that for every geometric point $s \to S$, the fiber $\bxi_s$ is a polarization (primitive polarization) on $\sX_s$. Here $\underline{\Pic}_{\sX/S}$ stands for the relative Picard functor. $\bxi$ is said to be of degree $d$ for a number $d$ if for every geometric point $s \to S$, $\bxi_s^{2n} = d$. 
    \end{enumerate}
\end{definition}

\begin{notation}
\label{not: define sheaves}
Let $S$ be a $\IZ_{(p)}$-scheme and $(f : \sX \to S, \bxi)$ be a polarized $\HKn$-space over $S$. Let $\bH_B^2$ be the second relative Betti cohomology of $\sX|_{S_\IC}$. For every prime $\ell$, let $\bH^2_\ell$ be the relative second \'etale cohomology of $\sX|_{S \tensor_\IZ \IZ[1/\ell]}$ with coefficients in $\underline{\IZ}_\ell$. Let $\bH_{\dR}^2$ be the vector bundle on $S$ given by the second relative de Rham cohomology of $\sX$, which comes with a natural filtration $\Fil^\bullet$.  We put together the relative $\ell$-adic cohomology sheaves $\bH_\ell^2$ to form $\bH_{\what{\IZ}}^2 := \prod_{\ell} \bH_\ell^2$ and $\bH_{\what{\IZ}^p}^2 := \prod_{\ell \neq p} \bH_\ell^2$. Over $S \tensor \IF_p$, let $\bH^2_{\cris}$ denote the crystal of vector bundles on $\Cris(S \tensor \IF_p/\IZ_p)$ given by the second crystalline cohomology of $\sX \tensor \IF_p$. For $* = B, \ell, \dR, \cris$, denote by $\bP^2_*$ the primitive part of $\bH^2_*$.
\end{notation}

In the above notation, we remark that we can put a natural Beauville-Bogomolov form on $\bH_{\what{\IZ}^p}^2(1)$ by first setting the value for $c_1(\bxi_{S})$ and then putting a pairing on the primitive part, using (\ref{wqperp}) with $w$ being the cup product, so that on each fiber over a point, the form agrees with that defined by \ref{extBBF0}. We make the following definitions, which are direct generalizations of \cite[Def.~1.1.5, 3.2.1]{Rizov1} and \cite[\S3.10]{Keerthi}\footnote{The definition of level structures in \cite[\S3.10]{Keerthi} should be slightly modified to take into account of the orientations in order for $H^0(S, I^p/\sK^{p, \ad})$ to be finite. This is also important for the construction of the period morphism (cf. \cite[Rmk~3.8]{Yang}).}:
\begin{definition}
\label{def: orientation and level structures}
Suppose that there is a pointed lattice $(\Lambda, \lambda)$ over $\IZ$ such that for every geometric point $s \to S$, $(\H^2_\et(\sX_s, \what{\IZ}^p), c_1(\bxi_s)) \iso (\Lambda \tensor \what{\IZ}^p, \lambda)$. Set $L^p := (\lambda^p)^\perp$. A (prime-to-$p$) \textbf{artificial orientation} is an isometric trivialization 
$$ \epsilon^p : \underline{\det(L^p)}_S \stackrel{\sim}{\to} \det(\bP^2_{\what{\IZ}^p}). $$
Let $\sK^p \subseteq \CSpin(L^p)$ be a compact open subgroup and $\sK^{p, \ad}$ be its image in $\SO(L^p).$ A \textbf{$\sK^p$-level structure} (with respect to $\epsilon^p$) is a global section $[\eta^p]$ in $\H^0(S, I^p/\sK^p)$, where $I^p$ is the (pro)-\'etale sheaf over $S$ defined by 
$$ I^p(S') = \{ \text{Isometries } \eta^p : \underline{\Lambda}^p_{S'} \stackrel{\sim}{\to} \bH_{\what{\IZ}^p, S'}^2(1) \text{ such that } \eta^p(\lambda^p) = c_1(\bxi_S) \text{ and } \eta \text{ preserves } \epsilon^p_{S'} \} $$
for every \'etale morphism $S' \to S$. 
\end{definition}

Note that $\sK^{p, \ad}$ acts trivially on $\disc(L^p)$, so every element in $\sK^{p, \ad}$ can be extended (necessarily uniquely) to an element of $\O(\Lambda^p, \lambda^p)$ (cf. \ref{2.1.7}).

\subsubsection{}\label{period+} Let $R$ be either $\IZ_{(p)}$ or $W(k)$ for some perfect field $k$ in characteristic $p$. Let $E$ be the fraction field of $R$ and choose an embedding $E \into \IC$. Let $S$ be an $R$-scheme and $(f : \sX \to S, \bxi)$ be a primitively polarized $\HKn$-space over $S$. We make the following assumptions on $S$: (a) $\sX$ is everywhere a universal deformation, i.e., for every geometric point $s \to S \tensor \IF_p$, the restriction of $\sX$ to the formal neighborhood of $s$ in $S \tensor_R W(k(s))$ can be identified with the universal family over $\Def(\sX_s, \bxi_s)$. (b) The hypothesis of \ref{def: orientation and level structures} is satisfied with $(\Lambda, \lambda)$ and set $L := \lambda^\perp$. Let $\wt{S}$ be the double cover of $S$ such that a morphism $T \to \wt{S}$ corresponds to a morphism $T \to S$ together with a trivialization $\epsilon_{2, T} : \underline{\det(L_2)}_T \stackrel{\sim}{\to} \det(\bP^2_{2, T})$. Let $\wt{S}^\#$ be the finite \'etale cover of $\wt{S}$ such that a morphism $T \to \wt{S}^\#$ corresponds to a morphism $T \to \wt{S}$ and a trivialization $\Delta_T : \underline{\disc(L^p)}_T \to \disc(\bP^2_{\what{\IZ}^p, T})$.

Recall the universal property of $\Sh^\ad_{\sK}(L)(\IC)$:
\begin{proposition}
\label{prop: universal property of Shimura}
For any complex analytic stack $T$, a morphism $T \to \Sh^\ad_{\sK}(L)(\IC)$ corresponds to tuple $((\bU_\IQ, \Fil^\bullet \bU_\IQ \tensor_\IQ \shO_T), \eta_T, \epsilon_T)$ where 
\begin{itemize}
    \item $(\bU_\IQ, \Fil^\bullet \bU_\IQ \tensor_\IQ \shO_T)$ is a polarized variation of $\IQ$-Hodge structures with Hodge numbers $h^{-1, 1} = h^{1, -1} = 1$ and $h^{0, 0} = \mathrm{rank\,} L - 2$ such that every point $t \in T$, $\bU_{\IQ, t} \iso L_\IQ$ as quadratic forms;
    \item $\epsilon_T$ is an isometry $\underline{\det(L_\IQ)} \sto \det(\bU_\IQ)$;
    \item $[\eta_T]$ is a global section of $I/ \sK^\ad$, where $I$ is the local system over $T$ locally given by the set of isometries from $\underline{L \tensor \IA_f}$ to $\bU_\IQ \tensor \IA_f$ which are compatible with $\epsilon_T$.
\end{itemize}
\end{proposition}
We remark that $\bU_\IQ$ is equipped with a natural $\IZ$-structure $\bU$, which is given by $\bU = \bU_\IQ \cap \eta_T(L \tensor \what{\IZ})$. Since $L$ contains a copy of the standard hyperbolic plane, by \ref{strongapprox} $\bU_t \iso L$ as quadratic $\IZ$-lattices. One easily checks that, when $\sK = \sK_0$, our description of the universal property agrees with that of \cite[Prop.~4.3]{Keerthi}.

Using the above universal property, we easily construct a natural morphism $\rho : \wt{S}^\#_\IC \to \Sh^\ad_{\sK_0}(L)_\IC$. By construction, there is an isometry $\alpha_{B}: \rho_{\IC}^* \bL_{B}(-1) \stackrel{\sim}{\to} \bP^2_{B, \IC}$ of $\IZ$-local systems and $\alpha_{\dR, \IC} : \rho_{\IC}^* \bL_{\dR}(-1) \to \bP^2_{\dR, \IC}$ of filtered vector bundles over $\wt{S}^\#_\IC$. 

\begin{proposition}
The morphism $\rho_{\IC}$ descends to a morphism $\rho_{E} : \wt{S}^\#_E \to \Sh^\ad_{\sK_0}(L)_E$.
\end{proposition}
\begin{proof}
This follows from the proof of  \cite[Cor.~5.4]{Keerthi} (cf. \cite[Prop.~10]{Charles}), with addition input from \cite{Andre}: The main point is to show that for any $s \in \wt{S}^\#_\IC$, and $\wt{t} \in \Sh_{\sK_0}(L)_\IC$ lifting $t := \rho(s)$, the isomorphism 
$$ \P^2(\sX_s, \IZ(1)) \iso \bpi_{B, \wt{t}}(\End(\H^1(\shA_{\wt{t}}, \IZ))) $$
induced by $\alpha_B$ is given by an \textit{absolute Hodge} cycle. Note that we already know that $\bpi_{B, \wt{t}}$ is absolute Hodge, because it is a Hodge tensor on an abelian variety. Therefore, both sides of the above isomorphism can be viewed as objects in the category of $\IZ$-motives over $\IC$ in Madapusi-Pera's terminology (\cite[\S2.2]{Keerthi}). Up to some slight adjustments, Andre's result \cite[Prop.~6.2.1]{Andre} tells us that the above isomorphism is actually given by a motivated cycle, which is in particular absolute Hodge. Now the proof of \cite[Cor.~5.4]{Keerthi} proceeds without change: We use the modular interpretation of $\Sh^\ad_{\sK_0}(L)_\IC$ to check that $\rho_\IC$ is $\Aut(\IC/E)$-equivariant. Therefore, $\rho_\IC$ descends to $\rho_E$. 
\end{proof}

\begin{corollary}
\label{prop: Extend epsilon}
There is a unique extension of $\epsilon_2$ to a trivialization 
$$ \epsilon : \underline{\det(L)}_{\wt{S}^\#_E} \sto \det(\bP^2_{\what{\IZ}, \wt{S}^\#_E})  $$
which is characterized by the following property: For every $\IC$-point $s$, there exists an isomorphism $\det(L) \to \det(\P^2(\sX_s, \IZ))$ which gives $\epsilon|_{s}$ when tensored with $\what{\IZ}$. Moreover, the prime-to-$p$ part of $\epsilon$ extends uniquely to an artificial orientation $\epsilon^p$ on $\sX|_{\wt{S}^\#}$.
\end{corollary}
\begin{proof}
It is not hard to see that such $\epsilon$ is unique provided that it exists. Deligne's big monodromy argument in \cite{Deligne1} allows us to show that there is a unique isometry $\alpha_{\what{\IZ}} : \rho_E^* \bL_{\what{\IZ}} \stackrel{\sim}{\to} \bP^2_{\what{\IZ}}|_{\wt{S}^\#_E}(1)$ whose restriction to the $\IC$-fibers is compatible with $\alpha_B$ via Artin's comparison theorem (see also \cite[Prop.~5.6(1)]{Keerthi}). Recall that, for a chosen geneator $\delta$ of $\det(L)$, there are global sections $\bd_{\what{\IZ}}$ of $\det(\bL_{\what{\IZ}})$ over $\Sh^\ad_{\sK_0}(L)$. Via the isomorphism $\alpha_{\ell}$, we transport $\delta$ to a global section of $\bP^2_{\what{\IZ}, \wt{S}^\#_E}(1)$. Since $\wt{S}^\#$ is normal, we can further extend the prime-to-$p$ part of $\epsilon$ (necessarily uniquely) to an artificial orientation on $\wt{S}^\#$. 
\end{proof}

\subsubsection{} \label{sec: period++} Let $\sK^p$ be a compact open subgroup of $\sK_0^p = \CSpin(L^p)$. Note that the image $\sK^{p, \ad} \subset G^\ad(\IA_f)$ of $\sK^p$ stabilizes $L^p$ and $\disc(L^p)$. Let $\wt{S}_{\sK^p}$ denote the finite \'etale cover of $\wt{S}^\#$ which parametrizes $\sK^{p}$-level structures on $\sX|_{\wt{S}^\#}$ compatible with $\epsilon^p$ and $\Delta$. Set $\sK := \sK_{0, p} \sK^p$. The morphism $\rho_E : \wt{S}_{E}^\# \to \Sh^\ad_{\sK_0}(L)_E$ can be promoted to $\rho_{\sK^p, E} : \wt{S}_{\sK^p, E} \to \Sh^\ad_\sK(L)_E$. The extension property of $\shS_p(L)$ allows us to extend $\rho_{\sK, E}$ to $\rho_\sK : \wt{S}_{\sK^p} \to \shS_\sK(L)_R$. In turn, by taking quotients, we obtain a morphism $\rho : \wt{S}^\# \to \shS^\ad_{\sK_0}(L)_R$ of Deligne-Mumford stacks. 

\begin{theorem}
\label{compMotive}
\begin{enumerate}[label=\upshape{(\alph*)}]
    \item $\rho$ is \'etale. 
    \item For every prime number $\ell$, there exists an isometry $\alpha_\ell : \rho^* \bL_\ell(-1) \stackrel{\sim}{\to} \bP^2_\ell$ over $\wt{S}^\#_E$ such that over $\wt{S}_\IC^\#$, $\alpha_\ell$ and $\alpha_B$ are compatible via the Artin comparison isomorphism; if $\ell \neq p$, the isomorphism extends uniquely over $\wt{S}^\#$.
    \item $\alpha_{\dR, \IC} : \rho^* \bL_\dR|_{\wt{S}^\#_\IC}(-1) \stackrel{\sim}{\to} \bP^2_\dR |_{\wt{S}^\#_\IC}$ over $\wt{S}^\#_\IC$ descends to $\alpha_{\dR,E} : \rho^* \bL_{\dR}|_{\wt{S}^\#_E}(-1) \stackrel{\sim}{\to} \bP_\dR^2|_{\wt{S}^\#_E}$ over $\wt{S}^\#_{E}$.
    \item $\alpha_{\dR, E}$ can be extended to $\alpha_{\dR} : \rho^* \bL_{\dR}(-1) \stackrel{\sim}{\to} \bP^2_\dR$, the isomorphism of crystals $\alpha_\cris : \rho^* \bL_{\cris} (-1) \stackrel{\sim}{\to} \bP_\cris^2$ induced by which preserves the Frobenius action. 
    \item For every geometric point $t$ on $\wt{S}^\#(k)$ and a point $s$ on $\shS_{\sK_0}(L)$ lifting $\rho(t)$, there exists a natural isomorphism
\begin{equation}
    \LEnd(\shA_{s}) \stackrel{\sim}{\to} \< \bxi_{t} \>^\perp \subset \NS(\shX_t)
\end{equation}
which is compatible with the isomorphisms in $(b)$ and $(d)$.
\end{enumerate}
\end{theorem}
\begin{proof}
These follow from straightforward adaptions of the proofs of results in \cite[\S5]{Keerthi}. We just explain how to adapt the proof of \cite[Thm~5.17(4)]{Keerthi} to get (e). If $\mathrm{char\,} k(t) = 0$, this follows from Hodge theory, so we just treat the case when $\mathrm{char\,} k(t) = p$. As in \textit{loc. cit.}, one can construct a map $\LEnd(\shA_s) \to [\< \bxi_t \>^\perp \subset \Pic(\shX_t)]$ such that the following diagrams commutes:
\begin{center}
    \begin{tikzcd}
    \LEnd(\shA_s) \arrow{r}{} \arrow{d}{} & \< \bxi_t \>^\perp  \arrow{d}{} \\
    \bL_\cris^{F = 1} \arrow{r}{\alpha_{\cris, t}} & \P^2_\cris(\sX_t/W)^{F = p}
    \end{tikzcd}
    \begin{tikzcd}
    \LEnd(\shA_s) \arrow{r}{} \arrow{d}{} & \< \bxi_t \>^\perp  \arrow{d}{} \\
    \bL_\ell \arrow{r}{\alpha_{\ell, t}} & \P^2_\et(\sX_t, \IZ_\ell(1))
    \end{tikzcd}
\end{center}
We now sketch the construction of the map  $\LEnd(\shA_s) \to \< \bxi_t \>^\perp$. Let $f \in \LEnd(\shA_s)$. By \cite[Prop.~7.18]{CSpin}, for some characteristic zero field $F$ there exists an $F$-valued point $s_F$ which specializes to $s$ such that $f$ lifts to an element $f_{F} \in \LEnd(\shA_{s_{F}})$. Since $\rho$ is \'etale, $s_F$ induces a lift $t_{F}$ of $t$, which corresponds to a deformation $\sX_{t_{F}}$ of $\sX_t$. Embed $F$ into $\IC$. Now $f_{F}$ induces a Hodge class on $\sX_{t_{F}}$. By Lefschetz (1, 1) theorem, up to replacing $F$ by a finite extension, there exists a line bundle $\zeta_F$ on $\sX_{t_F}$ realizing this class. Specialize $\zeta_{F}$ to a line bundle $\zeta \in \Pic(\sX_t)$. The map $f \mapsto \zeta$ is well-defined and has the desired properties because $\zeta$ is uniquely determined by its Chern classes (cf. \ref{NSfree}). 

If $\sX_s$ is not supersingular, then we can run the argument backwards and construct a map $\< \bxi_t \>^\perp \to \LEnd(\shA_s)$ because we can lift every pair of line bundles on $\shX_s$ to characteristic zero by \ref{flatloc}(a). If $\sX_s$ is supersingular, then we conclude by \ref{ssdefLEnd}(b) that the map  $\LEnd(\sA_s) \to \< \xi_t \>^\perp$ has to be surjective.  
\end{proof}

\section{Proofs of Theorems}
 
\subsection{Invariance of Deformation Types}
Let $k$ be an algebraically closed field of characteristic $p > n + 1$ and let $X$ be a $\HKn$-type variety over $k$. Let $\xi$ be a primitive polarization on $X$ such that for some finite flat extension $V$ of $W$, $(X, \xi)$ lifts to $(X_V, \xi_V)$ over $V$ whose generic fiber is of $\HKn$-type. By applying Artin's approximation theorem to $\Def(X; \xi)$, we obtain a family $(\sX \to S, \bxi)$ which universally deforms $(X, \xi)$ and satisfies the hypotheses of \ref{period+}. By \ref{flatloc} and \cite[Tag~055J]{stacks-project}, up to shrinking $S$, we may assume that $S$ is connected, flat over $W$, quasi-healthy and has connected generic fiber.

Using the notations in \ref{sec: period++}, we have a period morphism $\rho : \wt{S}^\# \to \shS^\ad_{\sK_0}(L)_W$. Let $b$ be the $k$-point on $S$ whose fiber is $(X, \xi)$ and set $S_0 = S \tensor_W k$. 
\begin{proposition}
\label{tateImprove}
If $X$ is supersingular, then $c_1 : \NS(X) \tensor \IZ_\ell \to \H^2_\et(X, \IZ_\ell)$ and $c_1 : \NS(X) \tensor \IZ_p \to \H^2_\cris(X/W)^{F = p}$ are isomorphisms. 
\end{proposition}
\begin{proof}
For the first statement, note that \ref{compMotive}(e) and \ref{ssdefLEnd} imply that both morphisms are isomorphisms rationally. To see that the integral statements hold, apply \ref{NSfree}. 
\end{proof}

\begin{lemma}
\label{ssdim}
If $X$ is supersingular, then the supersingular locus $S_0^\ss$ on $S$ is of dimension at most $10$.
\end{lemma}
\begin{proof}
If $\xi$ is of degree prime to $p$, then $L$ is self-dual at $p$, and the result follows directly from the \'etaleness of $\rho$ and the dimension formula \cite[Thm~C]{HP}.

To treat the general case, we give a geometric argument which is inspired by the proof of \cite[Thm~5.6]{Ogus}. It suffices to bound the dimension of the smooth locus of $S_0^\ss$, so we may move $b$ (and hence deform $(X, \xi)$) a little and assume that $S^\ss_0$ is smooth. Let $\what{S}^\ss$ be the completion of $S^\ss$ at $b$. By \ref{K3crystal}, the restriciton of $\bH^2_\cris$ to $\what{S}^\ss$ has the structure of a K3 crystal. In fact, it is a supersingular K3 crystal over $\what{S}^\ss$ in the sense of \cite[Prop.~5.1 (i)bis or (ii)]{Ogus}. Let $\eta$ be the generic point of $\what{S}$ and let $\bar{\eta}$ be a geometric point over $\eta$. Ogus' crystalline proof \cite[Prop.~5.5]{Ogus} of Artin's main theorem in \cite{Artin} applies without change in our situation and gives that $\Pic(\sX_{\what{S}^\ss}) = \Pic(\sX_\eta) = \Pic(\sX_{\bar{\eta}})$. Via the specialization map $\Pic(\sX_{\what{S}}) \to \Pic(\sX)$, we may view these as subspaces of $\Pic(X)$. 

Now since $(\sX \to S, \bxi)$ universally deforms $(X, \xi)$, there is an injective morphism from the tangent space $T_b S$ to the tangent space of $\Def(X)$, which can be identified with $\H^1(X, T_X)$. Let $T^\ss$ be the image of $T_b S^\ss$ in $\H^1(X, T_X)$. Let $d \log : \Pic(X) \to \H^1(X, \Ohm_X^1)$ be the natural morphism. By obstruction theory \cite[Cor.~1.14]{Ogus}, $T^\ss$ annhilates $d \log (\Pic(X_{\eta}))$ under the cup product pairing 
$$ \H^1(X, T_X) \times \H^1(X, \Ohm_X^1) \to \H^2(X, \sO_X) \iso k. $$
Since the above pairing is perfect \ref{cor: perfect pairing}, we have 
\begin{align}
\label{inequality}
    \dim \what{S}^\ss = \dim T^\ss \le \mathrm{codim}_{d \log (\Pic(X_{\eta}))} \H^1(X, \Ohm_X^1) \le \mathrm{codim}_{c_{1, \dR}(\Pic(X_\eta))} \Fil^1 \H^2_\dR(X/k).
\end{align}
Note that  
\begin{align}
    \label{discbound}
    \dim \frac{\H^2_\dR(X/k)}{c_{1, \dR}(\Pic(X_\eta))} \le \mathrm{length} \frac{\H^2_\cris(X/W)}{c_{1, \cris}(\Pic(X_\eta)) \tensor W}.
\end{align}

By the preceeding lemma, $\Pic(X_{\bar{\eta}}) \tensor \IZ_p = \H^2_\cris(\sX_{\bar{\eta}}/W(\bar{\eta}))^{F = p}$. Since $\H^2_\cris(\sX_{\bar{\eta}}/W(\bar{\eta}))$ is a supersingular K3 crystal of rank $23$, $\disc(\Pic(X_{\bar{\eta}})) = (\IZ/p \IZ)^{2 \sigma}$ for some $\sigma \le 11$ by \cite[Prop.~3.13]{Ogus}. Since $c_{1, \cris} \tensor W : \Pic(X_\eta) \tensor W \to \H^2_\cris(X/W)$ is an isometric embedding of quadratic lattices over $W$ and $\H^2_\cris(X/W)$ is self-dual, the length of the cokernel is at most $11$. Then we conclude by (\ref{inequality}) that $\dim \what{S} \le 10$, as $\Fil^1 \H^2_\dR(X/k)$ is of codimension one in $\H^2_\dR(X/k)$.  \end{proof}

\begin{proposition}
\label{liftpair} 
Let $(\zeta_0, \zeta_1, \cdots, \zeta_{9})$ be a tuple of line bundles on $X$ which spans a direct summand of $\Pic(X)$ with $\zeta_0 = \xi$. Then there exists a finite flat extension $W'$ of $W$, and a lifting $X_{W'}$ of $X$ to $W'$ such that the tuple $(\zeta_0, \zeta_1, \cdots, \zeta_{9})$ deforms to $X_{W'}$. 
\end{proposition}
\begin{proof}
This is a direct generalization of \cite[Prop.~1.5]{Charles2}, which was built on the proof of \cite[Prop.~A.1]{LO1} but stated a slightly stronger result.\footnote{That is, \cite[Prop.~1.5]{Charles2} implies that the dvr $R$ in \cite[Prop.~A.1]{LO1} can be taken to be a \textit{finite} flat extension of $W$.} The proof below is essentially an adaptation of the arguments in \cite[App.~A]{LO1} with an explanation of how the minor strengthening is done. By \ref{flatloc}(a) it suffices to treat the case when $X$ is supersingular. Choose $m \gg 0$ which is prime to $p$ such that $m \xi$ is very ample and $\H^i(X, m\xi) = 0$ for all $i > 0$. Let $P$ be the Hilbert polynomial of $m \xi$ and consider $\Hilb^{+, m}_P$, as defined in \ref{hilb+}. Let us denote $\Hilb^{+, m}_P$ simply by $U$ and the restriction of the universal family over it by $\sY$. By choosing global sections of $m\xi$ we find a point $s \in U$ such that $\sY_s \iso X$. 

Now consider the product $\shP := \Pic_{\sY/U} \times_U \cdots \times_U \Pic_{\sY/U}$ of $10$ copies of the relative Picard scheme $\Pic_{\sY/U}$ over $U$. The tuple $(\xi = \zeta_0, \cdots, \zeta_{9})$ corresponds to a lift $t \in \shP$ of $s \in U$. Note that as long as we show that some subscheme of $\shP$ has an irreducible component passing through $t$ which is of finite type and flat over $W$, we can find a $W'$ point $t_{W'}$ on $\shP$ extending $t$ for some finite flat extension $W'$ of $W$ by a standard systems of parameters argument (see for example the first paragraph of \cite[p.~63]{Deligne2}). Once such that $t_{W'}$ is found, we set $X_{W'} = \sY|_{t_{W'}}$. Note that since $W'$ is strictly Henselian, $\Pic_{X_{W'}/W'} (W') \iso \Pic(X_{W'})$ (see \cite[p.~203]{BLR}), so that $X_{W'}$ indeed carries lifts of $\xi$ and $\zeta_j$'s. 

Now it suffices to show that there exists a characteristic zero point on $\shP$ which specializes to $t$. Since $\Def(X; \xi, \zeta_1, \cdots, \zeta_{9})$ has dimension at least $11$, and the formal neighborhood of $t$ in $U$ is obviously smooth over $\Def(X; \xi)$, by \ref{ssdim} we can find a point $t' \in U(k[\![x]\!])$ which extends $t$ such that if $\bar{\eta}$ is a geometric generic point of $\mathrm{Spec\,} k[\![t]\!]$, the fiber $\sY_{t' \tensor \bar{\eta}}$ is not supersingular and carries lifts of all $\zeta_j$'s. Then we apply \ref{flatloc}(a) to lift $(\sY_{t' \tensor \bar{\eta}}; \xi, \zeta_1, \cdots, \zeta_9)$ further over $W(\bar{\eta})$, which gives rise to some $W(\bar{\eta})$-valued point on $\shP$ lifting $t'$. The image of the generic point of this $W(\bar{\eta})$-valued point will do the job. 
\end{proof}

\begin{corollary}
\label{transmit+} 
Let $A$ be any mixed characteristic complete DVR with residue field $k$. Suppose $X_A$ is a lifting of $X$ over $A$ such that some primitive polarization $\zeta$ on $X$ lifts to $X_A$, then the generic fiber of $X_A$ is also of $\HKn$-type. 
\end{corollary}
\begin{proof}
This follows directly from the above proposition and \ref{charpspread}.
\end{proof}

Note that Theorem~\ref{transmit} is now proved: (a) is given by \ref{flatloc} and (b) is slightly generalized by the above. 

\subsection{Global Moduli Theory}

\begin{notation}
Let $n, d$ be fixed positive numbers and $p$ be a prime. For any field $\kappa$ over $\IZ_{(p)}$, let $\sP_{n, d}(\kappa)$ be the set of isomorphism classes of primitively polarized $\HKn$-type varieties $(X, \xi)$ over $\kappa$ such that $\xi^{2n} = d$. If $m$ is a positive number, let $\sQ_{n, d, m}(\kappa)$ denote the set of possible Hilbert polynomials of $m \xi$, for an element $(X, \xi) \in \sP_{n, d}(\kappa)$. For any tuple $(k, \tau, (Y, \zeta))$, where $k$ is an algebraically closed field of characteristic $p$, $\tau$ is an isomorphism $\bar{K}_0 \iso \IC$, $(Y, \zeta)$ is a element of $\sP_{n, d}(\IC)$, denote by $\sS(Y, \zeta, k, \tau)$ the subset of $\sP_{n, d}(k)$ consisting of elements $(X, \xi)$ over $k$ with the following property: For some finite flat extension $V$ over $W = W(k)$ and embedding $V \subset \bar{K}_0 \stackrel{\tau}{\iso} \IC$, $(X, \xi)$ lifts to $(X_V, \xi_V)$ over $V$ such that there exists a parallel transport operator $\H^2(X_{V}(\IC), \IZ) \sto \H^2(Y, \IZ)$ which sends $c_1(\xi_V(\IC))$ to $c_1(\zeta)$. 
\end{notation}

Our first goal is to prove some boundedness results:

\begin{theorem}
\label{bounded}
Let $n, d$ be fixed positive numbers and $p > n + 1$ be a fixed prime. Let $\kappa$ be any algebraically closed field over $\IZ_{(p)}$. There exist numbers $m, N$ which depend only on $n, d, p$ such that for any $(X, \xi) \in \sP_{n, d}(\kappa)$, $m \xi$ is very ample, $\H^i(X, m \xi) = 0$ for all $i > 0$, and $\H^0(X, m \xi) \le N$. Moreover, the set $\sQ_{n,d, m}(\kappa)$ is finite and independent of $\kappa$. 
\end{theorem}

Before starting the proof we first prove a connectedness lemma.  
\begin{lemma}
\label{Sasha}
Let $R$ be a complete DVR with a characteristic zero fraction field $K_0$ and an algebraically closed residue field $\kappa$. Let $T = \mathrm{Spec\,} B$ be a connected flat affine scheme of finite type over $R$. If the special fiber $T \tensor_R \kappa$ is nonempty and reduced, then the generic fiber of $T$ is geometrically connected.
\end{lemma}
\begin{proof}
It suffices to show that for any finite Galois extension $K$ of $K_0$, $T \tensor_R K$ is connected. Let $V$ be the integral closure of $R$ in $K$. Since $K$ is separable over $K_0$, $V$ is also a DVR. By \cite[Tag~055J]{stacks-project}, it suffices to show that for any such $V$, $T' = T \tensor_R V$ is connected. The Galois group $\Gal(K/K_0)$ acts on the $W$-algebra $B' := B \tensor_R V$ through its action on $V$. Since $\kappa$ is algebraically closed, we can canonically identify $T \tensor_R \kappa$ and $T' \tensor_V \kappa$. Let $b \in T(\kappa)$ be a $\kappa$-point. It corresponds to a maximal ideal $\fm$ in $B$, or a $\Gal(K/K_0)$-invariant maximal ideal $\fm'$ of $B'$. 

By way of contradiction, suppose that $T'$ is not connected. Then there exists a nontrivial idempotent $e \in B'$ which corresponds to the connected component containing $b$, i.e., $\mathrm{ann}(e) \subseteq \fm'$ and $\mathrm{Spec\,} B'/ \mathrm{ann}(e)$ is connected. Let $\sigma \in \Gal(K/K_0)$ be any element. Then $\mathrm{Spec\,} B / \mathrm{ann}(\sigma(e))$ is also a connected component of $T'$, which is either the same as the one defined by $e$ or is disjoint from it. In the former case, $e = \sigma(e)$. In the latter, $e \sigma(e) = 0$ and $\mathrm{ann}(e) + \mathrm{ann}(\sigma(e)) = B'$.

Note that the latter case cannot happen: Since $\fm'$ is invariant under the $\Gal(K/K_0)$-action, $\mathrm{ann}(e) + \mathrm{ann}(\sigma(e)) \subseteq \fm'$, which contradicts $\mathrm{ann}(e) + \mathrm{ann}(\sigma(e)) = B'$. Therefore, $e = \sigma(e)$ for every $\sigma \in \Gal(K/K_0)$. However, if this is true for every $\sigma$, then $e \in B$, which contradicts the connectedness of $B$. \end{proof}

\begin{remark}
We will apply the above lemma to $R = W$. The hypotheses are added to rule out schemes like $W[x]/(x^2 - p)$ (nonreduced special fiber), $K_0[x]/(x^2 - p)$ (empty special fiber), and $\IZ_p[x]/(x^2 - a)$ for some $a \in \IZ_p$ which is not a quadratic residue (residue field not algebraically closed).
\end{remark}

The key step for \ref{bounded} is the following proposition: 
\begin{proposition} 
\label{prop: bounded}
Assume that $p > n + 1$. Set $k = \bar{\IF}_p$ and choose an isomorphism $\tau : \bar{K}_0 \iso \IC$. For each $(Y, \zeta) \in \sP_{n, d}(\IC)$, there exists a variety $T$ of finite type over $k$ and primitively polarized $\HKn$-scheme $(\sY, \bzeta)$ over $T$ such that for any $(X, \xi) \in \sS(Y, \zeta, k, \tau)$, there exists an $k$-point $t \in T$ with $(\sY_t, \bzeta_t) \iso (X, \xi)$.
\end{proposition}
\begin{proof}
We prove the proposition in three steps.\\ \textbf{Step 1:} Set $(\Lambda, \lambda) = (\H^2(Y, \IZ), c_1(\zeta))$ and $L := \lambda^\perp \subset \Lambda$. We use the notations and constructions in \ref{3.1.3} for $L$. Choose a compact open subgroup $\sK \subset \CSpin(L \tensor \IA_f) = G(\IA_f)$ of the form $\sK_{0, p} \sK^p$ for some $\sK^p \subset G(\IA_f^p)$ such that $\sK \subset \sK_0$ and $\shS^\ad_\sK(L)$ is a scheme. Recall that the image $\sK^{\ad} \subset G^\ad(\IA_f)$ preserves the lattice $L \tensor \what{\IZ}$ and acts trivially on $\disc(L)$. Therefore, $\sK$ acts on $(\Lambda \tensor \what{\IZ}, \lambda \tensor 1)$ (see \ref{2.1.7}). 

Apply Artin's approximation theorem to $\Def(X; \xi)$ to obtain a primitively polarized $\HKn$-scheme $(\sX \to S, \bxi)$ as in \ref{period+}. Recall that since $S$ is everywehre a universal deformation, by \ref{flatloc}(b), $S$ is flat over $W$ and the special fiber of $S$ is reduced. We may moreover assume that $S$ is affine and connected. By assumption, there exists in $\bar{K}_0$ a finite flat extension $V$ of $W$, and a deformation $(X_V, \xi_V)$ of $(X, \xi)$ over $V$, such that there exists a parallel transport operator $\eta^\dagger : (\Lambda, \lambda) \sto (\H^2(X_V(\IC), \IZ), c_1(\xi_V(\IC))$. Let $b$ (resp. $b_V$) be the $k$-point (resp. $V$-point) on $S$ whose fiber is $(X, \xi)$ (resp. $(X_V, \xi_V)$). Set $b_\IC := b_V \tensor \IC$, so that $(X_V \tensor_\tau \IC, \xi_V \tensor_\tau \IC) = (\sX_{b_\IC}, \bxi_{b_\IC})$. 

Recall that in \ref{period+} we constructed \'etale covers $\wt{S}^\#$ and $\wt{S}_{\sK^p}$ of $S$. In particular, the restriction of $\sX$ to $\wt{S}^\#$ or $\wt{S}_{\sK^p}$ comes equipped with a canonical aritificial orientation $\epsilon^p$ given by \ref{prop: Extend epsilon}. Moreover, by the modular interpretation of $\wt{S}_{\sK^p}$ relative to $S$, $\eta^\dagger$ induces a unique lift of $b_\IC$ to $\wt{S}_{\sK^p}$. We replace $S$ by the connected component of $\wt{S}_{\sK^p}$ containing this lift and replace $b, b_V, b_\IC$ and $\sX$ accordingly. Now for the family $\sX$ over $S$, we have a trivialization $\epsilon_2 : \underline{\det(L \tensor \IQ_2)}_S \to \det(\bP^2_{2})$, which extends canonically to an artificial orientation $\epsilon^p : \underline{\det(L^p)}_S \to \det(\bP^2_{\what{\IZ}^p})$. Moreover, we have a compatible $\sK^p$-level structure $[\eta^p]$. As in \ref{period+}, we obtain a period morphism $\rho : S \to \shS^\ad_\sK(L)$.  \\\\
\textbf{Step 2:} We briefly recall the description of $\rho$ on $\IC$-points. For each point $s \in S \tensor_\tau \IC$, we choose an \'etale path $\gamma_{\et, s}$ connecting $s$ and $b_\IC$, which induces an isomorphism $\gamma_\et^* : (\H^2_\et(\sX_s, \what{\IZ}^p), c_1(\bxi_s)) \sto (\H^2_\et(\sX_{b_\IC}, \what{\IZ}^p), c_1(\bxi_{b_\IC}))$. Choose an isomorphism $\mathring{\eta}_s : (\H^2(\sX_s, \IZ), c_1(\bxi_s)) \sto (\H^2_\et(\sX_{b_\IC}, \IZ), c_1(\bxi_{b_\IC}))$ such that $\mathring{\eta} \tensor \IZ_2$ sends $\epsilon_{2}|_{s}$ to $\epsilon_2|_{b_\IC}$. Set $\eta_{p, s} := \mathring{\eta} \tensor \IZ_p$ and $\eta_s := (\eta_{p, s} \oplus \gamma_{\et, s}^*) \circ (\eta^\dagger \tensor \what{\IZ})$. Then under the universal property \ref{prop: universal property of Shimura}, $\rho(s)$ is the $\IC$-point on $\Sh_{\sK}^\ad(L)_\IC$ defined by the tuple 
\begin{equation}
\label{eqn: period tuple}
    (\P^2(\sX_{s}, \IQ), \epsilon_2|_{s}, \eta_s \sK^{\ad}).
\end{equation}
Note that $(\gamma_{\et, s}^* \circ (\eta^\dagger \tensor \what{\IZ}^p)) \sK^p = [\eta^p]|_s$, and the orbit $\eta_{p, s} \sK_p^\ad$ is independent of the choice of $\mathring{\eta}_s$. Therefore, the orbit $\eta \sK^{\ad}$ is independent of the choices of $\gamma_\et$ and $\mathring{\eta}$. In particular, by \ref{Sasha}, $S_\IC$ is connected, so we may choose $\gamma_\et$ to be one that is induced by a topological path $[0,1] \to S_\IC$ connecting $s$ and $b_\IC$. \\\\
\textbf{Step 3:} For each $\sT = (X, \xi) \in \sS(Y, \zeta, k, \tau)$, take $\sX, S, b, b_V, b_\IC, \eta^\dagger, \rho, \epsilon^p, \eta^p$ by performing the constructions in Step 1. Let $\shT := \coprod_{\sT} S$, $\shX := \coprod_{\sT} \sX$, and $\rho_\shT : \shT \to \shS^\ad_\sK(L)_W$. By abuse of notation, denote the polarization on $\shX$ still by $\bxi$. Now suppose we have $k$-points $t_1, t_2$ on $\shT$ and a point $t \in \shS^\ad_\sK(L)_W$ such that $t = \rho_\shT(t_1) = \rho_\shT(t_2)$. Since $\shS^\ad_\sK(L)_W$ is flat over $W$, we may always find a finite extension $F$ of $K_0$ such that $t$ generizes to an $F$-point $\wt{t}$ via an $\sO_F$-point. Since $\rho_\shT$ is \'etale, we can find the corresponding $F$-point $\wt{t}_i$ on $\shT$ which specializes to $t_i$ such that $\rho_\shT(\wt{t}_i) = \wt{t}$. Set $s_i = \wt{t}_i(\IC)$. Since $\rho_\shT(s_1) = \rho_\shT(s_2)$, by our description in Step 2, there exist isomorphisms of pointed $\IZ$-lattices 
$$ \psi_i, \gamma_i : (\H^2(\shX_{s_i}, \IZ), c_1(\bxi_{s_i})) \sto (\Lambda, \lambda) $$
such that $\psi_1^{-1} \circ \psi_2$ preserves the Hodge structures, $\gamma_i$'s are parallel transport operators, and $(\psi_2 \circ \psi^{-1}_1) \tensor \what{\IZ}^p$ and $(\gamma_2 \circ \gamma_1^{-1}) \tensor \what{\IZ}^p$ lie in the same $\sK^{\ad, p}$-orbit. Therefore, we have $(\psi_2 \circ \psi^{-1}_1) \circ (\gamma_2 \circ \gamma_1^{-1})^{-1}$ lies in $\O(\Lambda, \lambda) \cap \sK^{\ad, p}$, and hence in $\Mon^2(Y)$ by \ref{spinorientation} and \ref{Markman}. This shows that $\psi_2 \circ \psi_1$ is a parallel transport operator. Verbitsky's theorem tells us that $\psi_2 \circ \psi_1$ is induced by an isomorphism $(\shX_{s_1}, \bxi_{s_1}) \iso (\shX_{s_2}, \bxi_{s_2})$. By \cite[Thm~2]{MM}, we obtain an isomorphism $(\shX_{t_1}, \bxi_{t_1}) \iso (\shX_{t_2}, \bxi_{t_2})$. 

Finally, note that $\mathrm{im}(\rho_\shT)$ is an open subscheme of $\shS^\ad_\sK(L)_W$ and is in particular quasi-compact. Therefore, there exists a finite subset $B$ of $\sS(Y, \zeta, k, \tau)$ such that setting $T := \coprod_{\sT \in B} S$, we have $\mathrm{im}(\rho_\shT) = \mathrm{im}(\rho_{\shT}|_{T})$. This implies that $T$ and $(\sY, \bzeta) := (\shX, \bxi)|_T$ satisfy the sought after properties. 
\end{proof}

\noindent \textit{Proof of \ref{bounded}}. By our spreading out lemmas \ref{defstrengthen} and \ref{charpspread}, it suffices to show that the bounds $m, N$ exist for $\kappa = \IC$ or $\bar{\IF}_p$. The $\kappa = \IC$ case follows from a theorem of K\'ollar and Matsusaka \cite{KM} (see also \cite[Thm~3.3]{Charles2}). The $\kappa = \bar{\IF}_p$ case follows from the $\kappa = \IC$ case, \ref{charpspread}, and \ref{prop: bounded}. The statement about the possible Hilbert polynomials is clear. \qed

\subsubsection{} Next, we study automorphisms of $\HKn$-type varieties.

\begin{proposition}
\label{liftauto}
Assume $p > n +1, 3$ and $k = \bar{\IF}_p$. Let $X/k$ be an excellent reduction of a $\HKn$-type variety such that $\H^2_\cris(X/W)$ is not supersingular. For any automorphism $f \in \Aut(X)$, there exists a finite flat extension $V$ of $W$, a deformation $X_V$ of $X$ over $V$ such that $f$ lifts to an automorphism of $X_V$ and the specialization map $\Pic(X_V) \to \Pic(X)$ is an isomorphism. In particular, $X_V$ is algebraizable.
\end{proposition}
\begin{proof}
Given \ref{locTor}, the above poposition follows from the proofs of \cite[Prop.~4.5, Lem.~4.9]{Yang}. We sketch the main steps so that the reader can fill in the details by looking at the proofs in \textit{loc. cit.}

First, note that $H := \H^2_\cris(X/W)$ admits a slope decomposition $H_{< 1} \oplus H_{= 1} \oplus H_{> 1}$. By Dieudonn\'e theory, the F-crystal $H_{< 1}$ is isomorphic to $\ID(G^*)$ for some one-dimensional formal Lie group $G$ of finite height over $k$. Fix this isomorphism. The pairing $H_{< 1} \oplus H_{> 1} \to W(-2)$ determines an isomorphism $H_{> 1} \iso \ID(G)(-1)$, under which the pairing is a Tate twist of the canonical pairing $\ID(G^*) \times \ID(G) \to W(-1)$. Let $V$ be any finite flat extension of $W$ and set $R := V/(p)$. A deformation $G_R$ of $G$ over $R$ determines a deformation $\bH$ of the K3 crystal $\H^2_\cris(X/W)$ over $R$, and a further deformation $G_V$ of $G_R$ over $V$ determines an isotropic line $\Fil_V \subset \bH_V$ which lifts the line $\Fil \subset \bH_R$ given by the K3 crystal structure.  

Now, by \ref{locTor}, to lift the automorphism $f \in \Aut(X)$, it suffices to find a pair $(\bH, \Fil_V)$, where $\bH$ is a deformation of the K3 crystal $\H^2_\cris(X/W)$ over $R$, and $\Fil_V \subset \bH_V$ is an isotropic line lifting $\Fil^2 \H^2_\dR(X/k) = [\Fil^1 \ID(G)_k](-1)$. We reduce the problem to finding a deformation $G_V$ of $V$ over some $V$ which also lifts the automorphism of $G$ induced by $f$. This we can always do by Lubin-Tate theory, for which we use the assumption that $k = \bar{\IF}_p$. Using the arguments in \cite[Prop.~4.5]{Yang}, one checks that the resulting lifting $X_V$ of $X$ carries liftings of the entire $\Pic(X)$. 
\end{proof}

\begin{remark}
If $X$ is a K3 surface, then by \cite[Thm~3.12]{NO} and its proof, there is a natural isomorphism $\ID(\what{\mathrm{Br}}_X^*) \iso H_{< 1}$. Moreover, $H_{=1}$ is interpreted as $\ID(\H^2_\fl(X, \mu_{p^\infty})^*)$. In the above proof though all that we needed is an abstract deformation result about K3 crystals and these geometric interpretations are unnecessary. It is very likely that these geometric interpretations hold for general $X$'s as above. However, we point out that one place where arguments in \cite{NO} do not directly generalize is that one cannot argue that $\H^2_\fl(X, \mu_{p^\infty})$ is $p$-divisible by flat duality (see the proof of \cite[Lem.~3.1]{NO}). 
\end{remark}

\begin{proposition}
\label{killauto}
Assume $n \ge 2$ and $p > n + 1$. Let $k$ be an algebraically closed field of characteristic $p$. Let $X$ be a $\HKn$-type variety over $k$. 
\begin{enumerate}[label=\upshape{(\alph*)}]
    \item The natural map $\Aut(X) \to \O(\H^2_\cris(X/W))$ is injective. 
    \item The natural map $\Aut(X) \to \O(\H^2_\et(X, \IZ_\ell))$ is injective for every prime $\ell \neq p$. 
    \item Suppose $f$ is an automorphism of $X$ which fixes a polarization. If for some number $m$ with $m \ge 3$ and $p \nmid m$, $f$ acts trivially on $\P^2_\et(X, \what{\IZ}^p)$ modulo $m$, then $f$ is the identity. 
\end{enumerate}
\end{proposition}
\begin{proof}
For (a), one can apply the same arguments in the proof of \cite[Cor.~2.5]{Ogus}: One can always choose a lift $X_W$ of $X$ over $W$ which carries the lift of a primitive polarization, by \ref{flatloc} and \ref{tateImprove}. If $f \in \Aut(X)$ acts trivially on $\H^2_\cris(X/W)$, it automatically preserves the filtration on $\H^2_\cris(X/W)$ induced by $X_W$. By the local Torelli theorem, $f$ lifts to $X_W$ and the lift of $f$ acts trivially on $\H^2_\dR(X_W/W)$. By Theorem~\ref{transmit}, for any embedding $W \into \IC$, $X_W \tensor \IC$ is of $\HKn$-type, but we know that hyperk\"ahler manifolds of $\HKn$-type are cohomologically rigid, so that the natural map $\Aut(X_\IC) \to \O(\H^2(X_\IC, \IZ))$ is injective.

(b) If $X$ is supersingular, then $\H^2_\cris(X/W)[1/p]$ and $\H^2_\et(X, \what{\IZ}^p)$ are both spanned by algebraic classes. Since automorphisms have to preserve $\Pic(X)$, (b) reduces to (a). If $X$ is not supersingular, then we first use a standard spreading out argument to reduce (b) to the case $k = \bar{\IF}_p$. Let $X_V$ be the lift of $X$ given by applying \ref{liftauto} to $f$. For any embedding $V \into \IC$, the lift of $f$ acts trivially on $\H^2_\et(X_V \tensor \IC, \IZ_\ell)$, so it acts trivially on $\H^2(X_V \tensor \IC, \IZ)$ and $\H^2_\dR(X_V/V) \tensor K$. By the Berthelot-Ogus isomorphism $\H^2_\cris(X/W) \tensor K \iso \H^2_\dR(X_V/V) \tensor K$, $f$ must act trivially on $\H^2_\cris(X/W)$, so we reduce to (a). 

(c) By a general result of Matsusaka \cite[Thm~6(iv)]{Mat}, the automorphism group $\Aut(X, \xi)$ is finite, for any polarization $\xi$ on $X$. Therefore, $f$ has finite order. Next, let $\P^2(X)$ be either $\P^2(X)= \P^2_\et(X, \IZ_\ell)$ for some $\ell \neq p$ or $\P^2_\cris(X/W)[1/p]$. We claim that the characteristic polynomial $\det(tI - f^*|_{\P^2(X)})$ has coefficients in $\IZ$ and is independent of the cohomology theory chosen for $\P^2(X)$. Again, in the supersingular case, this follows from \ref{tateImprove}. In the finite height case, we first reduce to $k = \bar{\IF}_p$ using standard spreading out argument, apply (b) and use that fact that automorphisms of complex manifolds preserve the Betti cohomology with integral coefficients. Now let $\alpha$ be an eigenvalue for $f^* |_{\P^2(X)}$. One checks that $\alpha$ is of the form $1 + mc$ for some algebraic integer $c$. By the lemma in \cite[208]{Mumford}, $\alpha = 1$ when $m \ge 3$. Since $f$ has finite order, $f^*|_{\P^2(X)}$ has to be the identity. Now (c) follows from (a) or (b). 
\end{proof}

\begin{remark}
\label{rmkkillauto}
Rizov proved the above results for K3 surfaces (\cite[Prop.~3.4.2, Cor.~3.4.5]{Rizov1}) by taking advantage of the fact that a K3 surface automatically admits a Chow-K\"unneth decomposition (see \cite[Lem.~3.4.4]{Rizov1}). Rizov's proof would have generalized directly were the Chow-K\"unneth decomposition known for $\HKn$-type varieties in general. Unfortunately, currently we only know this when $X$ is a Hilbert scheme of $n$-points on a K3 surface (\cite[Thm~1]{Vial}). We also remark that even when $n = 1$ (c) is not true if $X$ is supersingular (\cite[Thm~1.4]{EK}). Of course, in the supersingular case, our proof cannot proceed, but it is interesting that in \textit{loc. cit.} counterexamples are provided by considering complex dynamics. Finally, we caution the reader that the injectivity of $\Aut(X_\IC) \to \O(\H^2(X_\IC, \IZ))$ is not some general fact about hyperk\"ahler manifolds: It is not true for manifolds of generalized Kummer type. 
\end{remark}

Now consider the moduli functor $\Mod_{n, d}$ which sends each scheme $S$ over $\IZ_{(p)}$ to the groupoid of primitively polarized $\HKn$-spaces of degree $d$ over $S$. 
\begin{proposition}
$\Mod_{n, d}$ is reprsentable by a Deligne-Mumford stack of finite type over $\IZ_{(p)}$. 
\end{proposition}
\begin{proof}
The $n = 1$ case, i.e., the case for K3 surfaces, has already been treated by Rizov (\cite[Thm~4.3.3]{Rizov1}). Given our preparations, there is no obstruction to generalizing the proofs in \textit{loc. cit.} We sketch the main steps for the adaptation. First, let $m$, $N$ and $\sQ_{n, d, m}$ be as in \ref{bounded}. Again we assume $p \nmid m$. Consider the disjoint union $\coprod_{P \in \sQ_{n, d, m}} \Hilb^{+, m}_P$ and let the universal family over it be denoted by $\shZ$. We call a connected component $\sC$ of this union admissible if $\sC$ contains a geometric point $s$ such that $\shZ_s$ is of $\HKn$-type. One checks by \ref{defstrengthen}, \ref{Hilbconnected}, and \ref{charpspread} that $\sC$ is admissible if and only if for every geometric point $s$, $\shZ_s$ is of $\HKn$-type. In fact, by the density result of Mongardi and Pazienza \cite[Cor.~1.2]{Density}, to check that $\sC$ is admissible, it suffices to check that it contains a $\IC$-point $s$ such that $\shZ_s$ is birational to the Hilbert scheme of $n$ points on some K3 surface. Let $H_{m, d}$ be the union admissible components. Then we obtain a natural morphism $H_{m, d} \to \Mod_{n, d}$, which can easily be checked to be smooth and surjective. 

Next, one shows that given a $\IZ_{(p)}$-scheme $S$ and two objects $\sX, \sY$ of $\Mod_{n, d}$ over $S$, the functor which parametrizes isomorphisms $\mathrm{Isom}_S(\sX, \sY)$ is representable by a separated unramified scheme over $S$. The proof \cite[Lem.~4.3.8]{Rizov1} proceeds without change. The key is that a $\HKn$-type varieties has no global vector fields, and hence no nontrivial infinitesimal automorphisms (cf. \cite[Thm~3.3.1]{Rizov1}). Now we may conclude because $\Mod_{n, d}$ has a representable, separated and unramified diagonal and a smooth cover by $H_{m, d}$. \end{proof}

\subsubsection{}\label{Lambdadecomp} Let $\Lambda$ denote the $\HKn$-lattice $\Lambda_n$. It is easy to see that there is decomposition 
$$ \Mod_{n, d} = \coprod_{\{(\Lambda, \lambda)\}/\iso} \Mod_{n, d}^{(\Lambda, \lambda)} $$
where $\lambda$ runs through primitive vectors of $\Lambda$ with $\lambda_n (\< \lambda, \lambda \>)^n = d$ and $\Mod_{n, d}^{(\Lambda, \lambda)}$ is a union of connected components of $\Mod_{n, d}$ such that for every geometric point on $\Mod_{n, d}^{(\Lambda, \lambda)}$, the map $\PL$ (see \ref{defPLLem}) takes the fiber to $(\Lambda, \lambda)$. Since $\Mod_{n,d}$ is of finte type, for only finitely many isomorphism classes of $(\Lambda, \lambda)$, $\Mod_{n, d}^{(\Lambda, \lambda)} \neq \emptyset$. 

We restrict our attention to a single $\Mod_{n, d}^{(\Lambda, \lambda)}$. Let the universal family over it be denoted by $(\shX, \bxi)$. As in \ref{period+}, let $\bH^2_*$ and $\bP^2_*$ ($* = B, \ell, \dR, \cris$) be the sheaves given by the cohomology of the universal family. Let $L = \lambda^\perp \subset \Lambda$. Construct a finite \'etale cover $\wt{\Mod}_{n, d}^{(\Lambda, \lambda), \#}$ of $\Mod_{n, d}^{(\Lambda, \lambda)}$, so that there are universal trivializations $\epsilon : \underline{\det(L_2)} \stackrel{\sim}{\to} \det(\bP^2_2)$ and $\Delta : \underline{\disc(L)} \stackrel{\sim}{\to} \disc(\bP^2_{\what{\IZ}^p})$. As in \ref{period+}, we can construct a period morphism $\wt{\Mod}_{n, d, \IQ}^{(\Lambda, \lambda), \#} \to \Sh^\ad_{K_0}(L)$, using which we can extend $\epsilon$ uniquely to an artificial orientation $\epsilon^p$ over $\Mod_{n, d}^{(\Lambda, \lambda)}$ which is characterized by the property that for any $s \in \wt{\Mod}_{n, d, \IQ}^{(\Lambda, \lambda)}(\IC)$, $\epsilon^p|_s$ is obtained by tensoring a (necessarily unique) isometry $\det(L) \to \det(\P^2(\shX_s, \IZ))$ with $\what{\IZ}^p$. For any compact open subgroup $\sK^p$ of $\sK_0^p$, we can now speak of the finite \'etale cover $\wt{\Mod}_{n, d, \sK^p}^{(\Lambda, \lambda)}$ of $\wt{\Mod}_{n, d}^{(\Lambda, \lambda)}$ which parametrizes $\sK^p$-level structures compatible with $\epsilon^p$ and $\Delta$. 

\begin{theorem}
\label{fineMod}
For $\sK^p$ small enough, $\wt{\Mod}_{n, d, \sK^p}^{(\Lambda, \lambda)}$ is representable by a quasi-projective scheme over $\IZ_{(p)}$. 
\end{theorem}
\begin{proof}
Let $S$ be any $\IZ_{(p)}$-scheme. Let $(\sX \to S, \xi)$ be any polarized $\HKn$-space over a $\IZ_{(p)}$-scheme $S$. Since $\Aut_S(\sX, \xi)$ is representable and unramified over $S$, if $f \in \Aut_S(\sX, \xi)$ acts trivially on each geometric fiber, $f$ must be trivial. Therefore, when $\sK^p$ is small enough, \ref{killauto}(c) ensures that objects in $\wt{\Mod}_{n, d, \sK^p}^{(\Lambda, \lambda)}(S)$ have no nontrivial automorphisms, so that $\wt{\Mod}_{n, d, \sK^p}^{(\Lambda, \lambda)}$ is representable by an algebraic space of finite type. 

When $\sK^p$ is small enough, $\shS^\ad_\sK(L)$ is a quasi-projective scheme, where $\sK := \sK_{0,p} \sK^p$. Recall that there is an \'etale morphism $\rho : \wt{\Mod}_{n, d, \sK^p}^{(\Lambda, \lambda)} \to \shS^\ad_\sK(L)$. As $\wt{\Mod}_{n, d, \sK^p}^{(\Lambda, \lambda)}$ is an algebraic space, $\rho$ is representable. Reference \cite[Prop.~16.5]{LMB} tells us that $\wt{\Mod}_{n, d, \sK^p}^{(\Lambda, \lambda)}$ embeds as an open sub-algebraic space into a finite $\shS^\ad_\sK(L)$-scheme, so $\wt{\Mod}_{n, d, \sK^p}^{(\Lambda, \lambda)}$ is a quasi-projective scheme. 
\end{proof}

\subsection{Crystalline Torelli}
\label{4.3}

Let $k$ be an algebraically closed field of characteristic $p > n + 1$.

\begin{definition}
\label{defetPT}
Let $X, X'$ be two $\HKn$-type varieties over $k$. An isomorphism $\psi^p : \H^2_\et(X, \what{\IZ}^p) \stackrel{\sim}{\to} \H^2_\et(X', \what{\IZ}^p)$ is said to be an \'etale parallel transport operator if there exists a connected scheme $S$ over $W$ with $k$-points $b, b'$ and a $\HKn$-space $\sX \to S$ such that under some isomorphism $\sX_b \iso X$, $\sX_{b'} \iso X'$, $\psi^p$ is induced by an \'etale path on $S$ from $b$ to $b'$. If $\sX \to S$ can be taken to be (primitively) polarizable, then we say in addition that $\psi^p$ is (primitively) polarizable. 
\end{definition}

By an \'etale path we mean an isomorphism of fiber functors from the category of finite \'etale covers to the category of sets defined by $b$ and $b'$. For a reference, see \cite[Tag~03VD]{stacks-project}.

\subsubsection{} Let $X$ and $X'$ be two supersingular $\HKn$-type varieties over $k$ and $\psi : \NS(X)_\IQ \stackrel{\sim}{\to} \NS(X')_\IQ$ be a rational isometry. By \ref{tateImprove}, $\psi$ induces via the Chern class maps isomorphisms $\psi_\et: \H^2_\et(X, \IA_f^p) \stackrel{\sim}{\to} \H^2_\et(X', \IA^p_f)$ and $\psi_\cris: \H^2_\cris(X/W)\tensor K_0 \stackrel{\sim}{\to} \H^2_\cris(X'/W) \tensor K_0$. We remark that these maps can alternatively be viewed as induced by a correspondence $Z_\psi$.
\begin{lemma}
\label{lem: induced by an isog}
$\psi$ is induced by a correspondence $Z_\psi : X' \rightsquigarrow X$.
\end{lemma}
\begin{proof}
Indeed, let $e_1, \cdots, e_b$ (resp. $e_1', \cdots, e_b'$) be a basis of $\NS(X)_\IQ$ (resp. $\NS(X')_\IQ$). Assume that $e_i$ is orthogonal to $e_j$ for $i \neq j$ under the Beauville-Bogomolov form. Then the dual $e_i^* \in \NS(X)_\IQ^*$ of $e_i$ can be realized by intersecting with a rational multiple of $e_i^{n - 1}$. The correspondence we are seeking is given by a rational linear combination of $(e_i)^{n - 1} \cdot e_i'$'s (cf. \cite[Lem.~5.1]{Yang}). 
\end{proof}

\begin{lemma}
\label{lem: induce lifting}
Suppose that $\psi_\cris$ restricts to an integral isomorphism $\H^2_\cris(X/W) \stackrel{\sim}{\to} \H^2_\cris(X'/W)$, which we still denote by $\psi_\cris$. Let $V$ be a finite flat extension of $W$ and $X_V$ be a formal lift of $X$ over $V$. Assume $V = W$ or $p \ge 5$. Then $X_V$ induces a formal lift $X'_V$ of $X'$ over $V$ with the following properties: 
\begin{enumerate}[label=\upshape{(\alph*)}]
    \item For $R:= V/(p)$, $X_R := X_V \tensor R$ and $X'_R := X'_V \tensor R$, $\psi_\cris$ lifts to an isomorphism $\Psi_\cris : \H^2_\cris(X_R) \stackrel{\sim}{\to} \H^2_\cris(X'_R)$ of K3 crystals over $R$. Moreover, $(\Psi)_V$ respects the Hodge filtrations of $X_V$ and $X'_V$ via the crystalline-de Rham comparison isomorphism. 
    \item If $\xi, \xi'$ are line bundles on $X, X'$ such that $\psi(\xi) = \xi'$ and $\xi$ lifts to $X_V$, then $\xi'$ lifts to $X'_V$.
    \item Assume $p \ge 5$. Let $K$ be $V[1/p]$ and denote by $X_K$, $X_K'$ the rigid analytic generic fibers of $X_V$ and $X'_V$. The isomorphism $\psi_p : \H^2_\et(X_{\bar{K}}, \IQ_p) \sto \H^2_\et(X'_{\bar{K}}, \IQ_p)$ induced by $\psi_\cris \tensor K$ via the $p$-adic comparison isomorphisms\footnote{We use the $p$-adic comparison isomorphism given by \cite[Thm~1.1(i)]{BMS}.} restricts to an isomorphism of the underlying $\IZ_p$-lattices. 
\end{enumerate}
\end{lemma}
\begin{proof}
If $V = W$, which matters only when $n = 1$, then (a) follows from the locally Torelli theorem for K3's and (b) follows directly from \cite[Prop.~1.12]{Ogus}. Now we only consider the case when $p \ge 5$ but $V$ can be arbitrarily ramified. Then (a) is a direct consequence of \ref{locTor}.

(b) Note that $R \iso k[t]/t^e$, where $e$ is the ramification degree of $V$. For each $i \le e$, set $R_i = R/ (t^i)$. Let $\xi_R \in \Pic(X_R)$ be a lifting of $\xi$ and $(X_i, \xi_i)$ be the restriction of $(X_R, \xi_R)$ over $R_i$. Similarly, set $X'_i := X'_R \tensor R_i$. Suppose by induction that $\xi'$ has been lifted to $\xi_i' \in \Pic(X'_i)$ for some $i < e$. By \cite[Prop.~1.12]{Ogus}, showing that $\xi_i'$ deforms to $X_{i + 1}'$ amounts to showing that $c_{1, \cris}(\xi_i')_{R_{i + 1}} \in \Fil^1 \H^2_\dR(X_{i + 1}'/ R_{i + 1})$, where we identify $\H^2_\cris(X'_{i})_{R_{i + 1}}$ with $\H^2_\dR(X'_{i + 1}/ R_{i + 1})$. We already know that the corresponding condition is satisified by $\xi_i$ and $X_{i + 1}$, so it suffices to show that $c_{1, \cris}(\xi_i')_{R_{i + 1}} = (\Psi|_{R_{i}})(c_{1, \cris}(\xi_i)_{R_{i + 1}})$. 

Choose a finite flat extension $V$ of $W$ with ramification degree $i + 1$ and set $K:= V[1/p]$. Let $\pi$ be a uniformizer of $V$ and fix an isomorphism $V/(p) \iso R_{i + 1}$ which sends $\pi$ to $t$. Note that this isomorphism descends to $V/(\pi^i) \iso R_i$ and the ideal $(\pi^i) \subset V$ has a natural PD structure (see for example \cite[Lem.~3.9]{BO}). We get a Berthelot-Ogus isomorphism $\H^2_\cris(X/W) \tensor_W K \iso \H^2_\cris(X_i)_V \tensor_V K$, which sends $c_{1, \cris}(\xi) \tensor 1$ to $c_{1, \cris}(\xi_i)_{V} \tensor 1$ by \cite[Cor.~3.6]{BO}. There are similar isomorphisms for $X'_i$ which fit into 
a diagram
\begin{center}
    \begin{tikzcd}
    \H^2_\cris(X)_W \tensor K \arrow{r}{\sim} \arrow{d} & \H^2_\cris(X_i)_V \tensor_V K \arrow{d} \\
    \H^2_\cris(X')_W \tensor K  \arrow{r}{\sim} & \H^2_\cris(X'_i)_V \tensor_V K 
    \end{tikzcd}
\end{center}
where the vertical arrows are isomorphisms induced by $\Psi$. Since $\H^2_\cris(X_i')_V$ is $p$-torsion free, we know that $\Psi|_{R_i}$ sends $c_{1, \cris}(\xi_i)_{V}$ to $c_{1, \cris}(\xi_i')_{V}$. We get what we need to taking mod $p$ reductions. 

(c) We first explain how $\psi_p$ is constructed. Let $\Rep^\cris_K$ denote the category of crystalline $\Gal_K$-representations in $\IQ_p$-coefficients and $\MF^\varphi_K$ denote the category of filtered $(\varphi, N)$-modules over $K$ with $N = 0$. Fontained defined a functor $D_\cris : \Rep^\cris_K \to \MF^\varphi_K$ given by $Q \mapsto (Q \tensor B_\cris)^{\Gal_K}$ and on the essential image of $D_\cris$ we have a quasi-inverse $V_\cris$ defined by $D \mapsto \Fil^0 (D \tensor B_\cris)^{\varphi = 1}$, such that $V_\cris(D_\cris(Q))$ is canonically identified with $Q$ for $Q \in \Rep^\cris_K$. Now, the F-isocrystal $\H^2_\cris(X/W)[1/p]$, equipped with the Hodge filtration on $\H^2_\dR(X_K/K) \iso \H^2_\cris(X/W) \tensor K$, defines an object of $\MF^\varphi_K$. Under the $p$-adic comparison isomorphisms, we obtain $\psi_p$ by applying $V_\cris$ to $\psi_\cris$. 

Let $S$ denote Breuil's $S$-ring (see \cite[1200]{BLiu}). $\H^2_\cris(X/W)[1/p]$ can be functorially promoted to a $(\varphi, N)$-module $\shD$ over $S[1/p]$ (see \cite[1215]{BLiu}). By Thm~5.4 of \textit{loc. cit.}, $\H^2_\cris(X_R)_S$ defines a strongly divisible $S$-lattice in $\shD$,\footnote{Technically, $\H^2_\cris(X_R)_S$ denotes the restriction of the crystal $\H^2_\cris(X_R)$ to $S$, but it is indeed the same thing as $\H^2_\cris(X_R/S)$. Detail oriented readers can check this using \cite[Tag~07MJ]{stacks-project}.} from which one can functorially recover the $\IZ_p$-lattice $\H^2_\et(X_{\bar{K}}, \IZ_p)$. Now, if we define $\shD'$ for $X'_K$ in the say way, we can promote $\psi_\cris \tensor K$ to an isomorphism $\shD \sto \shD'$. By the existence of $\Psi_\cris$, this isomorphism sends $\H^2_\cris(X_R)_S$ isomorphically onto $\H^2_\cris(X'_R)_S$. Therefore, $\psi_p$ restricts to an isomorphism of the underlying $\IZ_p$-lattices. 
\end{proof}

\begin{lemma}
\label{AHLift}
Let $X, X', \psi_\cris, X_V, X_V'$ be as in {\upshape{\ref{lem: induce lifting}}}. Moreover, assume that
$X_V$ and $X_V'$ carry liftings $\xi_V, \xi'_V$ of some polarizations $\xi$ and $\xi'$ on $X$ and $X'$ respectively such that $p \nmid \xi^{2n}  = (\xi')^{2n}$ and $\psi(\xi) = \xi'$. Let $K$ denote $V[1/p]$ and $\psi_{\dR, K}$ denote the isomorphism $\H^2_\dR(X_K/K) \sto \H^2_\dR(X'_K/K)$ induced by $\psi_\cris$ via the Berthelot-Ogus isomorphism. 

Then given any embedding $K \into \IC$, there exists an isomorphism $\psi_{B} : \H^2(X_K(\IC), \IZ_{(p)}) \sto \H^2(X'_K(\IC), \IZ_{(p)})$ such that $\psi_B \tensor \IC = \psi_{\dR, K}(\IC)$ and $\psi_B \tensor \IA^p_f$ agrees with $\psi_{\et}$ via the Artin comparison and smooth proper base change isomorphisms. 
\end{lemma}
\begin{proof}
One easily deduces from \ref{defPLLem} and \ref{tateImprove} that $\PL(X, \xi) = \PL(X', \xi')$. Let this pointed lattice be $(\Lambda, \lambda)$. Set $L = \lambda^\perp$. Note that by our assumptions $L_p$ is self-dual. We endow $(X, \xi)$ with trivializations $\Delta : \disc(L^p) \stackrel{\sim}{\to} \disc(\P^2_\et(X, \what{\IZ}^p))$ and $\epsilon : \det(L_2) \stackrel{\sim}{\to} \det(\P^2_\et(X, \IZ_2))$. Similarly, endow $(X', \xi')$ with $(\Delta', \epsilon')$. We choose $\epsilon'$ to be the one induced by $\epsilon$ via $\psi \tensor \IZ_2$. With these choices, $(X, \xi)$ and $(X', \xi')$ give rise to $k$-points $t$ and $t'$ on $\wt{\Mod}^{(\Lambda, \lambda), \#}_{n, d}$. The liftings $(X_V, \xi_V)$ and $(X_V', \xi_V')$ correspond to some $V$-valued points $t_V, t_V'$ which specialize to $t, t'$. Recall that in \ref{3.1.3} we explained that for a chosen generator $\delta \in \det(L)$, there are global sections $\bd_\ell$ on $\det(\bL_{\ell})$ over $\shS_{\sK_0}(L)$ for every $\ell \neq p$. 

Consider the period morphism $\rho : \wt{\Mod}^{(\Lambda, \lambda), \#}_{n, d} \to \shS^\ad_{\sK_0}(L)$. Let $s, s'$ be $k$-points on $\shS_{\sK_0}(L)$ which lift $\rho(t), \rho(t')$. We will freely make use of isomorphisms in \ref{compMotive}. Consider the universal abelian scheme $\shA$ over $\shS_{\sK_0}(L)$. By \ref{ssIsog}, there exists a CSpin-isogeny $f : \shA_s \to \shA_{s'}$. Let $\< \xi \>^\perp, \< \xi' \>^\perp$ be the orthogonal complements of $\xi, \xi'$ in $\NS(X), \NS(X')$ and let $\psi^\perp : \< \xi \>^\perp \stackrel{\sim}{\to} \< \xi' \>^\perp$ be the restriction of $\psi$. By \ref{compMotive}(e), $\psi^\perp$ can be identified with an isomorphism $\LEnd(\shA_s) \stackrel{\sim}{\to} \LEnd(\shA_{s'})$. $f$ also induces by conjugation an isomorphism $f^\conj : \LEnd(\shA_s)_\IQ \stackrel{\sim}{\to} \LEnd(\shA_{s'})_\IQ$. The composition $ (\psi^\perp \tensor \IQ)^{-1} \circ f^\conj$ is an element of $\O(\LEnd(\shA_s)_\IQ)$. By our assumption on $\epsilon$ and $\epsilon'$, the isomorphism $\bL_{\what{\IZ}^p, s} \to \bL_{\what{\IZ}^p, s'}$ induced by $\psi^\perp$ respects the orientation tensors $\bd_{2, s}$ and $\bd_{2, s'}$. Since CSpin-isogenies also preserve orientation tensors, $(\psi^\perp \tensor \IQ)^{-1} \circ f^\conj$ lies in $\SO(\LEnd(\shA_s)_\IQ)$. Note that the group of CSpin-isogenies from $\shA_s$ to itself can be identified with $\CSpin(\LEnd(\shA_s)_\IQ)$, which surjects to $\SO(\LEnd(\shA_s)_\IQ)$. Therefore, up to composing with CSpin-isogeny $\shA_s \to \shA_s$, we may assume that $f$ induces $\psi^\perp$ by conjugation. By assumption, $\psi^\perp$ extends to an isomorphism $\bL_{\cris, s} \stackrel{\sim}{\to} \bL_{\cris, s'}$, so by \ref{relpos} we may assume that $f$ is a prime-to-$p$ isogeny. Since the map $\shS_{\sK_0}(L) \to \shS^\ad_{\sK_0}(L)$ is \'etale, the points $\rho(t_V), \rho(t_V')$ lift to some $V$-valued points $s_V, s'_V$ on $\shS_{\sK_0}(L)$ extending $s, s'$. 

Let $s_K, s_K'$ be the generic points of $s_V, s_V'$ and $s_{\bar{K}}, s_{\bar{K}'}$ be the geometric points corresponding to an algebraic closure $\bar{K}$ of $K$. We claim that for some $m \ge 1$, $g := p^m f$ lifts to an isogeny $g_V: \shA_{s_V} \to \shA_{s'_V}$. Indeed, recall that $\Fil^1 \bL_{\dR, s_V}$ is free of rank $1$, and $\H^1_\dR(\shA_{s_V}/V) = \bH_{\dR, s_V} = \ker \Fil^1 \bL_{\dR, s_V} \text{ and } \Fil^\bullet \bL_{\dR, s_V}(-1) = \Fil^\bullet \P^2_\dR(X_V/V)$, where $\bL_{\dR, s_V}$ is considered as a subspace of $\End (\bH_{\dR, s_V})$ via left multiplication. The same holds verbatim for $X_V'$ and $s'_V$, so the map $f$ preserves the Hodge filtrations induced by $\shA_{s_K}$ and $\shA_{s_K'}$ via the Berthelot-Ogus isomorphisms. The claim now follows from \cite[Thm~3.15]{BO}. 

Finally, the isomorphism $\P^2_\dR(X_K/K) \stackrel{\sim}{\to} \P^2_\dR(X'_K/K)$ given by the restriction of $\psi_{\dR, K}$ is equal to the isomorphism $\bL_{\dR, s_K} \stackrel{\sim}{\to} \bL_{\dR, s'_K}$ induced by $f_V$ after we identify $\P^2_\dR(X_K/K)$ and $\P^2_\dR(X'_K/K)$ with $\bL_{\dR, s_K}$ and $\bL_{\dR, s'_K}$. Given $K \into \IC$, $g_V(\IC)$ induces an isomorphism $g_B: \H^1(\shA_{s_V}(\IC), \IQ) \sto \H^1(\shA_{s'_V}(\IC), \IQ)$ which induces by conjugation $\mathrm{conj}(g_B): \bL_{B, s_\IC} \tensor \IQ \sto \bL_{B, s'_\IC} \tensor \IQ$, where $s_\IC, s'_\IC$ are the $\IC$-points induced by $s_K, s'_K$. This gives rise to an isomorphism $\psi_B : \H^2(X_K(\IC), \IQ) \sto \H^2(X'_K(\IC), \IQ)$. 

It remains to check that $\psi_B$ is $\IZ_{(p)}$-integral. If $V = W$, then in constructing $g_V$ we can simply take $g = f$, so that $g_V$ is a prime-to-$p$ isogeny. Hence $g_B$ and $\mathrm{conj}(g_B)$ both preserves the underlying $\IZ_{(p)}$-lattices. If $p \ge 5$, apply \ref{lem: induce lifting}(c). 
\end{proof}

\subsubsection{Proof of Theorem~\ref{crysTor}.} For now let us just assume that $\psi$ preserves the Beauville-Bogomolov forms, sends some polarization $\xi \in \NS(X)$ to another polarization $\xi' \in \NS(X')$ and fits into a commutative diagram with an isomorphism $\psi_\cris$ as in the theorem. By \cite[Prop.~3.13]{Ogus} the $\IZ_p$-lattice $\H^2_\cris(X/W)^{F = p}$ is of the form $p T_0 \oplus T_1$ for some self-dual $\IZ_p$-lattices $T_0, T_1$ and $T_0$ has even rank. Therefore, there exists some line bundle $\zeta \in \NS(X)$ with $p \nmid q_X(\zeta)$. Set $\zeta' := \psi(\zeta)$. Note that for $m \gg 0$, $p^m \xi + \zeta$ and $p^m \xi' + \zeta'$ are ample and $p \nmid (p^m \xi + \zeta)^{2n} = (p^m \xi' + \zeta')^{2n}$. When $n - 1 = 0, 1$ or a prime power, we assume that $p \neq q_X(\xi) = q_{X'}(\xi')$ to start with, and take $\zeta, \zeta'$ to be $\xi, \xi'$.  

By \ref{liftpair}, for some finite flat extension $V$ of $W$, there exists a lift $(X_V, \xi_V, \zeta_V)$ of the tuple $(X, \xi, \zeta)$ over $V$. As we have explained, this induces a lift $(X'_V, \xi'_V, \zeta'_V)$ of $(X', \xi', \zeta')$. When $n - 1 = 0, 1$ or a prime power, we may take $V = W$ by Prop.~\ref{flatloc}(b). 

Choose an isomorphism $\bar{K} \iso \IC$. Write $X_\IC, X_\IC'$ for $X_V \tensor \IC, X'_V \tensor \IC$. Lem.~\ref{AHLift} tells that $\psi_\cris \tensor \IC$, which we may identify with an isomorphism $\H^2_\dR(X_\IC/\IC) \stackrel{\sim}{\to} \H^2_\dR(X'_\IC/\IC)$, restricts to $\psi_B : \H^2(X_\IC, \IZ) \stackrel{\sim}{\to} \H^2(X'_\IC, \IZ)$. If $n - 1 = 0, 1$, or a prime power, then $\Mon^2(X_\IC) = \O_+(\H^2(X_\IC, \IZ))$, so that $\psi_B$ is automatically a parallel transport operator up to a sign. By \ref{spinorientation}, $\psi_B$ has the correct sign, so that by Verbitsky's theorem it is induced by an isomorphism $f_\IC : X'_\IC \stackrel{\sim}{\to} X_\IC$, which specializes to the desired isomorphism $f : X' \stackrel{\sim}{\to} X$ by \cite[Thm~2]{MM}. This gives the second statement of the theorem. 

Now we treat the general case, for which we need the assumption that $\psi_\et$ is induced by a polarized family $(\sX \to S, \bxi)$ of $\HKn$-type varieties, for some connected scheme $S$ of finite type over $W$. Let $s, s'$ be the points on $S$ whose fibers are $(X, \xi)$ and $(X', \xi')$ and $\varepsilon$ be the \'etale path connecting $s$ and $s'$ which induces $\psi_\et$. Let $(\Lambda, \lambda)$ be the pointed lattice $\PL(X, \xi) = \PL(X', \xi')$. Construct $\wt{S}^\#$ as in \ref{period+}, so that $\wt{S}^\#$ is equipped with universal trivializations $\epsilon$ and $\Delta$ of $\det(\bP^2_2)$ and $\disc(\bP^2_{\what{\IZ}^p})$, and we get a morphism $\iota : \wt{S}^\# \to \wt{\Mod}^{(\Lambda, \lambda), 
\#}_{n, d}$. Let $\epsilon^p$ be the induced artificial orientation on $\wt{S}^\#$. Let $\sK^p$ be a compact open subgroup of $\sK_0^p \subset \CSpin(L \tensor \IA_f^p)$ such that $\wt{\Mod}^{(\Lambda, \lambda)}_{n, d, \sK^p}$ is a scheme. Construct a finite \'etale cover $\wt{S}_{\sK^p}$ of $\wt{S}^\#$ such that a morphism $T \to \wt{S}_{\sK^p}$ corresponds to a morphism $T \to \wt{S}^\#$ together with a choice of $\sK^p$-level structure on $(\sX_T \to T, \bxi_T)$ which is compatible with $\epsilon^p$ and $\Delta$. Now we obtain a morphism $\iota : \wt{S}_{\sK^p} \to \wt{\Mod}^{(\Lambda, \lambda)}_{n, d, \sK^p} \tensor W$. By \ref{GalEx}, each connected component $\wt{S}_{\sK^p}$ is a Galois cover of $S$. We can lift $s, s'$ and $\varepsilon$ to a connected component of $\wt{S}_{\sK^p}$. This means that we might as well replace $S$ by this connected component.


For each $k$-point $t \in S$, we choose an affine open neighborhood $U_{\iota(t)}$ of $\iota(t)$ in $\wt{\Mod}^{(\Lambda, \lambda)}_{n, d, \sK^p} \tensor W$. Then we take $U := \union_{t} U_{\iota(t)}$ as $t$ runs through the $k$-points of $S$. Since the morphism $\iota$ factors through $U$, the morphism $\psi_\et$ is now induced by an \'etale path from $u:= \iota(s)$ to $u':=\iota(s')$. The liftings $(X_V, \xi_V)$ and $(X_V', \xi_V')$ induce $V$-points $u_V, u_V'$ lifting $u, u'$. Let $u_\IC, u_\IC'$ be the $\IC$-points on $U$ given by $u_V \tensor \IC, u_V' \tensor \IC$. There is a commutative diagram 
\begin{center}
    \begin{tikzcd}
    \H^2_\et(X_\IC, \what{\IZ}^p) \arrow{r}{\psi_B \tensor \what{\IZ}^p} \arrow{d}{} & \H^2_\et(X'_\IC, \what{\IZ}^p) \arrow{d}{} \\
    \H^2_\et(X, \what{\IZ}^p) \arrow{r}{\psi_\et} & \H^2_\et(X', \what{\IZ}^p)
    \end{tikzcd}
\end{center}
where the vertical maps are given by the canonical \'etale paths along $V$ from $u_\IC$ and $u'_\IC$ to $u$ and $u'$ respectively. Therefore, $\psi_B \tensor \what{\IZ}^p$ is induced by an \'etale path from $u_\IC$ to $u'_\IC$. 

Clearly $U$ is connected. Applying \ref{Sasha} to $U_{\iota(t)}$'s, we see that $U \tensor_W \IC$ is also connected. We now choose a topological path from $u'_\IC$ to $u_\IC$, parallel transport along which induces an isomorphism $\gamma : \H^2(X'_\IC, \IZ) \stackrel{\sim}{\to} \H^2(X'_\IC, \IZ)$. Note that by \ref{spinorientation}, both $\psi_B$ and $\gamma$ are orientation preserving. Therefore, $\gamma \circ \psi_B \in \O_+(\H^2(X_\IC, \IZ))$. Since $(\gamma \circ \psi_B) \tensor \what{\IZ}^p$ lies in the image of \'etale monodromy representation 
$$ \pi_1^\et(U, u_\IC) \to \O(\H^2_\et(X_\IC, \what{\IZ}^p), c_1(\xi_\IC)), $$ $(\gamma \circ \psi_B) \tensor \what{\IZ}^p$ acts trivially on $\disc(\H^2_\et(X_\IC, \what{\IZ}^p))$, by the existence of $\sK^{p, \ad}$-level structures on $U$. Markman's theorem \ref{Markman} then ensures that $\gamma \circ \psi_B \in \Mon^2(X_\IC)$, so $\psi_B$ is a parallel transport operator. Now we conclude as before: The isomorphism $f_\IC : X'_\IC \stackrel{\sim}{\to} X_\IC$ provided by Verbitsky's theorem specializes to the desired $f : X' \stackrel{\sim}{\to} X$. \qed
 
\subsection{Supersingular Specialization Operators}
\'Etale parallel transport operators are defined as rather formal objects. In this subsection, we explain that for supersingular $\HKn$-type varieties, deformations of line bundles naturally give rise to such operators. In particular, these operators come up naturally if one studies the crystalline period morphism for higher dimensional $\HKn$-type varieties in the style of Ogus. 

Again let $k$ denote an algebraically closed field of characteristic $p > n + 1$. 
\begin{definition}
Let $S$ be a scheme over $k$. We say that a $\HKn$-space $f: \sX \to S$ is supersingular if its geometric fibers are all supersingular. We say that the family $\sX/S$ \textit{admits a marking}, if the natural map $\underline{\Pic}_{\sX/S}(S) \to \Pic(\sX_s)$ has finite $p$-power torsion cokernel for every geometric point $s$. 
\end{definition}

The above definition is of course motivated by Ogus' definition of a marking \cite[Def.~2.1]{Ogus2}. Supersingular $\HKn$-spaces are easy to find:  

\begin{lemma}
Let $S$ be a smooth $k$-variety and $\sX \to S$ be any supersingular $\HKn$-space over $S$. There exists a \'etale cover $S' \to S$ such that $\sX_{S'}$ admits a marking. 
\end{lemma}
\begin{proof}
Let $s \in S$ be any $k$-point. Recall that Artin's theorem (\cite[Prop.~5.5]{Ogus}) works verbatim in our situation, so that we can find elements $\xi_1, \cdots, \xi_{23}$ which generate $\Pic(\sX_s)$ up to inverting $p$ and deform to the complete local ring of $s$. Then by applying Artin's approximation theorem to the relative Picard functor, we obtain an \'etale neighborhood $U$ of $s$ such that $\xi_i$'s deform to $\underline{\Pic}_{\sX_U/U}(U)$.    
\end{proof}

According to Ogus' philosophy, there is the following table of analogies: \begin{center}
    \begin{tabular}{|c|c | c |}\hline
      & supersingular & complex \\ \hline 
     single K3 $X$ &  $\NS(X)$ + F-crystal $\H^2_\cris(X/W)$ & HS on $\H^2(X, \IZ)$ \\ \hline
     K3 family $f : \sX \to S$ & $\underline{\Pic}_{\sX/S}$ + $\IR^2 f_{\cris*} \sO_{\sX/S}$  & VHS on $\IR^2 f_* \underline{\IZ}$ \\ \hline
\end{tabular}
\end{center}

Now we fit the notion of parallel transport into the picture. Note that if $f : \sY \to T$ is a complex analytic family of hyperk\"ahler manifolds such that $T$ is connected and $R^2 f_* \underline{\IZ}$ is a constant local system, then for any two points $t, t' \in T$, there is a \textit{canonical} isomorphism $\H^2(\sY_t, \IZ) \iso \H^2(\sY_{t'}, \IZ)$, which is a parallel transport operator. Conversely, every parallel transport operator is given this way, because we can always pass to the universal cover of the base. Similarly, on a supersingular $\HKn$-space $\sX$ over a connected scheme $S/k$ which admits a marking, the sheaf of abelian groups $\Pic_{\sX/S}[1/p]$ is constant and for any two $k$-points $s, s' \in S$, there is a canonical isomorphism $\NS(\sX_s)[1/p] \iso \NS(\sX_{s'})[1/p]$. Hence we make the following definition: 

\begin{definition}
\label{defssp}
Let $X$ and $X'$ be two supersingular $\HKn$-type varieties over $k$. An isomorphism $\psi : \NS(X)[1/p] \stackrel{\sim}{\to} \NS(X')[1/p]$ is called a \textit{supersingular specialization operator} if it is induced by a supersingular $\HKn$-space $\sX$ over a connected scheme $S$ over $k$ which admits a marking and contains both $X, X'$ as fibers. We say that $\psi$ is in addition (primitively) \textit{polarizable}, if moreover $\sX/S$ admits a (primitive) polarization.
\end{definition}

By specialization and generization via discrete valuation rings, one easily checks that supersingular specialization operators give rise to \'etale parallel transport operators in a natural way: 
\begin{proposition}
\label{ssptoetPT}
Let $X$ and $X'$ be two supersingular $\HKn$-type varieties over $k$. Suppose $\psi : \Pic(X)[1/p] \stackrel{\sim}{\to} \Pic(X')[1/p]$ is a supersingular specialization operator induced by a supersingular $\HKn$-space $\sX$ over a connected scheme $S$ over $k$ which admits a marking. Then the isomorphism $\psi \tensor \what{\IZ}^p : \H^2_\et(X, \what{\IZ}^p) \stackrel{\sim}{\to} \H^2_\et(X', \what{\IZ}^p)$ is induced by an \'etale path on $S$. 
\end{proposition} 
In the above statement we implicitly used the natural isomorphisms $\Pic(X) \tensor \what{\IZ}^p \stackrel{c_1}{\iso} \H^2_\et(X, \what{\IZ}^p)$ and  $\Pic(X') \tensor \what{\IZ}^p \stackrel{c_1}{\iso} \H^2_\et(X', \what{\IZ}^p)$ provided by \ref{tateImprove}. We remark that, if $S$ is noetherian, and $\underline{\Pic}_{\sX/S}(S)$ admits an element $\bzeta$ such that $p \nmid \bzeta_s^{2n}$,\footnote{For example, this condition is always satisfied for $N$-marked families of supersingular K3's when $N$ is one of the supersingular K3 lattices.} then $\sX/S$ admits a primitive polarization if and only if it admits any polarizaiton. Indeed, $\bzeta$ has to be primitive and if $\bxi$ is any polarization on $\sX$, use $p^m \bxi + \bzeta$ for $m \gg 0$. 

\section{Examples and Applications}
\subsection{Moduli of Sheaves on K3's} \subsubsection{}\label{ModSh} Now we discuss some concrete examples of $\HKn$-type varieties. Let $k$ denote an algebraically closed field and $S$ be a K3 surface over $k$. We can consider moduli spaces of stable sheaves on $S$ (cf. \cite[\S2.1]{Charles}, \cite[\S4.2]{LiFu}). Let $\alpha$ be the numerical equivalence class of a closed point on $S$. A Mukai vector $v$ on $S$ is an element in the \textit{algebraic Mukai lattice} $N(S) := \IZ \oplus \NS(S) \oplus \IZ \alpha$. Denote by $c_1(v)$ the component of $v$ in $\NS(S)$. The lattice $N(S)$ is equipped with a Makai pairing defined by $\<(a, b, c), (a', b', c')\> = bb' - ac' - a'c$. Let $H$ be a polarization on $S$ and $v$ be a Mukai vector such that $v^2 + 2 = 2n$. If $H$ is in general position with respect to $v$, the moduli space of $H$-stable sheaves with Mukai vector $v$ on $S$ is representable by a smooth projective variety $\sM_H(S, v)$. Now suppose $\mathrm{char\,} k = p > n + 1$ and for some finite flat extension $V$ of $W$, the triple $(S, H, c_1(v))$ lifts to $(S_V, H_V, c_1(v)_V)$. Then the Mukai vector $v$ also lifts to $S_V$ and the relative moduli space $\sM_{H_V}(S_V, v)$ is a lift of $\sM_H(S, v)$. $\sM_{H_V}(S_V, v)$ is projective over $V$ and its generic fiber is of $\HKn$-type (cf. the proof of \cite[Thm~2.4]{Charles2}).

\begin{proposition}
\label{ModSh+}
$\sM_H(S, v)$ is a variety of $\HKn$-type if $\sM_H(S, v)$ has the same Hodge numbers as a complex $\HKn$-type variety.
\end{proposition}
\begin{proof}
Our goal is to show that $\sM_{H_V}(S_V, v)$ as in the preceeding paragraph can be chosen such that some primitive polarization on $\sM_H(S, v)$ lifts to $\sM_{H_V}(S_V, v)$. Note that the Hodge-de Rham spectral sequence necessarily degenerates at $E_1$-page for $\sM_{H_V}(S_V, v)$ and its generic and special fibers are equipped with Beauville-Bogomolov forms (cf. \ref{extBBF0}).

If $S$ is non-supersingular, then by \cite[Thm~4.1]{LM}, there exists a lift $S_V$ as above with $V = W$. Since $\sM_{H_V}(S_V, v)$ is still projective, there exists some polarization $\wt{\xi}$ on $\sM_{H_V}(S_V, v)$. By \cite[Thm~3.8]{BO} or \cite[Prop.~1.12]{Ogus}, the primitive polarization associated to the restriction of $\wt{\xi}$ at the special fiber lifts to $\sM_{H_V}(S_V, v)$. 

Now suppose $S$ is supersingular. Let $v^\perp$ be the orthogonal complement of $v$ in $N(S)$. We claim that there exists an element $w \in v^\perp$ such that $p \nmid w^2$. Indeed, since $ p \nmid v^2$, we have $\disc((v^\perp)_p) = \disc(N(S)_p) = \disc(\NS(S)_p) \iso (\IZ/ p \IZ)^{2 \sigma}$ for some $\sigma \le 10$. If such $w$ does not exist, then we must have $|\disc((v^\perp)_p)| \ge p^{\mathrm{rank\,} v^\perp} = p^{23}$, which contradicts $|\disc((v^\perp)_p)| \le p^{20}$. By \cite[Prop.~1.5]{Charles2}, we can find a lift $S_V$ of $S$ over a finite flat extension $V$ of $W$ as in the first paragraph such that $c_1(w)$ also lifts to $S_V$. A lift of $c_1(w)$ induces a lift of $w$. Let $K$ be the fraction field of $V$ and $\bar{K}$ be an algebraic closure of $K$. By Thm~2.4(iv) in \textit{loc. cit.}, there exists an injective isometry
$$ \theta_v : [\wt{v}^\perp \subset N(S_V \tensor \bar{K})] \to \NS(\sM_{H_V}(S_V, v) \tensor \bar{K}) $$
where $\wt{v}$ is the lift of $v$ to $S_V \tensor \bar{K}$. This implies that up to replacing $V$ by a further finite flat extension, the lift of $w$ to $S_V \tensor \bar{K}$ induces via $\theta_v$ a line bundle $\wt{\zeta}$ on $\sM_{H_V}(S_V, v)$ which specializes to a line bundle $\zeta$ on $\sM_H(S, v)$. Since $\theta_v$ is an isometry, $w^2 = \zeta^2$. Now let $\wt{\xi}$ be any polarization on $\sM_{H_V}(S_V, v)$ and let $\xi$ be its restriction to the special fiber. For $h \gg 0$, $p^h \wt{\xi} + \wt{\zeta}$ is polarization on $\sM_{H_V}(S_V, v)$. Since $p\nmid (p^h \xi + \zeta)^2$ and the Beauville-Bogomolov form on $\NS(\sM_H(S, v))$ takes integral values, $p^h \xi + \zeta$ cannot be a $p$-power of another line bundle. The primitive line bundle associated to $p^h \xi + \zeta$ lifts to $\sM_{H_V}(S_V, v)$. 
\end{proof}

In practice it is hard to check whether a particular $\sM_H(S, v)$ has the expected Hodge numbers, i.e., those of a complex $\HKn$-type variety, although we believe this should be generically true. The following gives a sufficient condition: 

\begin{proposition}
\label{prop: example 2n + 1}
Suppose $p > 2n + 1$ and one of the following is satisfied: $S$ is non-supersingular, or $c_1(v)$ is a multiple of $H$ and $c_{1, \dR}(H) \not\in \Fil^2 \H^2_\dR(X/k)$. Then $\sM_H(S, v)$ has the same Hodge numbers as a complex $\HKn$-type variety. 
\end{proposition}
\begin{proof}
The hypothesis guarantees that we can always find a lift $(X_W, H_W, c_1(v)_W)$ of $(X, H, c_1(v))$ over $W$. Then we can consider $\sM_{H_W}(S_W, v)$ as before. Let us write $X$ for $\sM_H(S, v)$ and $X_W$ for $\sM_{H_W}(S_W, v)$. Since $\dim X = 2n$, a well known theorem of Deligne-Illusie guarantees that the Hodge-de Rham spectral sequence of $X$ degenerates at $E_1$-page. Since $\H^j_\et(X_W \tensor \bar{K}_0, \IZ_p)$ is torsion-free for all $j$, $H^j_\cris(X/W)$ is torsion-free for $j < p - 1$ (\cite[Rmk~6.4]{FM}). This shows that $\H^j_\dR(X_W/W)$ is free over $W$ for $j \le 2n$. Poincar\'e duality of algebraic de Rham cohomology (\cite[Tag 0FW7]{stacks-project}) ensures that $\dim \H^j_\dR$ of the special fiber agrees with that of the generic fiber.
\end{proof}

As a special case, the Hilbert scheme of $n$ points $S^{[n]}$ can be viewed as $\sM_H(v)$ for any polarization $H$ and $v = (1, 0, 1-n)$. Therefore, $S^{[n]}$ is a $\HKn$-type variety whenever $p > 2n + 1$. 

\begin{remark}
\label{ChQues}
In the introduction of \cite{Charles2}, Charles asked whether for fixed $n, d$, there exists bounds $m, N, D$ such that if $X$ is a moduli space of stable sheaves on K3 surfaces of dimension $2n$ and $L$ is a big and nef line bundle on $X$ with $L^{2n} = d$, the complete linear series $|m L|$ induces a birational map from $X$ to a subvariety of $\IP^{n'}$ of degree at most $D$ for some $n' \le N$. Suppose now we fixed a characteristic $p> n + 1$ for the base field of $X$. By \ref{ModSh+}, $X$ is a $\HKn$-type variety as long as it has the expected Hodge numbers. If we restrict to considering such $X$'s, \ref{bounded} gives an affimative answer to Charles' question when $L$ is ample, in which case birational maps can be replaced by isomorphisms. The bounds $m, N$ we give are not explicit and may depend on $p$.
\end{remark}

Finally, we remark that non-explicit examples are very easy to construct: Take a complex $\HKn$-type variety. By spreading out and the arguments for \ref{prop: example 2n + 1}, one produces a $\HKn$-type variety in all but finitely prime characteristic. Once we find one $\HKn$-type variety, we obtain a whole 20-dimensional (if $n \ge 2$) equicharacteristic family by deformation. Therefore, abstractly we know there are lots of $\HKn$-type varieties to which our theorems apply. 

\subsection{Cubic Fourfolds}
We first review some basic theory about complex cubic fourfolds (cf. \cite{BD}). Let $Y_\IC \subset \IP^5_\IC$ be a smooth cubic hypersurface. Let $h$ be its hyperplane section. The Fano variety of lines on $Y_\IC$, which we denote by $F(Y_\IC)$, is a $\mathrm{K3}^{[2]}$-type hyperk\"ahler variety. Via the Pl\"ucker embedding, $h$ induces a canonical polarization on $F(Y_\IC)$, which we denote by $g$. The natural incidence correspondence on $Y_\IC \times F(Y_\IC)$ induces an Abel-Jacobi map 
$$ \mathrm{AJ}: \H^4(Y_\IC, \IZ) \to \H^2(F(Y_\IC), \IZ) $$
which sends $h^2$ to $g$ and restricts to a Hodge isometry $P^4(Y_\IC, \IZ(2)) \stackrel{\sim}{\to} \P^2(F(Y_\IC), \IZ(1))$. Here the pairing on the LHS is the usual Poincar\'e pairing and the pairing on the RHS is the Beauville-Bogomolov form, and $\P^4(Y_\IC, \IZ(1))$ (resp. $\P^2(F(Y_\IC), \IZ)$) stands for the orthogonal complement of $h^2$ (resp. $g$). Note that before we restric to primitive parts, $\AJ^\vee$ is not an isometry, because $h^4 = 3$ and $g^2 = 6$. We also remark that as $Y_\IC \subset \IP^5_\IC$ is smooth ample divisor, the natural map $\H^j(\IP_\IC^5, \IZ) \to \H^j(Y_\IC, \IZ)$ is bijective for $j \le 3$.

Let $k$ be an algebraically closed field of characteristic $p > 0$. Let $Y$ be a cubic fourfold over $k$ with hyperplane section $h$. We again denote the Fano variety of lines on $Y$ by $F(Y)$ and the canonical polarization on $F(Y)$ by $g$. Clearly, we can lift $Y$ to a cubic fourfold $Y_W$ over $W$. A lift $Y_W$ of $Y$ over $W$ induces a lift $F(Y_W)$ of $F(Y)$. $Y$ is said to be supersingular if the Newton polygon of $\H^4_\cris(Y/W)$ is a straight line. The incidence correspondence on $Y \times F(Y)$ induces Abel-Jacobi maps $\AJ : \H^4_\et(Y, \IZ_\ell(2)) \stackrel{\sim}{\to} \H^2_\et(F(Y), \IZ_\ell(1))$ and $\AJ : \H^4_\cris(Y/W) \tensor K_0(2)  \stackrel{\sim}{\to} \H^2_\cris(F(Y)/W) \tensor K_0(1)$ which send the class of $h^2$ to $g$ and restrict to isometries on the primitive parts. We are primarly concerned with the case $p \ge 7$, which ensures that $\H^*_\cris(Y/W)$ and $\H^*_\cris(F(Y)/W)$ are torsion-free.
We denote the orthogonal complement of $h^2 \in \H^4_*(Y)$ ($* = \dR, \ell, \cris, $etc.) by $\P^4_*$. 

Charles and Pirutka established the \textit{integral} Tate conjecture for cubic fourfolds (\cite{CharlesIntTate}), in particular:
\begin{proposition}
\label{C-P}
\emph{(Charles-Pirutka)}
Let $p \neq 2, 3$ be a prime and $Y$ be a supersingular cubic fourfold over $\bar{\IF}_p$. Then for any $\ell \neq p$, the cycle class map $\cl_\et : \Ch^2(Y)_\num \tensor \IZ_\ell \to \H^4_\et(Y, \IZ_\ell)$ is an isomorphism. 
\end{proposition}

We fill in the crystalline counterpart of the above result, in analogy to \ref{tateImprove}. 
\begin{proposition}
\label{C-Pcris}
Let $k$ be algebraically closed field of characteristic $p \ge 7$ and $Y$ be a supersingular cubic fourfold over $k$. The map $\cl_\cris : \Ch^2(Y)_\num \tensor \IZ_p \to H_\cris^4(Y/W)^{F = p^2}$ is an isomorphism.
\end{proposition}
\begin{proof}
We write $X$ for $F(Y)$. Consider the adjoint $\AJ^\vee : \P^6_\cris(X/W) \to \P^4_\cris(Y/W)$ of $\AJ$. Note that $\AJ^\vee$ is also induced by the incidence correspondence, so we have a compatible map on Chow groups $\Ch^3(X) \to \Ch^2(Y)$. 

For every $j \ge 0$ set $\P^{j}(X) = \{ b \in \Ch^j(X)_\num : b \cup g^{4 - j} = 0 \} $ and $\P^j(Y) = \{ a \in \Ch^j(Y)_\num : a \cup h^{4 - j} = 0 \}$. Note that since $h^4 = 3$ is prime to $p$, $\Ch^2(Y)_\num$ splits into an orthogonal direct sum $W h^2 \oplus \P^2(Y)$. Therefore, it suffices to show that $\P^2(Y) \tensor \IZ_p \to \P^6_\cris(Y/W)^{F = p^2}$ is an isomorphism. First, we remark that $\AJ^\vee$ restricts to a map $P^3(X) \to \P^2(Y)$. Indeed, for any $b \in \Ch^3(Y)_\num$ and $a \in \Ch^2(X)_\num$, $\AJ^\vee(b) \cup a = \AJ(a) \cup b$. Therefore, if $b \in \P^3(Y)$, then $\AJ^\vee(b) \cup h^2 = \AJ(h^2) \cup b = g \cup b = 0$. 

Now we have a diagram 
\begin{center}
    \begin{tikzcd}
    \P^3(X) \tensor \IZ_p \arrow{d}{} \arrow{r}{} & \P^2(Y) \tensor \IZ_p \arrow{d}{} \\
    \P^6_\cris(X/W(3))^{F = 1} \arrow{r}{} & \P^4_\cris(Y/W(2))^{F = 1}
    \end{tikzcd},
\end{center}
so it suffices to show that the bottom arrow is an isomorphism, or equivalently, $\Ch^3(X)_\num \tensor \IZ_p \to \H^6_\cris(Y/W(3))^{F = 1}$ is an isomorphism. Consider the following diagram:
\begin{center}
    \begin{tikzcd}
    \Ch^1(X)\tensor \IZ_p \arrow{r}{\cup g^2} \arrow{d}{} & \Ch^3(X)\tensor \IZ_p \arrow{d}{} \\
    \H^2_\cris(X/W(1))^{F = 1} \arrow{r}{\cup g^2} & \H^6_\cris(X/W(3))^{F = 1}
    \end{tikzcd}
\end{center}
The left vertical arrow is an isomorphism by \ref{tateImprove}. Therefore, it suffices to show that the bottom horizontal arrow is an isomorphism, for which it suffices to show $\cup g^2 : \H^2_\cris(X/W) \to \H^6_\cris(X/W)$ is an isomorphism of $W$-modules. Since the Poincar\'e pairing $\H^2_\cris(X/W) \times \H^6_\cris(X/W) \to W$ is perfect, we only need to show that the pairing $\H^2_\cris(X/W)^{\tensor 2} \to W$ defined by $(x, y) \mapsto x \cup y \cup g^2$ is perfect. Since $q_X(g) = 6$ and $6 \in W^\times$ $\H^2_\cris(X/W)$ splits into an orthogonal direct sum $W g \oplus \P^2_\cris(X/W)$. We claim that $(x, y) \mapsto x \cup y \cup g^2$ defines a perfect pairing for both $W g$ and $\P^2_\cris(X/W)$. For the former, this follows from $g^4 = \lambda_2 q_X(g)^2 = 108$. For the latter, this pairing is perfect because it differs from the restriction of $q_X$ to $\P^2_\cris(X/W)$ by a scalar in $W^\times$. 
\end{proof}

\noindent \textit{Proof of Theorem~\ref{cubicTorelli}.} There is a Kuga-Satake period map for cubic fourfolds (\cite[6.13]{Keerthi}). Let $L$ be the primitive lattice of a cubic fourfold, which is isomorphic to
\begin{align}
\label{primCF}
    U^{\oplus 2} \oplus E_8^{\oplus 2} \oplus \begin{bmatrix}2 & 1 \\ 1 & 2
\end{bmatrix}
\end{align}
where $U$ and $E_8$ are as introduced in \ref{sec: known BBFs}. Let $\CF_{\IZ_{(p)}}$ be the moduli stack of cubic fourfolds over $\IZ_{(p)}$. There is a period morphism $\rho : \tilde{\CF}_{\IZ_{(p)}} \to \shS^\ad_{\sK_0}(L)$, where $\tilde{\CF}_{\IZ_{(p)}}$ is an orientation double cover of $\tilde{\CF}_{\IZ_{(p)}}$. \ref{compMotive} holds verbatim when $\P^2_*$ is replaced by $P^4_*$ ($* = \dR, \ell, \cris, $etc.)

Choose an isomorphism $\bar{K}_0 \iso \IC$. Choose liftings $Y_W, Y_W'$ of $Y, Y'$ such that $\psi$ preserves the induced Hodge filtrations. By the same proof of \ref{AHLift}, there exists a Hodge isomorphism $\psi(\IC) : \H^2(Y_W'(\IC), \IZ) \stackrel{\sim}{\to} \H^2(Y_W(\IC), \IZ)$ which is compatible with $\psi$ in the obvious sense. In fact, one first produces an rational absolutely Hodge isomorphism $\psi(\IC) : \H^2(Y_W'(\IC), \IZ_{(p)}) \stackrel{\sim}{\to} \H^2(Y_W(\IC), \IZ_{(p)})$ and show that it preserves the integral structures using \ref{C-P}. By the Torelli theorem for cubic fourfolds (\cite{Voisin}, see also \cite{CharlesCubic}), $\psi(\IC)$ is induced by an isomorphism $f_\IC : Y_W(\IC) \stackrel{\sim}{\to} Y_W(\IC)$ which preserves the hyperplane sections. Clearly $f_\IC$ induces an isomorphism $F(f_\IC) : F(Y_W(\IC)) \stackrel{\sim}{\to} F(Y'_W(\IC))$. By \cite[Thm~2]{MM} again, $F(f_\IC)$ specializes to an isomorphism $F(f) : F(Y) \stackrel{\sim}{\to} F(Y')$. Note that $F(f_\IC)$ and $F(f)$ preserve hyperplane sections induced by Pl\"ucker embeddings. By \cite[Prop.~4]{CharlesCubic}, $F(f)$ is induced by an isomorphism $f : Y \stackrel{\sim}{\to} Y'$.   \qed

\begin{remark}
Our arguments will have to go through Fano varieties of lines because Matsusaka-Mumford's theorem cannot be applied directly to cubic fourfolds---they may be rational, hence ruled. Unlike Theorem \ref{crysTor}, Theorem \ref{cubicTorelli} is restricted to the $k = \bar{\IF}_p$ because we do not know whether \ref{C-P} holds for a general algebraically closed field. In fact, only $\ell = 2, 3$ may cause a problem, because these are the primes dividing $q_X(g) = 6$. Otherwise, one can just carry out the arguments in \ref{C-Pcris} for \'etale cohomology. We do conjecture that \ref{C-P} is true for a general algebraically closed field, based on the crystalline variational Tate conjecture (cf. \cite[Conj.~0.1]{Morrow}). 
\end{remark}

\subsubsection{} \label{sspCubic} Let $k$ be an algebraically closed field in characterstic $p \ge 7$. Let $Y$ be a supersingular cubic fourfold over $k$. Since the Abel-Jacobi map induces an isometry $P^4_\cris(Y/W(1)) \stackrel{\sim}{\to} \P^2_\cris(F(Y)/W)$, and $\H^4_\cris(Y/W) \iso W h^2 \oplus P^4_\cris(Y/W)$, $\H^4_\cris(Y/W(1))$ has the structure of a supersingular K3 crystal. Therefore, we can freely make use of results in \cite{Ogus}. In particular, the $\IZ_p$ lattice $\Ch^2(Y)_\num \tensor \IZ_p \iso \H^4_\cris(Y/W)^{F = p^2}$ has discriminant $(\IZ/p \IZ)^{2 \sigma}$ for some $\sigma \ge 1$ (cf. \cite[Prop.~3.13]{Ogus}). We call $\sigma$ the \textit{Artin invariant} of $Y$ and if $\sigma = 1$, we say that $Y$ is superspecial. 

We will use the theory of complex multiplication to produce superspecial cubic fourfolds, in the spirit of Honda-Tate theory. First, we make a preliminary definition.

\begin{definition}
The transcendental lattice of a cubic fourfold $Y_\IC$ (resp. K3 surface $X_\IC$), which we denoted by $T(Y_\IC)$ (resp. $T(X_\IC)$), is the orthogonal complement of the $\H^{0,0}$ part in $\H^4(Y_\IC, \IZ(2))$ (resp. $\H^2(X_\IC, \IZ(1))$. $Y_\IC$ (resp. $X_\IC$) is said to have complex multiplication (CM) if the Mumford-Tate group $\MT(T(Y_\IC))$ (resp. $\MT(X_\IC)$) is commuative. 
\end{definition}

The following lemma gives a sufficient condition for a cubic fourfold to be superspecial.
\begin{lemma}
\label{prodssp}
Suppose $Y_\IC$ is complex cubic fourfold such that $\rank \, T(Y_\IC) = 2$ and $p$ is prime to $|\mathrm{disc}(T(Y_\IC))|$. Let $k$ be an algebraically closed field over $\IF_p$. Suppose $Y$ is a supersingular cubic fourfold over $k$ such that for some finite flat extension $V$ of $W$, there exists a lift $Y_V$ of $Y$ over $V$ and for some embedding $V \into \IC$, $Y_V \tensor \IC \iso Y_\IC$. Then $Y$ has Artin invariant $1$. 
\end{lemma}
\begin{proof}
There is an isometric injection $\Ch^2_\num(Y_\IC) \into \Ch^2_\num(Y)$. Since $p$ is prime to $|\mathrm{disc}(T(Y_\IC))|$ and $\H^4(Y_\IC, \IZ)$ is self-dual, $\Ch^2_\num(Y_\IC) \tensor \IZ_p$ is self dual. Therefore, the Tate module $\H^4_\cris(Y/W)^{F = p^2}$ admits a self-dual sublattice of corank $2$. By \cite[Prop.~3.13]{Ogus}, $\H^4_\cris(Y/W)^{F = p^2}$ is of the form $T_1 \oplus pT_0$ for some perfect $\IZ_p$-lattices $T_0, T_1$ and $\mathrm{rank\,}T_0 = 2 \sigma \ge 2$. Therefore, $\sigma$ has to be $1$.  
\end{proof}

\begin{lemma}
\label{ArtInv1}
For every quadratic imaginary field $E$, there exists a complex cubic fourfold $Y_\IC$ such that $\End_{\Hdg} T(Y_\IC)_\IQ = E$ and $F(Y_\IC) \iso X_\IC^{[2]}$ for some complex K3 surface $X_\IC$. 
\end{lemma}

\begin{proof}
Set $d = 14$. Let $\sC^{\mathrm{mar}}_{d}$ be the moduli space of marked special cubic fourfolds of discriminant $d$ (cf. \cite[\S5.2]{Hassett}). By \cite[Thm~1.0.2]{Hassett}, there is an open immersion $\iota$ from $\sC^{\mathrm{mar}}_d$ to the moduli space of polarized K3 surfaces of degree $d$. If the moduli point of a cubic fourfold $Y_\IC$ is sent to that of a polarized K3 surface $(X_\IC, \xi_\IC)$, then there is an isomorphism of Hodge structures $T(Y_\IC(1)) \iso T(X_\IC)$. In particular, $Y_\IC$ has CM by $E$ if and only if $X_\IC$ does. By \cite[Cor.~3.8.3]{Rizov}, the set of polarized K3 surfaces $(X_\IC, \xi_\IC)$ of degree $d$ with CM by $E$ is dense in its moduli. Therefore, the set of cubic fourfolds with CM is also dense in $\sC^{\mathrm{mar}}_{d}$. By \cite[Thm~1.0.3]{Hassett}, $F(Y_\IC)$ is the Hilbert scheme of some K3 surface for a generic choice of $Y_\IC$ parametrized by $\sC^{\mathrm{mar}}_{d}$. 
\end{proof}

\begin{lemma}
\label{prodsupsing}
Let $d \in \IZ_{> 0}$ be any positive integer. For all but finitely many $p$ which is inert in $\IQ(\sqrt{-d})$, there exists a supersingular cubic fourfold $Y$ over $\bar{\IF}_p$ with Artin invariant $1$ such that $\Phi : X^{[2]} \stackrel{\sim}{\to} F(Y)$ for some supersingular K3 surface $X$.
\end{lemma}
\begin{proof}
Let $E := \IQ(\sqrt{-d})$ and let $Y_\IC$ be a cubic fourfold such that $\End_{\Hdg} T(Y_\IC)_\IQ = E$ and $F(Y_\IC) \iso X_\IC^{[2]}$ for some K3 surface $X_\IC$. Let $\tilde{\CF}_\IQ$ be the orientation double cover of the moduli stack of cubic fourfolds over $\IQ$ and $\rho : \tilde{\CF}_\IQ \to \Sh^\ad_{\sK_0}(L)$ be the Kuga-Satake period morphism, where $L$ is the lattice in (\ref{primCF}) (see \cite[\S6.13]{Keerthi}).  Let $s \in \tilde{\CF}_{\IQ}(\IC)$ be a point which corresponds to $Y_\IC$. Then $\rho(s) \in \Sh^\ad_{\sK_0}(L)(\IC)$ is a special point with reflex field $E$. Let $t \in \Sh_{\sK_0}(L)(\IC)$ be a point which maps to $\rho(s)$. Let $A_\IC$ be the fiber over $t$ of the universal abelian scheme over $\shS_{\sK_0}(L)$, i.e., the Kuga-Satake abelian variety for $Y_\IC$. Choose a polarization of $X_\IC$ and apply the similar construction to $X_\IC$, we obtain a Kuga-Satake abelian variety $B_\IC$ for $X_\IC$. By Morrison's computation \cite[Prop.~4.3.3]{K3book} and the theory of Kummer lattices  \cite[Rmk~14.3.22]{K3book}, $A_\IC$ and $B_\IC$ are both isogenous to a power of $\sE_\IC$ for an elliptic curve $\sE$ with CM by $E$. 

Since $A_\IC$ and $B_\IC$ correspond to special points with reflex field $E$ on spin Shimura varieties, $A_\IC, B_\IC$, and hence $Y_\IC$ and $X_\IC$ all descend to some finite extension $F$ of $E$. Denote their $F$-models by $A, B, Y, X$. The main point of the Kuga-Satake construction is that there are natural $\Gal(\bar{F}/F)$-equivariant inclusions 
\begin{align}
\label{KSCF}
    \P^4_\et(Y_{\bar{F}}, \IQ_\ell(2)) \subset \End (\H^1_\et(A_{\bar{F}}, \IQ_\ell)) \text{ and } \P^2_\et(X_{\bar{F}}, \IQ_\ell(1)) \subset \End (\H^1_\et(B_{\bar{F}}, \IQ_\ell)).
\end{align}

Up to extending $F$, we may assume that $A, B$ are isogenous to powers of an $F$-model $\sE$ of $\sE_\IC$ and the isomorphism $Y_\IC \iso X^{[2]}_\IC$ descends to $F$. Let $\sO_F$ be the ring of integers of $F$. For some open subscheme $U \subset \Spec \sO_F$ $Y$,  $A, B, X, Y, \sE$ spread to schemes $\shA, \mathscr{B}, \shX, \shY, \shE$ over $U$. By the reciprocity law applied to $\sE$, for any $\fp \in \Spec U$ above $p$ with reside field $\kappa(\fp)$, $\shA \tensor \kappa(\fp)$ and $\mathscr{B} \tensor \kappa(\fp)$ are supersingular if and only if $p$ is inert in $\IQ(\sqrt{-d})$. If $\shA \tensor \kappa(\fp)$ and $\mathscr{B} \tensor \kappa(\fp)$ are supersingular, then so are $\shY \tensor \kappa(\fp)$ and $\shX \tensor \kappa(\fp)$ by (\ref{KSCF}). 
Choose an embedding $\kappa(\fp) \into \bar{\IF}_p$. For all but finite many $p \ge 7$, there exists a place $\fp$ above $p$ such that the isomorphism $F(Y) \iso X^{[2]}$ specializes to an isomorphism $F(\shY_{\kappa(\fp)}) \iso \shX_{\kappa(\fp)}^{[2]}$. Let $p$ be such a prime. If furthermore $p$ does not divide $|\disc(T(Y_\IC))|$, then by \ref{prodssp}, $Y_{\bar{\IF}_p}$ has Artin invariant $1$. 
\end{proof}

\paragraph{Example} By Shioda and Katsura's results, we know that the Fermat cubic fourfold in characteristic $p$ is supersingular if and only if $p \equiv -1 \mod 3$ (cf. \cite[Thm~2.10, Rmk.~3.8]{Fermat}). We explain how their results relate to our arguments: we view the Fermat equation as defining a cubic fourfold $\shY$ over $\IZ[1/3, e^{\pi i /3}]$. The quadratic form $T(\shY_\IC)$ is given by (cf. \cite[Thm~1.8(1)]{LazaZheng})
$$ - \begin{bmatrix} 6 & 3 \\ 3 & 6 \end{bmatrix} $$
This is a Kummer lattice which is equipped with a \textit{unique} Hodge structure of K3 type (cf. \cite[\S14.3.4]{K3book}), so 
$$ \End_{\Hdg} T(\shY_\IC)_\IQ \iso \IQ(\sqrt{-3}). $$
By the proof of \ref{prodsupsing}, when $p \ge 7$,\footnote{We are working with $p \ge 7$ in order to make sure that the crystalline cohomology of cubic fourfolds are torsion-free, but it should be totally possible to prove this supersingularity statement using Shimura varieties for all primes.} we see that the reduction $\shY \tensor \IF_p$ is supersingular if and only if $p$ is inert in $\IQ(\sqrt{-3})$ if and only if $p \equiv -1 \mod 3$, which agrees with Shioda and Katsura's result. \ref{ArtInv1} also tells us that when $p \ge 7$, the Fermat cubic is \textit{superspecial} whenever it is supersingular. 

\begin{remark}
Laza and Zheng studied a number of cubic fourfolds with big symmetric groups. In particular, they gave explicit formulas of these fourfolds and their transcendental lattices (cf. \cite[Thm~1.8]{LazaZheng}). Using these explicit formulas, one can compute the range of primes for which those cubic fourfolds have supersingular reduction. 
\end{remark} 

We write $\mu(S)$ for the Dirichlet density for a set of primes $S$. 

\begin{theorem}
\label{ssext}
If $S$ is the set of primes $p$ such that there exists a superspecial cubic fourfold over $\bar{\IF}_p$, then $\mu(S) = 1$.
\end{theorem}
\begin{proof}
Let $r$ be any number and $\Pi_r = \{ p_1, \cdots, p_r \}$ be any set of $r$ primes. Consider the set of primes $$S(\Pi_r) := \{ p \textit{ is inert in } \IQ(\sqrt{-p_i}) \textit{ for some } p_i \in \Pi_r \} $$
Then $\mu(S(\Pi_r)) = 1 - 2^r$. By \ref{prodsupsing}, $\mu(S) \ge \mu(S(\Pi_r))$. Since $r$ can be arbitrarily large, $\mu(S) = 1$. 
\end{proof}


\begin{appendix}
\section{Formal Brauer Groups, Serre's Witt Vector Cohomology, and Line Bundles}

Let $Y$ be a smooth proper variety over an algebraically closed field $k$ of characteristic $p > 0$. Assume that $H^1(\sO_Y) = H^3(\sO_Y) = 0$, i.e., $h^{0,1} = h^{0,3} = 0$. Then the work of Artin and Mazur tells us that the formal Brauer group $\hat{\Br}_Y$ is pro-represented by a smooth formal Lie group over $k$. Moreover, there are natural isomorphism $\Lie(\hat{\Br}_Y) \iso \H^2(\sO_Y)$ and $\ID(\hat{\Br}_Y^*) \iso \H^2(Y, W(\sO_Y))$ (cf. \cite[\S II.4]{AM})\footnote{Unlike our convention, in \cite{AM} $\ID(-)$ stands for the \textit{covariant} Dieudonn\'e functor, so our isomorphism differs from the one in \cite[II Cor.~4.3]{AM} by a dual.}. Here we are making use of Serre's Witt vector cohomology. By the slope spectral sequence \cite[\S II, Cor.~3.5]{Illusie}, we know that there is a natural isomorphism of F-isocrystals $\ID(\hat{\Br}_Y^*)[1/p] \iso \H^2_\cris(X/W)[1/p]_{< 1}$. 

In this appendix, we give an exposition of the proof of \cite[Prop.~10.3]{GK}, in order to show that it holds in the following generality:
\begin{AppProp}
\label{GK} Let $Y$ be a smooth proper variety over $k$ such that $h^{0,1} = h^{3,0} = 0$, $h^{1,0} = 0$, $h^{0,2} = 1$ and the height $h$ of $\hat{\Br}_Y$ is finite. Then the map $d \log : \NS(Y) \tensor_\IZ k \to H^1(\Ohm_Y^1)$ is injective. 
\end{AppProp}

For the rest of the appendix, let $Y$ be as above. In the proofs below, the assumptions on $Y$ are boxed when they are used. Note that \fbox{$h^{0, 2} = 1$} implies that $\hat{\Br}_Y$ is an one-dimensional formal group over $k$. Let $W_i = W_i(\sO_Y)$ denote the sheaf Witt rings of length $i$. Recall that $\H^2(W(\sO_Y))$ is defined as the limit $\varprojlim_i \H^2(W_i)$. Denote by $F : W_i \to W_i, V : W_i \to W_{i + 1}, R : W_{i+1} \to W_i$ the Frobenius, verschiebung and restriction maps. There are exact sequences of abelian sheaves $0 \to W_{i - 1} \stackrel{V}{\to} W_i \stackrel{R^{i - 1}}{\to} \sO_Y \to 0$ and $0 \to \sO_Y \to W_i \stackrel{R}{\to} W_{i - 1} \to 0$. By taking cohomology of the these sequences and making use of \fbox{$h^{0, 1} = h^{0, 3} = 0$}, we see that $H^1(W_i) = H^3(W_i) = 0$, and there are exact sequences (cf. \cite[Lem.~4.1,4.2,4.5]{GK}): 
\begin{align}
\label{b1}
0 &\to \H^2(\sO_Y)\to \H^2(W_i) \stackrel{R}{\to} \H^2(W_{i - 1}) \to 0 \\ \label{b2} 0 &\to \H^2(W_{i - 1}) \stackrel{V}{\to} \H^2(W_i) \to \H^2(\sO_Y) \to 0
\end{align}
In particular, $\H^2(W_i) \stackrel{R}{\to} \H^2(W_{i - 1})$ is surjective. Using that $R$ commutes with $V$, we easily deduce that $RV \H^2(W_i) = V \H^2(W_{i - 1})$ (cf. \cite[Lem~4.3]{GK}). Set $H := \H^2(W(\sO_Y))$. Because of the isomorphism $\ID(\hat{\Br}_Y^*) \iso H$, we deduce using Dieudonn\'e theory that \fbox{$h = \mathrm{rank\,} H$}, $H/ VH = \Lie(\hat{\Br}_Y) \iso \H^2(\sO_Y)$, \fbox{$\dim_k H/FH = h - 1$}, and $V^{h - 1} H = FH$. 
\begin{AppLem}
\label{5.2}
\emph{(\cite[Cor.~5.2]{GK})}
$h = \min \{ i \ge 1: \H^2(W_i) \stackrel{F}{\to} \H^2(W_i) \text{ is nonzero.} \}$
\end{AppLem}
\begin{proof}
First, we remark that for every $j \le i$, since $R^{i-j} : \H^2(W_i) \to \H^2(W_j)$ is surjective and commutes with $F$, the natural map $\H^2(W_i)/ F \H^2(W_i) \to \H^2(W_j)/ F \H^2(W_j)$ is surjective (cf. \cite[Lem.~4.4]{GK}). 

Suppose now $F$ is the zero map on $\H^2(W_i)$. Then it must also be zero on $\H^2(W_j)$ for every $j \le i$. Since $FVR = p$, $p$ annhilates $\H^2(W_j)$, so that $\H^2(W_j)$ is a $k$-vector space. By (\ref{b1}), $\dim_k H^{2}(W_j) = j$. Since $H/FH$ always surjects to $\H^2(W_i) / F\H^2(W_i)$, and $\dim_k H/F H = h - 1$, we deduce that $i \le h - 1$. 

Conversely, we show that if $i \le h - 1$. It suffices to show the statement for $i = h - 1$, in which case the we have $F\H^2(W_{h - 1}) = V^{h - 1} \H^2(W_{h - 1}) = R^{h - 1} V^{h - 1} \H^2(W_{h - 1}) = V^{h - 1} \H^2(W_0) = 0$. 
\end{proof}
As a corollary to the above lemma, we deduce that (cf. \cite[Cor.~5.6]{GK})
\begin{align}
\label{stabdim}
    \dim_k \ker(\H^2(W_i) \stackrel{F}{\to} \H^2(W_i)) = \min(i, h - 1). 
\end{align}
Again, since $FVR = p$, the above kernel is indeed a $k$-vector space. The above lemma already tells us that (\ref{stabdim}) is true if $i \le h - 1$. To see the other half, note that $H \iso \ID(\hat{\Br}_Y^*)$ has a $W$-basis of the form $\{ w, Vw, \cdots, V^{h - 1} w \}$. Let $i$ be any number and $\bar{w}$ be the image of $w$ in $\H^2(W_i)$. $F = 0$ on $\H^2(W_i)$ if and only if $F(\bar{w}) = 0$. For any $j \le i$, we have $F(R^j V^j \bar{w}) = V^j(F(R^j \bar{w})) = 0$ if and only if $F(R^j \bar{w}) = 0$, as $V^j$ is injective by (\ref{b2}). If $i \ge h - 1$, $F(R^j \bar{w}) = 0$ excatly when $j \ge i - (h - 1)$. Hence we have (\ref{stabdim}). 

Let $B_i$ be the subsheaves of $\Ohm_Y$ defined inductively as follows: $B_0 = 0, B_1 = d \sO_Y, B_{i + 1} = C^{-1} B_i$, where $C : \Ohm_{Y, \mathrm{closed}} \to \Ohm_Y$ is the Cartier operator (cf. \cite[\S6]{GK}).

\begin{AppLem}
\emph{(cf. \cite[Thm~6.1]{GK})}
\label{6.1}
$\dim H^1(B_i) = h - 1$ for $i \ge h - 1$. 
\end{AppLem}
\begin{proof}
Serre constructed maps $D_i : W_i \to B_i$ which induces isomorphisms of abelian sheaves $W_i / F W_i \iso B_i$ (cf. \cite[(4)]{GK}). Since $H^1(W_i) = H^3(W_i) = 0$, the exact sequence $0 \to W_i \stackrel{F}{\to} W_i \to W_i / F W_i \to 0$ gives rise to an exact sequence 
\begin{align}
    0 \to H^1(W_i/ F W_i) \to \H^2(W_i) \stackrel{F}{\to} \H^2(W_i) \to \H^2(W_i / FW_i) \to 0.
\end{align}
Now the lemma follows from \ref{5.2} and (\ref{stabdim}). 
\end{proof}

From now on we consider maps $\beta_n : H^1(B_n) \to H^1(\Ohm_Y^1)$ induced by the inclusion $B_n \subset \Ohm_Y^1$. Note that $B_n$ is a subsheaf of $B_m$ for every $m \ge n$. 
\begin{AppLem}
\label{9.2}
\emph{(\cite[Cor.~9.2]{GK})}
For every $n \ge 1$, the map $\beta_n$ is injective. 
\end{AppLem}
\begin{proof}
We first prove the following claim (cf. \cite[Prop.~9.1]{GK}): If for some $n \ge 0$, $\beta_n$ is not injective, then for all $m \ge n$, $\beta_m$ is not injective, and $\dim_k H^1(B_{m + 1}) < \dim H^1(B_{m + 1})$. 

Since \fbox{$h^{1, 0} = 0$}, $H^0(B_i) = 0$ for every $i$. By taking cohomology of the exact sequence $0 \to B_i \to B_{i + 1} \stackrel{C^i}{\to} B_1 \to 0$, we see that $H^1(B_i) \to H^1(B_{i + 1})$ is injective (cf. \cite[Lem.~6.1]{GK}). This shows that when $n \le m$, $\beta_m$ cannot be injective if $\beta_n$ is not injective.

It remains to show that if $\beta_n$ is not injective, then there exists a class in $H^1(B_{n + 1})$ which does not come from $H^1(B_n)$. Let $U_i$ be an affine covering of $Y$ and let $(s_{ij} \in B_n(U_{ij} := U_i \cap U_j))$ be a $1$-cocycle which represents a nontrivial class in $\ker \beta_n$. Since $(s_{ij})$ vanishes as a class in $H^1(\Ohm_Y^1)$, there exists a collection of $\w_i \in \Ohm^1_Y(U_i)$ such that $s_{ij} = \w_i - \w_j$. For every affine open $U \subseteq Y$, the Cartier isomorphism gives a surjective map $H^0(\Ohm_{U, \mathrm{closed}}) \to H^0(\Ohm_U)$ with kernel the exact forms. Therefore, we can find lifts $\wt{s}_{ij} \in B_{n + 1}(U)$, $\wt{\w}_{i} \in \Ohm_{U, \mathrm{closed}}$ of $s_{ij}$ and $\w_i$ such that for some $g_{ij} \in \sO_Y(U_{ij})$, we have 
$$ \wt{s}_{ij} + d g_{ij} = \wt{w}_i - \wt{w}_j. $$
Note that $(\wt{s}_{ij} + d g_{ij})$ defines a $1$-cocycle for the sheaf $B_{n + 1}$, as $d g_{ij} \in B_1(U_{ij}) \subseteq B_{n + 1}(U_{ij})$. We claim that this define a class in $B_{n+1}$ which does not come from $B_n$. Otherwise, we can find a $1$-cocyle $(t_{ij} \in B_n(U_{ij}))$ which represent the same class as $(\wt{s}_{ij} + d g_{ij})$, so that there are relations $t_{ij} = \hat{\w}_i - \hat{\w}_j$ with $\hat{w}_i$ differs from $\wt{\w}_i$ by an element of $B_{n + 1}({U_i})$. In particular, $\hat{\w}_i$ is closed. Since $C^n$ kills $t_{ij}$, $C^n \hat{\w}_i$ defines a global $1$-form on $Y$, which must be zero by the assumption \fbox{$h^{1,0} = 0$}. This implies that $(t_{ij})$ defines a trivial class in $H^1(B_{n + 1})$. However, the image $(t_{ij})$ under the map $C : \H^2(B_{n + 1}) \to \H^2(B_n)$ is the class defined by $(s_{ij})$, which is nonzero, so we get a contradiction and the claim is now proved. 

To conclude the lemma, as \fbox{$h < \infty$}, we note that by \ref{6.1} $\dim H^1(B_n)$ eventually stabilizes. Therefore, it cannot happen that $\beta_n$ is not injective for some $n$. 
\end{proof} 

At this point, assumptions on $Y$ are no longer needed and the proofs of Prop.~10.2 and 10.3 in \cite{GK} work verbatim of $Y$ to give Prop.~\ref{GK}. We refer the reader to \textit{loc. cit.} for more details. 

\end{appendix}

\medskip
 
\printbibliography

\end{document}